\documentclass[11pt,reqno]{amsart}

\usepackage{amsfonts,amssymb,amsmath,gensymb,color,epic,eepic,float,fullpage}
\usepackage{empheq}
\usepackage[mathscr]{euscript}
\usepackage[T1]{fontenc}
\usepackage{mathtools}
\usepackage{comment}
\usepackage{hyperref}
\usepackage{bbm}
\usepackage{tikz}
\usepackage{enumitem}
\usetikzlibrary{graphs,patterns,decorations.markings,arrows,matrix}
\usetikzlibrary{calc,decorations.pathmorphing,decorations.markings,
decorations.pathreplacing,patterns,shapes,arrows}
\tikzstyle{block} = [rectangle,draw,text width=10em,text centered,rounded corners,minimum height=4em]
\tikzstyle{line} = [draw, -latex']
\allowdisplaybreaks

\newtheorem{defn}{Definition}
\newtheorem{lem}{Lemma}
\newtheorem{thm}{Theorem}

\newtheorem{cor}{Corollary}
\newtheorem{rmk}{Remark}
\newtheorem{prop}{Proposition}
\newtheorem{ex}{Example}

\makeatletter
\DeclareFontFamily{OMX}{MnSymbolE}{}
\DeclareSymbolFont{MnLargeSymbols}{OMX}{MnSymbolE}{m}{n}
\SetSymbolFont{MnLargeSymbols}{bold}{OMX}{MnSymbolE}{b}{n}
\DeclareFontShape{OMX}{MnSymbolE}{m}{n}{
    <-6>  MnSymbolE5
   <6-7>  MnSymbolE6
   <7-8>  MnSymbolE7
   <8-9>  MnSymbolE8
   <9-10> MnSymbolE9
  <10-12> MnSymbolE10
  <12->   MnSymbolE12
}{}
\DeclareFontShape{OMX}{MnSymbolE}{b}{n}{
    <-6>  MnSymbolE-Bold5
   <6-7>  MnSymbolE-Bold6
   <7-8>  MnSymbolE-Bold7
   <8-9>  MnSymbolE-Bold8
   <9-10> MnSymbolE-Bold9
  <10-12> MnSymbolE-Bold10
  <12->   MnSymbolE-Bold12
}{}

\let\llangle\@undefined
\let\rrangle\@undefined
\DeclareMathDelimiter{\llangle}{\mathopen}%
                     {MnLargeSymbols}{'164}{MnLargeSymbols}{'164}
\DeclareMathDelimiter{\rrangle}{\mathclose}%
                     {MnLargeSymbols}{'171}{MnLargeSymbols}{'171}
\makeatother

\newmuskip\pFqmuskip

\newcommand*\pFq[6][8]{%
  \begingroup 
  \pFqmuskip=#1mu\relax
  \mathcode`\,=\string"8000
  \begingroup\lccode`\~=`\,
  \lowercase{\endgroup\let~}\pFqcomma
  {}_{#2}\phi_{#3}{\left[\genfrac..{0pt}{}{#4}{#5};{#6}\right]}%
  \endgroup
}
\newcommand{\pFqcomma}{\mskip\pFqmuskip}

\def\b{\bar}

\renewcommand{\leq}{\leqslant}
\renewcommand{\geq}{\geqslant}

\def\a{\alpha}
\def\b{\beta}

\def\k{\kappa}
\def\l{\lambda}
\def\m{\mu}
\def\n{\nu}
\def\s{\sigma}

\def\r{\rho}
\def\e{\epsilon}
\def\t{\tau}
\def\th{\theta}
\def\r{\rho}
\def\LL{\Lambda}
\def\MI{\mathcal{I}}
\def\MJ{\mathcal{J}}

\def\ML{\mathcal{L}}
\def\MM{\mathcal{M}}
\def\MP{\mathcal{P}}
\def\MB{\mathcal{B}}
\def\MQ{\mathcal{Q}}
\def\MZ{\mathcal{Z}}

\def\nt{\tilde{\n}}
\def\ellt{\tilde{\ell}}

\def\st{\tilde{\s}}
\def\nt{\tilde{\n}}
\def\tt{\tilde{\t}}
\def\rt{\tilde{\r}}
\def\MF{\mathcal{F}}
\def\MC{\mathcal{C}}

\def\MTL{\tilde{\mathcal{L}}}
\def\Phit{\tilde{\Phi}}
\def\MN{\mathcal{N}}
\def\PP{\mathcal{P}}
\def\Uq{U_{q}(\widehat{sl_{n+1})}}
\def\cm{{\sf m}}

\def\Lt{\tilde{L}}
\def\Kt{\tilde{K}}
\def\Jt{\tilde{J}}
\def\It{\tilde{I}}

\def\cf{{\it cf.}\/\ }

\def\dg{\color{green!65!black}}

\colorlet{lgray}{white!85!black}
\colorlet{lred}{white!85!red}

\newcommand{\bra}[1]{\left\langle #1\right|}
\newcommand{\ket}[1]{\left|#1\right\rangle}

\DeclareMathOperator{\Tr}{Tr}

\begin{document}

\title{Modified Macdonald polynomials and integrability}

\author{Alexandr Garbali and Michael Wheeler}

\maketitle

\begin{abstract}
We derive combinatorial formulae for the modified Macdonald polynomial $H_{\l}(x;q,t)$ using coloured paths on a square lattice with quasi-cylindrical boundary conditions. The derivation is based on an integrable model associated to the quantum group of $\Uq$.
\end{abstract}

\section{Introduction}

\subsection{Hall--Littlewood polynomials and expansions}

In symmetric function theory \cite{MacdBook}, the Kostka--Foulkes polynomials $K_{\nu,\lambda}(t)$ are the transition coefficients from the Schur polynomials $s_{\nu}(x)$ to the Hall--Littlewood polynomials $P_{\lambda}(x;t)$:
\begin{align}
\label{schur-to-HL}
s_{\nu}(x)
=
\sum_{\lambda} K_{\nu,\lambda}(t) P_{\lambda}(x;t),
\end{align}
where the sum is taken over all partitions $\lambda$. These coefficients are classic objects of representation theory and combinatorics; we refer the reader to \cite{Kirillov2,DLT,Haiman2} and to references therein for a survey of what is known about them. One of their fundamental properties is that they are polynomials in $t$ with nonnegative integer coefficients; that is, $K_{\nu,\lambda}(t) \in \mathbb{N}[t]$. There is an extensive number of combinatorial formulae for computing the Kostka--Foulkes polynomials; the two most famous expressions are due to Lascoux--Sch\"utzenberger \cite{LascouxS} and Kirillov--Reshetikhin \cite{KirillovR2}. 

When $t=1$, the Hall--Littlewood polynomial in (\ref{schur-to-HL}) reduces to a monomial symmetric function $m_{\lambda}(x)$, and we read the expansion of a Schur polynomial over the monomial basis; accordingly $K_{\nu,\lambda}(1)$ recovers the Kostka numbers $K_{\nu,\lambda}$ \cite{Kostka}. 

By duality arguments, another way of introducing the Kostka--Foulkes polynomials is via the {\it modified Hall--Littlewood polynomials} \cite{DLT,LLT,HKKOTY,Kirillov}, in this work denoted $H_{\lambda}(x;t)$. The latter objects have the expansion
\begin{align}
\label{mHL-to-schur}
H_{\lambda}(x;t)
=
\sum_{\nu} K_{\nu,\lambda}(t) s_{\nu}(x).
\end{align}
Because $K_{\nu,\lambda}(t) \in \mathbb{N}[t]$, as an immediate consequence of (\ref{mHL-to-schur}) one sees that the modified Hall--Littlewood polynomials are manifestly positive over the monomial basis:
\begin{align}
\label{mHL-to-monom}
H_{\lambda}(x;t)
=
\sum_{\nu} 
\sum_{\mu}
K_{\nu,\lambda}(t)
K_{\nu,\mu} 
m_{\mu}(x)
\equiv
\sum_{\mu}
\mathcal{P}_{\lambda,\mu}(t)
m_{\mu}(x),
\qquad
\mathcal{P}_{\lambda,\mu}(t)
\in
\mathbb{N}[t].
\end{align}
The problem of explicitly computing the coefficients 
$\mathcal{P}_{\lambda,\mu}(t)$, improving upon the brute-force expansion $\mathcal{P}_{\lambda,\mu}(t) = \sum_{\nu} K_{\nu,\lambda}(t) K_{\nu,\mu}$, was considered in \cite{HKKOTY} (see also \cite{Kirillov}). Letting $x = (x_1,\dots,x_N)$ denote an alphabet of cardinality $N$, it was found that
\begin{align}
\label{eq:HLmon-intro}
\mathcal{P}_{\lambda,\mu}(t)= 
\sum_{\{\nu\}} t^{c(\nu)}
\prod_{k=1}^{N-1} \prod_{i \geq 1}
\binom{\nu_{i}^{k+1}-\nu_{i+1}^{k}}{\nu_{i}^{k}-\nu_{i+1}^{k}}_t,
\qquad
c(\nu)
=
\frac{1}{2}
\sum_{k=1}^{N}
\sum_{i \geq 1}(\nu_{i}^{k}-\nu_{i}^{k-1})(\nu_{i}^{k}-\nu_{i}^{k-1}-1),
\end{align}
with Gaussian binomial coefficients given by
\begin{align}
\label{gauss}
\binom{a}{b}_t
=
\frac{(1-t^a)(1-t^{a-1}) \cdots (1-t^{a-b+1})}{(1-t)(1-t^2) \cdots (1-t^b)}
\cdot
\theta(a \geq b \geq 0),
\end{align}
and where the summation in (\ref{eq:HLmon-intro}) runs over flags of partitions 
$\{\nu\}=\{\emptyset \equiv \nu^{0} \subseteq \nu^{1}\subseteq \dots \subseteq \nu^{N} \equiv \lambda'\}$ such that $|\nu^{k}|=\mu_1+\dots +\mu_k$ for all $1\leq k \leq N$. Here we have used the standard definitions \cite{MacdBook} of partition conjugate $\lambda'$ and weight $|\nu|$.

One of the preliminary results of the present work is an explanation of the formula (\ref{eq:HLmon-intro}) at the level of exactly solvable lattice models; this is given in Section \ref{sec:HL}. It is natural to expect such an interaction with integrability:  connections of the (ordinary) Hall--Littlewood polynomials $P_{\lambda}(x;t)$ with the integrable $t$-boson model have been known for some time \cite{Tsilevich,WheelerZJ,Korff}, and this theory was extended to finite-spin lattice models by Borodin in \cite{Borodin}, yielding a rational one-parameter deformation of the Hall--Littlewood family.

\subsection{Macdonald polynomials and expansions}

Both of the coefficients $K_{\nu,\lambda}(t)$ and $\mathcal{P}_{\lambda,\mu}(t)$ can be naturally generalized to the Macdonald level. The two-parameter Kostka--Foulkes polynomials $K_{\nu,\lambda}(q,t)$ are defined by the expansion
\begin{align}
\label{Hqt-to-schur}
H_{\lambda}(x;q,t)
=
\sum_{\nu}
K_{\nu,\lambda}(q,t)
s_{\nu}(x),
\end{align}
where $H_{\lambda}(x;q,t)$ denotes a {\it modified Macdonald polynomial}; the latter will be defined in Section \ref{sec:W}. In \cite{Macd88,MacdBook}, Macdonald conjectured that $K_{\nu,\lambda}(q,t) \in \mathbb{N}[q,t]$; this problem, the {\it positivity conjecture}, developed great notoriety and was only fully resolved more than ten years later by Haiman \cite{Haiman}. The equation (\ref{Hqt-to-schur}) has an interpretation in terms of the Hilbert scheme of $N$ points \cite{GarsiaH} and indeed the proof of \cite{Haiman} relied heavily on geometric techniques, which did not yield a direct formula for $K_{\nu,\lambda}(q,t)$. It is a major outstanding problem to find a manifestly positive combinatorial expression for these coefficients, although this has been solved in some partial cases \cite{HaglundHL1,Stembridge,Assaf1,Assaf2,Blas}.

Similarly to above, one can convert (\ref{Hqt-to-schur}) to an expansion over the monomial symmetric functions, leading to two-parameter coefficients 
$\mathcal{P}_{\lambda,\mu}(q,t)$:
\begin{align}
\label{Hqt-to-monom}
H_{\lambda}(x;q,t)
=
\sum_{\nu} 
\sum_{\mu}
K_{\nu,\lambda}(q,t)
K_{\nu,\mu} 
m_{\mu}(x)
\equiv
\sum_{\mu}
\mathcal{P}_{\lambda,\mu}(q,t)
m_{\mu}(x),
\qquad
\mathcal{P}_{\lambda,\mu}(q,t)
\in
\mathbb{N}[q,t].
\end{align}
Combinatorially speaking, more is known about $\mathcal{P}_{\lambda,\mu}(q,t)$ than $K_{\nu,\lambda}(q,t)$. In particular, a manifestly positive expression for 
$\mathcal{P}_{\lambda,\mu}(q,t)$, in terms of fillings of tableaux with certain associated statistics, was obtained by Haglund--Haiman--Loehr in \cite{HaglundHL1}.

The current paper will also be concerned with the study of 
$\mathcal{P}_{\lambda,\mu}(q,t)$. We offer a new, positive formula for these coefficients, which is of a very different nature to the formula obtained in \cite{HaglundHL1}; rather, it is in the same spirit as the expansion (\ref{eq:HLmon-intro}), and can be seen to degenerate to the latter at $q=0$. Let us now start to describe it.

\subsection{Basic hypergeometric series and positivity}

Let $\{a_1,\dots,a_N\}$ and $\{b_1,\dots,b_{N-1}\}$ be two sets of complex parameters. Recall the definition of the basic hypergeometric series $_N\phi_{N-1}$ \cite{GasperR}:
\begin{align}
\label{hyp-def}
\pFq[5]{N}{N-1}{a_1\text{,}\dots \text{,}a_{N-1}\text{,}a_{N}}{b_1\text{,}\dots\text{,}b_{N-1}}{z}
=
\sum_{s=0}^{\infty}
z^s
\prod_{j=1}^{N-1}\frac{(a_j;t)_s}{(b_j;t)_s}
\cdot
\frac{(a_{N};t)_s}{(t;t)_s},
\qquad
(a;t)_s
=
\prod_{i=0}^{s-1} (1-a t^i).
\end{align}
In what follows we will tacitly assume that $|z| < 1$, so that the sum in (\ref{hyp-def}) converges. Given two increasing sequences of nonnegative integers $\nu = (\nu^{1} \leq \cdots \leq \nu^N)$ and $\tilde\nu = (\tilde\nu^1 \leq \cdots \leq \tilde\nu^N)$ such that $\nu^N = \tilde\nu^N$,
we define the following evaluation of the above hypergeometric series:
\begin{align}
\label{phi-def}
\phi_{\nu|\tilde\nu}(z;t)
:=
\sum_{s=0}^{\infty}
z^s
\prod_{j=1}^{N-1}\frac{(t^{\alpha_j};t)_s}{(t^{\beta_j};t)_s}
\cdot
\frac{(t^{\alpha_N};t)_s}{(t;t)_s},
\qquad
\alpha_j = \tilde\nu^j-\nu^{j-1}+1,
\qquad
\beta_j = \tilde\nu^j-\nu^j+1,
\end{align}
which is non-singular only under the assumption that $\nu^i \leq \tilde{\nu}^i$ for all $1 \leq i \leq N-1$. By adjusting the normalization of (\ref{phi-def}), we are able to forget about this constraint. In particular, we define
\begin{align}
\label{Phi-def}
\Phi_{\n|\nt}(z;t)
:=
(z;t)_{\nt^N+1} \prod_{k=1}^{N-1}\binom{\nt^{k+1}-\n^k}{\nt^k-\n^k}_t \phi_{\n|\nt}(z;t)
=
(z;t)_{\nt^{N}+1}
\sum_{s=0}^{\infty}
z^s
\prod_{k=0}^{N-1}
\binom{\nt^{k+1}-\n^{k}+s}{\nt^{k}-\n^{k}+s}_t,
\end{align}
where each term in the right hand side of (\ref{Phi-def}) is non-singular for arbitrary increasing sequences $\nu$, $\tilde\nu$.

The function $\Phi_{\n|\nt}(z;t)$ plays a key role in this work. One of our results will be to show that in spite of its definition as an infinite series, 
$\Phi_{\n|\nt}(z;t)$ is in fact a {\it positive polynomial} in $z$. This truncation of the summation is non-trivial and arises because of the inclusion of the Pochhammer function $(z;t)_{\nt^N+1}$ in (\ref{Phi-def}), while allows term-by-term cancellations: 
\begin{thm}[Theorem \ref{th:Fpos} below]
For any sequences of nonnegative integers 
$\nu = (\nu^{1} \leq \cdots \leq \nu^N)$ and $\tilde\nu = (\tilde\nu^1 \leq \cdots \leq \tilde\nu^N)$ such that $\nu^N = \tilde\nu^N$, one has
\begin{align}
\label{Phi-pos}
\Phi_{\n|\nt}(z;t) \in \mathbb{N}[z,t].
\end{align}
\end{thm}
\noindent
Our proof of the statement (\ref{Phi-pos}) will be constructive; it yields an explicit positive formula for $\Phi_{\n|\nt}(z;t)$:
\begin{align}\label{eq:PhiIntro}
\Phi_{\n|\nt}(z;t)=
\sum_{s}
 t^{\eta(s,\nt)}
 \prod_{k=1}^{N-1}  
 z^{s_{N}^k-s_{k}^{k}}  
 \binom{ \nt^{k+1}-s_{k+1}^{1,k}}
 {\nt^{k}-s_k^{1,k}}_t
  \prod_{i\leq k}
\binom{s_{k+1}^{i}}{s_k^{i}}_t,
\end{align} 
where the sum runs over $\{s_k^i\}_{1\leq i\leq k\leq N-1}$, such that $s_{N}^k = \n^k-\n^{k-1}$, $k=1,\dots,N$, the double superscript is defined by $s_{i}^{j,k}=\sum_{a=j}^k s_i^a$, and the $\eta$ exponent is given by $\eta(s,\nt)=\sum_{1\leq i\leq k \leq N-1}(s_{k+1}^i-s_k^i)(\nt^k-s_k^{i,k})$  (see Theorem \ref{th:Fpos}). For fixed $\n$ and $\nt$ this sum consists of a finite number of terms due to the binomial coefficients.

\subsection{Positive expansion of modified Macdonald polynomials}

Now we state the main result of the paper:
\begin{thm}[\bf Theorem \ref{thm:H} below]
Fix two partitions $\lambda = (\lambda_1,\dots,\lambda_N)$ and $\mu = (\mu_1,\dots,\mu_N)$, and let $\lambda'$ denote the conjugate of $\lambda$. We also write $n = \lambda_1$ for the largest part of $\lambda$. Then the coefficient $\mathcal{P}_{\lambda,\mu}(q,t)$, as defined in (\ref{Hqt-to-monom}), is given by
\begin{align}
\label{eq:HmonP-intro}
\mathcal{P}_{\l,\m}(q,t)= 
\sum_{\{\n\}}
t^{\chi(\n)}
\prod_{1 \leq i\leq j \leq n}
\Phi_{\n_{i+1,j}|\n_{i,j}}(q^{j-i} 
t^{ \l'_{i}-\l'_{j}};t),
\end{align}
with the sum running over $n(n+1)/2$ sequences $\n_{i,j} \equiv (\n_{i,j}^{1} \leq \cdots \leq \n_{i,j}^N)$, $1\leq i\leq j\leq n$, satisfying the constraints
\begin{align*}
\n_{i,j}^{N}
=\l'_j-\l'_{j+1},
&
\qquad
\forall\ 1\leq i \leq j \leq n,
\\
\n_{i+1,i}^{k}
=
0,
&
\qquad
\forall\ 1 \leq i \leq n,\ \ 1 \leq k \leq N,
\\
\sum_{1 \leq i \leq j \leq n} \n_{i,j}^{k}
=
\m_1+\dots+\m_k,
&
\qquad
\forall\ 1\leq k \leq N,
\end{align*}
and where the exponent appearing in (\ref{eq:HmonP-intro}) is given by
\begin{align*}
\chi(\n)=
\sum_{k=1}^{N}  
\sum_{1\leq i\leq j\leq n}
\Big{(}
\frac{1}{2}
(\n_{i,j}^k-\n_{i,j}^{k-1})(\n_{i,j}^k-\n_{i,j}^{k-1} -1) + 
\sum_{\ell>j} (\n_{i,j}^k-\n_{i,j}^{k-1})(\n_{i,\ell}^{k}-\n_{i+1,\ell}^{k-1})
\Big{)}.
\end{align*}
\end{thm}
The formula (\ref{eq:HmonP-intro}) bears clear resemblance to the expression (\ref{eq:HLmon-intro}) for the coefficients $\mathcal{P}_{\lambda,\mu}(t)$. In fact, one can show that (\ref{eq:HmonP-intro}) reduces to (\ref{eq:HLmon-intro}) at $q=0$. What underlies this reduction is the fact that, when using the polynomial form of $\Phi_{\nu|\tilde\nu}(z;t)$, both $\Phi_{\nu|\tilde\nu}(0;t)$ and $\Phi_{\nu|\tilde\nu}(1;t)$ can be expressed as a product of binomial coefficients (\ref{gauss}) (see (\ref{eq:Phi01})). This leads to considerable simplification of (\ref{eq:HmonP-intro}) at $q=0$; however, one still needs to use certain summation identities for Gaussian binomial coefficients in order to complete the reduction to (\ref{eq:HLmon-intro}), \cf\ the discussion in Section \ref{ssect:mHLlattice}.

\subsection{Yang--Baxter integrability}

Another key results of this paper is the method that we use to derive (\ref{eq:HmonP-intro}). In particular, we will show how $\mathcal{P}_{\lambda,\mu}(q,t)$ can be computed using an {\it integrable lattice model construction} of the modified Macdonald polynomials $H_{\lambda}(x;q,t)$. 

To our knowledge, the appearance of Yang--Baxter integrable structures \cite{baxterbook} is a novel development in the theory of the modified Macdonald polynomials. On the one hand we find our approach expedient, since it leads directly to the positivity expansion (\ref{Hqt-to-monom}), and on the other hand it is natural, since it allows the symmetry of $H_{\lambda}(x;q,t)$ in $(x_1,\dots,x_N)$ to be seen as a simple corollary of the Yang--Baxter equation. Furthermore, we expect that the methodology adopted here offers new avenues for attacking the more difficult problem of the $(q,t)$ Kostka--Foulkes polynomials $K_{\nu,\lambda}(q,t)$.

Let us explain in more detail how the expression (\ref{eq:HmonP-intro}) arises:
\begin{itemize}[label={$\bullet$}]
\item Our starting point is a known expression for the (ordinary) Macdonald polynomials $P_{\lambda}(x;q,t)$, obtained in \cite{CantinidGW}. This formula, called a {\it matrix product} in \cite{CantinidGW}, expresses $P_{\lambda}(x;q,t)$ as a trace of products of certain infinite dimensional matrices. These matrices arise naturally in the context of integrable $U_{t^{1/2}}(\widehat{sl_{n+1}})$ lattice models; they are nothing but the {\it monodromy matrix} elements in a certain rank-$n$ analogue of the $U_{t^{1/2}}(\widehat{sl_{2}})$ $t$-boson model, and each comes with an associated {\it spectral parameter} $x_i$.

The matrix product formula of \cite{CantinidGW} has an interesting combinatorial interpretation; it can be viewed as the partition function for an ensemble of coloured lattice paths (the colours take values in $\{1,\dots,n\}$) on a cylinder. Any configuration of the paths receives a weight, which is a function of $(x_1,\dots,x_N;t)$, and that weight can in turn be decomposed into a product of local factors, coming from the $\mathcal{L}$-matrix of the model. The parameter $q$ enters when lattice paths {\it wrap around} the cylinder; when $q=0$, no wrapping is allowed, and one recovers known lattice model expressions for the Hall--Littlewood polynomials $P_{\lambda}(x;t)$ \cite{Tsilevich,WheelerZJ,Korff,Borodin}. Sections \ref{ssec:MPACdGW}-\ref{ssec:columnsCdGW} will present a review of these facts.

\item As described in Section \ref{ssec:ReductionMacH} the modified polynomials $H_{\lambda}(x;q,t)$ are obtained from the ordinary ones $P_{\lambda}(x;q,t)$ via two operations. The first is to multiply $P_{\lambda}(x;q,t)$ by a certain normalizing factor $c_{\lambda}(q,t)$, converting them to their {\it integral form} $J_{\lambda}(x;q,t)$ \cite{MacdBook}, when they become polynomials in $(q,t)$. The second is the plethystic substitution $p_{k}(x) \mapsto p_{k}(x) / (1-t^k)$, where $p_k(x) = \sum_{i=1}^{N} x_i^k$ denotes the $k$-th power-sum.

One then needs to investigate what these operations do to the matrix product expression of \cite{CantinidGW}. Both turn out to be non-trivial. In particular, we will show that the above plethystic substitution has a known meaning in the integrable setting; it induces {\it fusion} (in the sense of Kulish--Reshetikhin--Sklyanin \cite{KulishRS}) of the lattice model originally used in \cite{CantinidGW}. The resulting fused model that one obtains has at least two interesting features: its $\mathcal{L}$-matrix is doubly bosonic (meaning that both its horizontal and vertical lattice edges can be occupied by an arbitrary number of paths), and the Boltzmann weights of the model are manifestly elements of $\mathbb{N}[x,t]$. It can be recovered as a certain special case of an $\mathcal{L}$-matrix found by Bosnjak--Mangazeev in \cite{BosM}, and was previously written down by Korff \cite{Korff} in the $U_t(\widehat{sl_2})$ case, but as far as we know this model (and the observation that its weights are positive) has not explicitly appeared in the literature.

\item In the final stage, one needs to understand the combinatorics behind the fused matrix product described in the previous step. This is explained in Section \ref{sec:modM} and leads directly to the formula (\ref{eq:HmonP-intro}), with the integer sequences $\nu_{i,j}$ encoding lattice configurations of coloured paths on a cylinder. At this point, the function $\Phi_{\nu|\tilde\nu}(z;t)$ remains defined as an infinite series, \cf equation (\ref{Phi-def}). It is a separate, purely technical endeavour to then demonstrate that $\Phi_{\nu|\tilde\nu}(z;t)$ is a positive polynomial (\ref{eq:PhiIntro}); this we do in Section \ref{sec:phi}.
\end{itemize}

\subsection{Polynomials with two alphabets}

Throughout much of the paper we will work with a natural {\it double alphabet} extension of the modified Macdonald polynomials, which we denote 
$W_{\lambda}(x;q,t;z)$. These polynomials appeared in \cite{HaglundHL1}. They are defined via the following plethystic substitution:
\begin{align}
\label{W-def}
W_{\lambda}(x;q,t;z)
:=
J_{\lambda}(x;q,t)
\Big| p_k(x) \mapsto \frac{p_k(x)-p_k(-z)}{1-t^k},
\end{align}
where $J_{\lambda}(x;q,t)$ is the integral form of a Macdonald polynomial \cite{MacdBook}, and $p_k(x) = \sum_{i=1}^{N} x_i^k$, $p_k(-z) = \sum_{i=1}^{N} (-z_i)^k$ are power sums in the alphabets $(x_1,\dots,x_N)$ and $(-z_1,\dots,-z_N)$. 

These polynomials cast the Macdonald theory in a very symmetric light. By their definition, they are clearly symmetric in $(x_1,\dots,x_N)$ and $(z_1,\dots,z_N)$ separately. It also turns out that
\begin{align*}
W_{\lambda}(x;q,t;z)
\in \mathbb{N}[x_1,\dots,x_N,q,t,z_1,\dots,z_N];
\end{align*} 
or in other words, they are positive polynomials in all parameters, which does not appear to be obvious from the definition (\ref{W-def}) and the fact that $H_{\lambda}(x;q,t)$ is positive. In addition to this, $W_{\lambda}(x;q,t;z)$ has at least four interesting reductions:
\begin{center}
\begin{tikzpicture}[node distance = 3.3cm,auto]
\node (W) {$W_{\lambda}(x_1,\dots,x_N;q,t;z_1,\dots,z_N)$};
\node [below right of=W] (Hp) 
{$H_{\lambda'}(z_1,\dots,z_N;t;q)$};
\node [below left of=W] (H) 
{$H_{\lambda}(x_1,\dots,x_N;q,t)$};
\node[left of=H] (P)
{$J_{\lambda}(x_1,\dots,x_N;q,t)$};
\node[right of=Hp] (Pp)
{$J_{\lambda'}(z_1,\dots,z_N;t,q)$};
\path [line] (W) -- node[left] {$z_i=-tx_i$\ \ \ \ }  (P);
\path [line] (W) -- node[left] {$z_i=0$} (H);
\path[line] (W) -- node[right] {$x_i=0$} (Hp); 
\path[line] (W) -- node[right] {\ \ $x_i=-qz_i$} (Pp);
\end{tikzpicture}
\end{center}
with the indicated substitutions taken over all $1 \leq i \leq N$, and where 
$\lambda'$ denotes the dual of the partition $\lambda$. We outline these facts in Section \ref{sec:W}, together with the Cauchy summation identities of the polynomials $W_{\lambda}(x;q,t;z)$.

A further result that we present is the meaning of the parameters $(z_1,\dots,z_N)$ in the integrable setting. They have the following interpretation: {\bf 1.} In performing fusion of the lattice model used in \cite{CantinidGW}, one needs to replace the auxiliary space of the $\mathcal{L}$-matrix -- which starts off in the defining, or {\it fundamental} representation of $U_{t^{1/2}}(\widehat{sl_{n+1}})$ -- with a symmetric tensor representation of weight $\mathcal{J}$; {\bf 2.} The resulting fused $\mathcal{L}$-matrix depends analytically on $t^{\mathcal{J}}$, but is otherwise independent of the value of the integer $\mathcal{J}$. One can then analytically continue $\mathcal{J}$ to all complex values, performing the substitution $t^{\mathcal{J}} \mapsto -z/x$, where $x$ is the spectral parameter of the $\mathcal{L}$-matrix and $z$ is a new parameter that we call a {\it spin}. Doing this on a row-by-row basis, with a different $(x_i,z_i)$ for each row, gives rise to the full alphabet 
$(z_1,\dots,z_N)$. Such ideas of fusion combined with analytic continuation have already surfaced in a number of works on lattice models and symmetric functions; see for example \cite{CorwinP,Borodin,BorodinP,BorodinW}.

As a result of these considerations, we obtain a matrix product construction of the polynomials $W_{\lambda}(x;q,t;z)$; this is given in Section \ref{sec:modM}. This would allow one, in principle, to write down monomial expansions of the form (\ref{Hqt-to-monom}) with all parameters $(x_1,\dots,x_N)$ and $(z_1,\dots,z_N)$ kept generic. However we will not do this in the present work, as the resulting formulae become rather long. Instead we will focus on the $z_i = 0$ reduction, which yields equation (\ref{eq:HmonP-intro}), as well as the $x_i = 0$ reduction, which leads to a {\it dual} analogue of (\ref{eq:HmonP-intro}) (see Theorem \ref{thm:H}).

\subsection{Outline}
In Section \ref{sec:W} we introduce the polynomials $W_\l(x;q,t;z)$, examine their symmetries, reductions and summation identities. In Section \ref{sec:HL} we present a lattice construction for  the modified Hall--Littlewood polynomials $H_\l(x;t)$ and show that it recovers the combinatorial formula of \cite{HKKOTY}. Next, in Section \ref{sec:modM}, we turn to the lattice constructions of $W_\l(x;q,t;z)$ and the modified Macdonald polynomials $H_\l(x;q,t)$ and derive two combinatorial formulae for the latter. In Section \ref{sec:proof} we prove that our lattice partition functions indeed match with $H_\l(x;q,t)$. In doing so we present full details of the computation of the fused $\ML$-matrix used in the construction of $W_\l(x;q,t;z)$. Finally, in Section \ref{sec:phi} we present the details of the resummation of the normalized hypergeometric function $\Phi_{\n|\nt}(z;t)$ into a finite polynomial with positive coefficients.


\section{The polynomials $W_{\l}(x;q,t;z)$}\label{sec:W}
In this section we define the $W$-polynomials $W_{\l}(x;q,t;z)$ which generalize Macdonald polynomials $P_{\l}(x;q,t)$, where $x=(x_1,\dots,x_N)$ and $z=(z_1,\dots,z_N)$ are two alphabets and $q$ and $t$ are the standard parameters of  Macdonald theory. 
We show that $W$-polynomials reduce to Macdonald polynomials and modified Macdonald polynomials, and also to their dual counterparts. We also discuss the symmetries of $W_{\l}(x;q,t;z)$ and their Cauchy identities. 
For basic facts of Macdonald theory we refer to the book \cite{MacdBook}.

\subsection{Notation}
Define the ring of symmetric functions $\LL_N:=\mathbb{Z}[x_1,\dots,x_N]^{\mathfrak{S}_N}$, where $\mathfrak{S}_N$ is the permutation group of $N$ letters, the field $\mathbb{F}=\mathbb{Q}(q,t)$ and the ring  $\LL_{N,\mathbb{F}}:=\LL_N\otimes\mathbb{F}$. 
Let $\PP$ be the set of all partitions. Let the length $\ell(\l)$ of the partition $\l\in \PP$ be the number of non-zero parts in $\l$. The weight of the partition $\l$ is denoted by $|\l|$ and equals the sum of its parts. If $\l$ is a permutation of a partition (an integer composition) then $|\l|$ also denotes the sum of parts of $\l$. 
To each partition we assign a Young diagram in the standard way. The partition $\l'$ is called dual to $\l$ and its parts are given by $\l'_i=\text{Card}(j:\l_j\geq i)$. The Young diagram of the dual partition is the reflection along the main diagonal of the Young diagram of the partition $\l$. For $\l,\m \in \PP$ let $\l\subseteq\m$ be the partial order on $\PP$ defined by the inclusion of the corresponding Young diagram $\l$ in $\m$. Let $\l\preceq\m$ be the partial order with the same inclusion condition as before, but with the additional constraint that the skew diagram $\l-\m$ contains no more than one square in each column, i.e. it is a horizontal strip. Assuming $|q|<1$ and $|t|<1$ we define the $(q,t)$-shifted factorials
\begin{align}
&(w;q)_{\infty}:=\prod_{i=0}^{\infty} (1-w q^i), \qquad (w;q)_{k}:=\prod_{i=0}^{k-1} (1-w q^i),  
\qquad 
(w;q)_{\l}=\prod_{i=1}^N(w;q)_{\l_{i}},\label{weq:qpoc}\\
&(w;q,t)_{\infty} := \prod_{k=0}^{\infty} \prod_{\ell=0}^{\infty} (1-w q^k t^\ell), \qquad (w;q,t)_{\l}:=\prod_{i=1}^{N}(w t^{1-i};q)_{\l_i}\label{weq:qtpoc},
\end{align}
where $\l=(\l_1,\dots,\l_N)$ is a partition. Let $\bullet=(i,j)\in{\mathbb{Z}^2}$ be a coordinate of the Young diagram of $\n=(\n_1,\dots,\n_N)$; $i\in\{1,\dots,N\}$ and $j\in\{1,\dots,\n_i\}$. Define its arm length $a(\bullet)=\n_i-j$ and leg length $l(\bullet)=\n'_j-i$. 
Define the coefficients $b_\n(q,t)$, $c_\n(q,t)$ and $c'_\n(q,t)$:
\begin{align}
\label{eq:bcc}
b_{\nu}(q,t)
=
\frac{c_{\nu}(q,t)}{c'_{\nu}(q,t)},
\quad
c_{\nu}(q,t)
=
\prod_{\bullet \in \nu}
(1-q^{a(\bullet)} t^{l(\bullet)+1}),
\quad
c'_{\nu}(q,t)
=
\prod_{\bullet \in \nu}
(1-q^{a(\bullet)+1} t^{l(\bullet)}).
\end{align}
These coefficients satisfy 
\begin{align}
&c'_{\n}(t,q)=c_{\n'}(q,t), 		\label{eq:ccprime}	\\
&c_{\n}(q,t)=(-t)^{|\n|}t^{n(\n)}q^{n(\n')}c_{\n}(1/q,1/t), 		\label{eq:ccinv}	\\
&b_\l(q,t) ={b_{\l'}(t,q)}^{-1}	\label{eq:bbprime},
\end{align}
where $n(\n)=\sum_{i=1}^N \n_i (i-1)$. Define also the indicator function $\theta(\text{True})=1$, $\theta(\text{False})=0$. 
\subsection{Branching formula}
The Macdonald polynomials $P_{\lambda}(x;q,t)$ form a basis of the ring $\LL_{N,\mathbb{F}}$ \cite[ch IV, (4.7)]{MacdBook}. With $P_\l(\varnothing;q,t)=\theta(\ell(\l)=0)$, Macdonald polynomials can be constructed recursively via the branching formula
\begin{align}\label{eq:branchP}
P_{\l}(x_1,\dots,x_N;q,t)= \sum_{\m\subseteq \l} P_{\l/\m}(x_N;q,t) P_{\m}(x_1,\dots,x_{N-1};q,t),
\end{align}
where the branching coefficients are given by the skew Macdonald polynomial in one variable:
\begin{align}
 &P_{\l/\m}(x;q,t) = \theta(\m\preceq \l) x^{|\l-\m|} \psi_{\l/\m}(q,t), \label{eq:branchcoefP}\\
 &\psi_{\l/\m}(q,t)=\prod_{1\leq i\leq j\leq \ell(\m)} \frac{f(q^{\m_i-\m_j}t^{j-i})f(q^{\l_i-\l_{j+1}}t^{j-i})}{f(q^{\l_i-\m_j}t^{j-i})f(q^{\m_i-\l_{j+1}}t^{j-i})},	\label{eq:psicoef}
\end{align} 
where $f(a)=(a t;q)_{\infty}/(a q;q)_{\infty}$, \cite[ch VI, sec 6, Ex. 2(b)]{MacdBook}. Note that the sum in (\ref{eq:branchP}) can be restricted to $\m\preceq \l$ due to the presence of the indicator function in (\ref{eq:branchcoefP}). 

Branching rules, under certain constraints (see e.g. \cite{LW}), allow one to define symmetric polynomials in a recursive manner. Using this idea one can generalize Macdonald polynomials \cite{LW}; we pursue one such generalization here. 

First, let us define an evaluation homomorphism. Let $p_r$ be the $r$-th power sum symmetric function 
\begin{align*}
p_r(x)=p_r(x_1,\dots,x_N)=\sum_{i=1}^N x_i^r.
\end{align*}
The functions $p_\l=\prod_{i=1}^N p_{\l_i}$ form a basis of $\LL_N$. 
Define the evaluation homomorphism $\pi^{x}_{y,z}$ (cf. \cite[p. 338]{MacdBook})
\begin{align}\label{eq:evhom}
\pi^{x}_{y,z}: p_r(x) \mapsto \frac{y^r - (-1)^r z^r}{1-t^r}.
\end{align}
Applying this evaluation homomorphism to the Macdonald polynomials gives \cite[(6.17)]{MacdBook}
\begin{align*}
\pi^{x}_{y,z}(P_\l(x;q,t))= t^{n(\l)}y^{|\l|}\frac{(-z/y;q,t)_\l}{c_\l(q,t)}.
\end{align*}
The skew Macdonald polynomials and their non-skew counterparts are related via \cite[p. 344]{MacdBook}
\begin{align*}
P_{\l/\m}(x;q,t)= \sum_{\n} \frac{b_\m(q,t) b_\n(q,t)}{b_\l(q,t)} f^{\l}_{\m,\n}(q,t) P_{\n}(x;q,t),
\end{align*}
where the sum runs over $\n$ such that $|\l|=|\m|+|\n|$ and $f^{\l}_{\m,\n}(q,t)$ is defined as the expansion coefficient in the product of two Macdonald polynomials:
\begin{align*}
P_{\m}(x;q,t)
P_{\n}(x;q,t)= \sum_{\l} f^{\l}_{\m,\n}(q,t) P_{\l}(x;q,t).
\end{align*}
Therefore the evaluation homomorphism on $P_{\l/\m}$ is given by 
\begin{align*}
\pi^{x}_{y,z}(P_{\l/\m}(x;q,t))=
y^{|\l-\m|} \frac{b_\m(q,t)}{b_\l(q,t)}\sum_{\n} t^{n(\n)}\frac{(-z/y;q,t)_\n}{c'_\n(q,t)} f^{\l}_{\m,\n}(q,t).
\end{align*}
We use this as our definition of the skew $W$-polynomial in one variable $y$ and one variable $z$:
\begin{align}\label{eq:skewW}
W_{\l/\m}(y;q,t;z)=
y^{|\l-\m|} \frac{c'_\l(q,t)}{c'_\m(q,t)}\sum_{\n} t^{n(\n)}\frac{(-z/y;q,t)_\n}{c'_\n(q,t)} f^{\l}_{\m,\n}(q,t).
\end{align}
Note that $W_{\lambda/\m}$ is generically non-vanishing for any $\m \subset \lambda$, not just in the case when $\lambda/\m$ forms a horizontal strip. For $\m = \varnothing$, the above sum over $\nu$ trivializes (since $f^{\lambda}_{\varnothing, \nu} = \delta_{\lambda \nu}$), and we obtain an explicit formula for one variable $W$-polynomials:
\begin{align*}
W_{\lambda}(y;q,t;z)
=
y^{|\l|} t^{n(\l)}(-z/y;q,t)_\l.
\end{align*}

Now we build up two alphabets $(x_1,\dots,x_N)$ and $(z_1,\dots,z_N)$ by the following branching rule:
\begin{align}\label{eq:branchW}
W_{\lambda}(x_1,\dots,x_N;q,t;z_1,\dots,z_N)
=
\sum_{\m}
W_{\lambda/\m}(x_N;q,t;z_N)
W_{\m}(x_1,\dots,x_{N-1};q,t;z_1,\dots,z_{N-1}),
\end{align}
with the initial condition $W_\l(\varnothing;q,t;\varnothing)=1$.


\subsection{Symmetries}
Our construction of $W_{\l}(x;q,t;z)$ using the branching rule does not guarantee that the resulting polynomials are symmetric in the alphabets $(x_1,\dots,x_N)$ and $(z_1,\dots,z_N)$.
\begin{prop}
The polynomials $W_{\l}(x;q,t;z)$ are symmetric in the alphabet $x=(x_1,\dots,x_N)$ and separately in the alphabet $z=(z_1,\dots,z_N)$. 
\end{prop}
\begin{proof}
Our approach is based on the evaluation homomorphism (plethystic substitution). Take a set $X=(x^{(1)},\dots,x^{(N)})$ to be a collection of alphabets $x^{(i)}=(x_1^{(i)},\dots,x_{m_i}^{(i)})$ such that each $m_i$ is arbitrary. By the definition of the power sums we have:
\begin{align}\label{eq:psumadd}
p_r(X)=p_r(x^{(1)})+\dots+p_r(x^{(N)}).
\end{align}
Let $Y=(y_1,\dots,y_N)$ and $Z=(z_1,\dots,z_N)$ be two more alphabets. 
The evaluation homomorphism on a collection of alphabets is defined by the composition of evaluation homomorphisms on individual alphabets
\begin{align*}
&\pi^{X}_{Y,Z}:=
\pi^{x^{(N)}}_{y_N,z_N}\dots\pi^{x^{(1)}}_{y_1,z_1}.
\end{align*}
We find that
\begin{align*}
\pi^{X}_{Y,Z}: 
p_r(x^{(1)}+\dots+x^{(N)}) \mapsto \frac{y_1^r+\dots+y_N^r - (-1)^r\left(z_1^r+\dots +z_N^r\right)}{1-t^r},
\end{align*}
or more compactly,
\begin{align}\label{eq:ev}
\pi^{X}_{Y,Z}: 
p_r(X) \mapsto \frac{p_r(Y)-(-1)^rp_r(Z)}{1-t^r}.
\end{align}
Next, take (\ref{eq:skewW}) rewritten as 
\begin{align}\label{eq:piskew}
W_{\l/\m}(y;q,t;z)=
c_\l(q,t)\pi^{x}_{y,z}(P_{\l/\m}(x;q,t)),
\end{align}
and apply $\pi^{X}_{Y,Z}$ to (\ref{eq:branchP}):
\begin{align}
\pi^{X}_{Y,Z}\left(P_{\l}(x^{(1)}+\dots+x^{(N)};q,t)\right)&= 
\sum_{\m\subseteq \l} \pi^{x^{(N)}}_{y_N,z_N}\left(P_{\l/\m}(x^{(N)};q,t)\right)  \\
&\times
\pi^{x^{(N-1)}}_{y_{N-1},z_{N-1}}\dots\pi^{x^{(1)}}_{y_{1},z_{1}}
\left(P_{\m}(x^{(1)}+\dots+x^{(N-1)};q,t)\right). \nonumber
\end{align}
Since the initial conditions of the branching recurrences for $P_\l$ and $W_\l$ coincide, we conclude that 
\begin{align}\label{eq:evW}
W_{\l}(y_1,\dots,y_N;q,t;z_1,\dots,z_N)=
c_\l(q,t)\pi^{X}_{Y,Z}\left(P_{\l}(x^{(1)}+\dots+x^{(N)};q,t)\right).
\end{align}
The Macdonald polynomial on the right hand side can be expanded in $p_r(X)$ on which the evaluation homomorphism $\pi^{X}_{Y,Z}$ acts. Due to (\ref{eq:ev}) this means that $W_\l$ are symmetric polynomials in both alphabets $Y$ and $Z$ separately. Alternatively, one can think of (\ref{eq:evW}) as the definition of $W_\l$.
\end{proof}
This proposition can also be proven by applying two times the branching rule (\ref{eq:branchW}) to the polynomials $W_{\l}(x_1,\dots,x_{N-1},x_N;q,t;z_1,\dots,z_{N-1},z_N)$ and $W_{\l}(x_1,\dots,x_N,x_{N-1};q,t;z_1,\dots,z_{N-1},z_N)$ and demanding that the results are equal in both cases. This leads to an identity on the branching coefficients $W_{\l/\m}$ similar to \cite[Equation  (3.3)]{LW}. 

The $W$-polynomials also have interesting symmetries involving the parameters $q$ and $t$. 
First we examine the symmetries of the skew polynomial (\ref{eq:skewW}) in one variable. 
Under the inversion of $q$ and $t$ we have the following relations:
\begin{align*}
&f^{\l}_{\m,\n}(q^{-1},t^{-1})=f^{\l}_{\m,\n}(q,t),
\qquad c'_{\l}(q^{-1},t^{-1})=(-q)^{-|\l|}q^{-n(\l')}t^{-n(\l)} c'_{\l}(q,t), 
\end{align*}
\begin{align*}
& \frac{t^{-n(\n)}(-z/y;q^{-1},t^{-1})_{\n}}{c'_{\n}(q^{-1},t^{-1})}=
(-z/y)^{|\n|}q^{|\n|} \frac{t^{n(\n)}(-y/z;q,t)_{\n}}{c'_{\n}(q,t)},
\end{align*}
where the relation involving $f^{\l}_{\m,\n}$ can be found in \cite[p. 343]{MacdBook}. These relations lead to 
\begin{align*}
W_{\l/\m}(y;q^{-1},t^{-1};z)
&=
y^{|\l-\m|} \frac{c'_\l(q^{-1},t^{-1})}{c'_\m(q^{-1},t^{-1})}\sum_{\n} t^{-n(\n)}\frac{(-z/y;q^{-1},t^{-1})_\n}{c'_\n(q^{-1},t^{-1})} f^{\l}_{\m,\n}(q^{-1},t^{-1})\\
&=
z^{|\l-\m|} \frac{q^{-n(\l')}t^{-n(\l)} c'_\l(q,t)}{q^{-n(\m')}t^{-n(\m)} c'_\m(q,t)}
\sum_{\n}  \frac{t^{n(\n)}(-y/z;q,t)_{\n}}{c'_{\n}(q,t)} f^{\l}_{\m,\n}(q,t).
\end{align*}
Therefore we find the following symmetry:
\begin{align}\label{eq:syminv}
W_{\l/\m}(y;q^{-1},t^{-1};z)=
t^{n(\m)-n(\l)}q^{n(\m')-n(\l')}W_{\l/\m}(z;q,t;y).
\end{align}
In a similar vein, under the exchange of $q$ and $t$ we have \cite[(7.3)]{MacdBook}
\begin{align*}
&f^{\l}_{\m,\n}(t,q)=f^{\l'}_{\m',\n'}(q,t) \frac{b_{\l}(t,q)}{b_{\m}(t,q)b_{\n}(t,q)}, 
\end{align*}
\begin{align*}
&q^{n(\n)}(-z/y;t,q)_\n=(z/y)^{|\n|}t^{n(\n')}(-y/z;q,t)_{\n'}.
\end{align*}
Together with (\ref{eq:ccprime}) this leads to 
\begin{align*}
W_{\l/\m}(y;t,q;z)=&
y^{|\l-\m|} \frac{c'_\l(t,q)}{c'_\m(t,q)}\sum_{\n} q^{n(\n)}\frac{(-z/y;t,q)_\n}{c'_\n(t,q)} f^{\l}_{\m,\n}(t,q)= \\
&z^{|\l'-\m'|} \frac{c'_{\l'}(q,t)}{c'_{\m'}(q,t)}\sum_{\n}t^{n(\n')}\frac{(-y/z;q,t)_{\n'}}{c'_{\n'}(q,t)} f^{\l'}_{\m',\n'}(q,t).
\end{align*}
Therefore we find another symmetry:
\begin{align}\label{eq:symqt}
W_{\l/\m}(y;t,q;z)=W_{\l'/\m'}(z;q,t;y).
\end{align}
\begin{prop}
The polynomials $W_{\l}(x;q,t;z)=W_{\l}(x_1,\dots,x_N;q,t;z_1,\dots,z_N)$ satisfy
\begin{itemize}
\item Quasi-symmetry under the inversion of $q$ and $t$:
\begin{align}\label{eq:Wsyminv}
W_{\l}(x;q^{-1},t^{-1};z)=
t^{-n(\l)}q^{-n(\l')}W_{\l}(z;q,t;x).
\end{align}
\item Quasi-symmetry under the exchange of $q$ and $t$:
\begin{align}\label{eq:Wsymexch}
W_{\l}(x;t,q;z)=
W_{\l'}(z;q,t;x).
\end{align}
\end{itemize}
\end{prop}
\begin{proof}
Property (\ref{eq:Wsyminv}) follows from repetitive use of the branching rule together with (\ref{eq:syminv}), while property (\ref{eq:Wsymexch}) follows similarly but using (\ref{eq:symqt}).
\end{proof}
%


\subsection{Reductions to Macdonald polynomials and their duals}
\label{ssec:ReductionMac}
Based on (\ref{eq:branchW}) we find reductions to ordinary Macdonald polynomials at a specialization of $z$ variables, and at a specialization of $x$ variables. 
\begin{prop}
The polynomials $W_{\l}(x;q,t;z)=W_{\l}(x_1,\dots,x_N;q,t;z_1,\dots,z_N)$ satisfy
\begin{itemize}
\item The following reduction at $z_i=-t x_i$, for all $1\leq i\leq N$
\begin{align}\label{eq:Wred1}
W_{\l}(x_1,\dots,x_N;q,t;-t x_1,\dots,-t x_N)=
c_\l(q,t)
P_{\l}(x_1,\dots,x_N;q,t).
\end{align}
\item The following reduction at $x_i=-q z_i$, for all $1\leq i\leq N$
\begin{align}\label{eq:Wred2}
W_{\l}(-q z_1,\dots,-q z_N;q,t;z_1,\dots, z_N)=
c_{\l'}(t,q)P_{\l'}(z_1,\dots,z_N;t,q).
\end{align}
\end{itemize}
\end{prop}
In fact, due to the symmetry in $x$ and $z$ variables we can relax the two reductions to $z_i=-t x_{\s(i)}$ and $x_i=-q z_{\s(i)}$, where $\s$ is any permutation in $\mathfrak{S}_N$. Let us also remark that the reduction (\ref{eq:Wred1}) readily follows from a specialization in (\ref{eq:ev}). We will, however, use the explicit branching rules to prove both equations (\ref{eq:Wred1}) and (\ref{eq:Wred2}).

\begin{proof}
Set $z=- t y$ in (\ref{eq:skewW}). The function $(t y/y;q,t)_\n=(t;q,t)_\n$ vanishes unless $\nu$ consists of a single row, $\nu=(\nu_1)$, and $\nu_1=|\l-\m|$. This selects one term in the summation in  (\ref{eq:skewW}), and we have
\begin{align*}
W_{\l/\m}(y;q,t;-t y)=
y^{|\l-\m|} \frac{c'_\l(q,t)}{c'_\m(q,t)}\frac{(t;q,t)_{|\l-\m|}}{c'_{|\l-\m|}(q,t)} f^{\l}_{\m,|\l-\m|}(q,t).
\end{align*}
This simplifies due to
\begin{align*}
c'_{|\l-\m|}(q,t)=(q;q,t)_{|\l-\m|}, \qquad f^{\l}_{\m,|\l-\m|}(q,t)=\theta(\m\preceq \l)\frac{(q;q,t)_{|\l-\m|}}{(t;q,t)_{|\l-\m|}}\frac{b_\l(q,t)}{b_\m(q,t)}\psi_{\l/\m}(q,t),
\end{align*}
where the formula for the coefficients $f^{\l}_{\m,|\l-\m|}(q,t)$ can be found in \cite[(6.24)]{MacdBook}.
Thus the branching coefficient becomes 
\begin{align*}
W_{\l/\m}(y;q,t;-t y)=
\theta(\m\preceq \l)
y^{|\l-\m|}\frac{c_\l(q,t)}{c_\m(q,t)}
\psi_{\l/\m}(q,t).
\end{align*}
If we substitute this branching coefficient into (\ref{eq:branchW}) we recover the branching rule for Macdonald polynomials (\ref{eq:branchP}) times $c_\l(q,t)$, which proves the statement (\ref{eq:Wred1}). 

The second statement (\ref{eq:Wred2}) is proved analogously. Set $y=-q z$ in (\ref{eq:skewW}). The function $(z/qz;q,t)_\n=(q^{-1};q,t)_\n$ vanishes unless $\nu$ consists of a single column, $\nu=1^p$, where $p=|\l-\m|$. This selects one term in the summation in  (\ref{eq:skewW}):
\begin{align*}
W_{\l/\m}(-q z;q,t;z)=
(-q z)^{|\l-\m|} \frac{c'_\l(q,t)}{c'_\m(q,t)}t^{n(1^p)}\frac{(q^{-1};q,t)_{1^p}}{c'_{1^p}(q,t)} f^{\l}_{\m,1^p}(q,t).
\end{align*}
Due to simplifications
\begin{align*}
& c'_{1^p}(q,t)=\prod_{i=1}^{p}(1-q t^{i-1}), 
 \qquad t^{n(1^p)}(q^{-1};q,t)_{1^p}=(-q)^{-p} \prod_{i=1}^{p}(1-q t^{i-1}),  \\
& f^{\l}_{\m,1^p}(q,t)=\theta(\m'\preceq \l')\psi_{\l'/\m'}(t,q),
\end{align*}
we can write the branching coefficient as
\begin{align*}
W_{\l/\m}(-q z;q,t;z)=
\theta(\m'\preceq \l')
z^{|\l-\m|} \frac{c'_\l(q,t)}{c'_\m(q,t)} \psi_{\l'/\m'}(t,q).
\end{align*}
Now since
\begin{align*}
P_{\l'/\m'}(z;t,q) = \theta(\m'\preceq \l') z^{|\l-\m|} \psi_{\l'/\m'}(t,q),
\end{align*}
and by virtue of (\ref{eq:ccprime}),  we find that
\begin{align*}
W_{\l/\m}(-q z;q,t;z)=
\frac{c_{\l'}(t,q)}{c_{\m'}(t,q)} P_{\l'/\m'}(z;t,q).
\end{align*}
Substituting this into the branching rule (\ref{eq:branchW}), we recover the branching rule of the form (\ref{eq:branchP}) which is expressed in terms of the polynomial $c_{\l'}(t,q)P_{\l'}(z_1,\dots,z_N;t,q)$. This proves (\ref{eq:Wred2}).
\end{proof}
In \cite[ch VI, Ex (5,1)]{MacdBook} we find the involution $\omega_{q,t}$ which relates $P_{\l}$ to its dual $Q_{\l}(x;q,t)=b_\l(q,t) P_{\l}(x;q,t)$:
\begin{align*}
&\omega_{q,t}P_{\l}(x;q,t)=
Q_{\l'}(x;t,q),\\
&\omega_{q,t}Q_{\l}(x;q,t)=
P_{\l'}(x;t,q).
\end{align*}

\begin{rmk}
Combining (\ref{eq:Wred2}) with the homomorphism (\ref{eq:ev}), we recover the Macdonald involution:
\begin{align*}
\pi^{x}_{q x,x}\left(Q_{\l}(x;q,t)\right)=
P_{\l'}(x;t,q),
\end{align*}
while $\pi^{x}_{x,t x}$ acts as the identity map.
\end{rmk}
%


\subsection{Modified Macdonald polynomials and their duals}
\label{ssec:ReductionMacH}
Let us recall the definition of the modified Macdonald polynomials \cite{Macd88}. First we need the Macdonald polynomials in the integral form
\begin{align}\label{eq:MacJ}
J_{\l}(x;q,t):=
c_{\l}(q,t) P_{\l}(x;q,t).
\end{align}
The function $J_{\l}(x;q,t)$ considered as a polynomial in $x_1,\dots,x_N$ has coefficients in $\mathbb{Z}[q,t]$. We define the modified Macdonald polynomials $H_\l$ by\footnote{Our definition differs from the definition of \cite{Haiman2} by the inversion of $t$ and the factor $t^{n(\l)}$.} 
\begin{align}\label{eq:MacH}
H_{\l}(y;q,t):=
\pi_{y,0}^{x}(J_{\l}(x;q,t)).
\end{align}
The $W$-polynomials  reduce to modified Macdonald polynomials. Set for convenience 
\begin{align}\label{eq:Wx}
W_{\l}(x_1,\dots,x_N;q,t)
&
:=
W_{\l}(x_1,\dots,x_N;q,t;0,\dots,0), \\
\label{eq:Wz}
W_{\l}(q,t;z_1,\dots, z_N)
&
:=
W_{\l}(0,\dots,0;q,t;z_1,\dots, z_N).
\end{align}
\begin{prop}\label{prop:WH}
The $W$-polynomials satisfy
\begin{itemize}
\item The following reduction at $z_i=0$, for all $1\leq i\leq N$:
\begin{align}\label{eq:WredH1}
W_{\l}(x_1,\dots,x_N;q,t)=
H_{\l}(x_1,\dots,x_N;q,t).
\end{align}
\item The following reduction at $x_i=0$, for all $1\leq i\leq N$:
\begin{align}\label{eq:WredH2}
W_{\l}(q,t;z_1,\dots, z_N)=
H_{\l'}(z_1,\dots,z_N;t,q).
\end{align}
\end{itemize}
\end{prop}
\begin{proof}
The proof of (\ref{eq:WredH1}) is a straightforward consequence of (\ref{eq:ev}) and (\ref{eq:evW}). The proof of (\ref{eq:WredH2}) follows from (\ref{eq:WredH1}) and the property (\ref{eq:Wsymexch}).
\end{proof}
\begin{rmk}
The modified Macdonald polynomials $H_{\l}(x_1,\dots,x_N;q,t)$ satisfy
\begin{align}\label{eq:HinvH}
H_{\l}(x_1,\dots,x_N;q,t)=
t^{n(\l)}q^{n(\l')}
H_{\l'}(x_1,\dots,x_N;t^{-1},q^{-1}).
\end{align}
\end{rmk}
\begin{proof}
This relation is a consequence of the property (\ref{eq:Wsyminv}) and (\ref{eq:WredH1})-(\ref{eq:WredH2}).
\end{proof}
A combinatorial formula for $H_\l$ was obtained in \cite{HaglundHL1,HaglundHL2}. This formula, in particular, implies that the coefficient of each monomial of $H_\l$ is in $\mathbb{N}[q,t]$. Our construction in Section $\ref{sec:modM}$ gives combinatorial formulae for $H_\l$ and separately for $H_{\l'}$ which are manifestly positive. 
The construction of Section $\ref{sec:modM}$ leads to positive combinatorial formulae for $W_\l$, as well, but we omit these from the present work as the formulae for $H_\l$ are less technically involved.

\begin{rmk}
The modified Macdonald polynomials $H_{\l}(x_1,\dots,x_N;q,t)$ reduce to the modified Hall--Littlewood polynomials  $H_{\l}(x_1,\dots,x_N;t)$ when $q=0$
\begin{align}\label{eq:HtoHL}
H_{\l}(x_1,\dots,x_N;t)=
H_{\l}(x_1,\dots,x_N;0,t).
\end{align}
\end{rmk}


\subsection{Cauchy identities}\label{ssec:CI}
Macdonald polynomials have the following Cauchy and dual Cauchy identities
\begin{align}
\sum_{\lambda}
P_{\lambda}(x;q,t)
Q_{\lambda}(y;q,t)
&=
\prod_{i,j}
\frac{(t x_i y_j;q)_{\infty}}{(x_i y_j;q)_{\infty}}, \label{eq:CIPQ} \\
\sum_{\lambda}
P_{\lambda}(x;q,t)
P_{\lambda'}(y;t,q)
&=
\prod_{i,j}
(1+x_i y_j). \label{eq:CIPQdual}
\end{align}
A similar identity holds for $W_{\l}$.
\begin{prop}
The Cauchy identity for $W_{\l}(x;q,t;z)$ takes the form
\begin{align}\label{eq:CIW}
\sum_{\lambda}
\frac{1}{c'_{\l}(q,t)c_{\l}(q,t)}
W_{\lambda}(x;q,t;z)
W_{\lambda}(y;q,t;w)
=
\prod_{i,j}
\frac{(-z_i y_j;q,t)_{\infty}(-x_i w_j;q,t)_{\infty}}{(x_i y_j;q,t)_{\infty}(z_i w_j;q,t)_{\infty}}.
\end{align}
\end{prop}
\begin{proof}
Choose the alphabets of $P_\l$ and $Q_\l$ in (\ref{eq:CIPQ}) to be $X=(x^{(1)},\dots,x^{(N)})$ and $Y=(y^{(1)},\dots,y^{(N)})$ respectively, and expand both sides in power sums. The right hand side is known to be equal to
\begin{align}\label{eq:expp}
\exp\left( \sum_{r=1}^{\infty} \frac{1}{r}\frac{1-t^r}{1-q^r}p_r(X)p_r(Y)\right).
\end{align}
We apply the ring homomorphism $\pi^{X}_{X,Z}\pi^{Y}_{Y,W}$ to both sides of (\ref{eq:CIPQ}), writing the right hand side in the above form. This yields
\begin{align}\label{eq:cexp}
&\sum_{\lambda}
\pi^{X}_{X,Z}\left(P_{\lambda}(x;q,t) \right)
\pi^{Y}_{Y,W}\left(Q_{\lambda}(y;q,t)\right) 	\nonumber \\
&=
\exp\left( \sum_{r=1}^{\infty} \frac{1}{r}\frac{1-t^r}{1-q^r}\left(\frac{p_r(X)-(-1)^rp_r(Z)}{1-t^r}\right)
\left(\frac{p_r(Y)-(-1)^rp_r(W)}{1-t^r}\right) \right) 	\nonumber \\
&=
\exp\left( \sum_{r=1}^{\infty} \frac{1}{r}\frac{p_r(X)p_r(Y)-(-1)^rp_r(X)p_r(W)-(-1)^rp_r(Z)p_r(Y)+p_r(Z)p_r(W)}{(1-q^r)(1-t^r)}\right).
\end{align}
In the last expression the denominator $(1-t^r)^{-1}$ (or $(1-q^r)^{-1}$) can be expressed as a geometric series $1+t^r+t^{2 r}+\dots$. Hence we can convert the exponential in (\ref{eq:cexp}) into a product of exponentials each of which is of the form (\ref{eq:expp}). Thus we recover (\ref{eq:CIW}).
\end{proof}
\begin{cor}
We have the ``mixed'' Cauchy identities:
\begin{align}
\sum_{\lambda}
\frac{1}{c'_{\l}(q,t)}W_{\lambda}(x;q,t;z)
P_{\lambda}(y;q,t)
&=
\prod_{i,j}
\frac{(-z_i y_j;q)_{\infty}}{(x_i y_j;q)_{\infty}}, \label{eq:CIWQ} \\
\sum_{\lambda}
\frac{1}{c_{\l}(q,t)} 
W_{\lambda}(x;q,t;z)
P_{\lambda'}(w;t,q)
&=
\prod_{i,j}
\frac{(-x_i w_j;t)_{\infty}}{(z_i w_j;t)_{\infty}}. \label{eq:CIWP}
\end{align}
\end{cor}
\begin{proof}
In order to check the validity of (\ref{eq:CIWQ}) and (\ref{eq:CIWP}) we notice that:
\begin{align}\label{eq:qtcancel}
\frac{(x;q,t)_{\infty}}{(t x;q,t)_{\infty}}= (x;q)_{\infty},\qquad \frac{(x;q,t)_{\infty}}{(q x;q,t)_{\infty}}= (x;t)_{\infty},
\end{align}
which follow from the definition (\ref{weq:qtpoc}). 
The first identity (\ref{eq:CIWQ}) follows from (\ref{eq:CIW}) after setting $w=-t y$, in which case on the left hand side we substitute (\ref{eq:Wred1}),
while on the right hand side we use the first equation in (\ref{eq:qtcancel}):
\begin{align}
\prod_{i,j}
\frac{(-z_i y_j;q,t)_{\infty}(t x_i y_j;q,t)_{\infty}}{(x_i y_j;q,t)_{\infty}(-t z_i y_j;q,t)_{\infty}} 
=
\prod_{i,j} 
\frac{(-z_i y_j;q)_{\infty}}{(x_i y_j;q)_{\infty}}.
\end{align}
The second identity (\ref{eq:CIWP}) follows from (\ref{eq:CIW}) after setting $y=-q w$. In this case on the left hand side we use the reduction (\ref{eq:Wred2}), while on the right hand side we use the second equation in (\ref{eq:qtcancel}):
\begin{align}
\prod_{i,j}
\frac{(q z_i w_j;q,t)_{\infty}(-x_i w_j;q,t)_{\infty}}{(-q x_i w_j;q,t)_{\infty}(z_i w_j;q,t)_{\infty}}
=
\prod_{i,j} 
\frac{(-x_i w_j;t)_{\infty}}{(z_i w_j;t)_{\infty}}.
\end{align}
\end{proof}
If we further specialize $z=-t x$ in (\ref{eq:CIWQ}) (or $x=-q z$ in (\ref{eq:CIWP})) we will recover the standard Cauchy identity for Macdonald polynomials (\ref{eq:CIPQ}). On the other hand, if we set $x=-q z$ in (\ref{eq:CIWQ}) (or $z=-t x$ in (\ref{eq:CIWP})) and use (\ref{eq:bbprime}),  we will recover the dual Cauchy identity for Macdonald polynomials (\ref{eq:CIPQdual}).


\section{The modified Hall--Littlewood polynomials}\label{sec:HL}
In this section we give a lattice formulation of the modified Hall--Littlewood polynomials. Our purpose here is to lay the foundations of the approach that we subsequently extend to the study of modified Macdonald polynomials $H_{\l}(x;q,t)$. Thus we omit proofs and focus on the construction itself. 

Let us recall once again the combinatorial formula for $H_{\l}(x;t)$:
\begin{align}\label{eq:HLmonsum}
H_{\l}(x;t)= \sum_{\m} \MP_{\l,\m}(t) m_{\m}(x),
\end{align}
where the coefficients  $\MP_{\l,\m}(t)$ are given by the combinatorial expansion
\begin{align}\label{eq:HLmon}
\MP_{\l,\m}(t)= 
\sum_{\{\n\}} t^{c(\n)} 
\prod_{k=1}^{N-1}\prod_{i=1}^{n}
\binom{\n_{i}^{k+1}-\n_{i+1}^{k}}{\n_{i}^{k}-\n_{i+1}^{k}}_t,
\end{align}
where the summation runs over flags of partitions 
$\{\n\}=\{\emptyset =\n^{0} \subseteq \n^{1}\subseteq \dots \subseteq \n^{N}=\l'    \}$, such that $|\n^{k}|=\m_1+\dots +\m_k$, for all $1\leq k \leq N$, and
\begin{align}\label{eq:c}
c(\n)=\frac{1}{2} \sum_{k=1}^{N} \sum_{i=1}^n (\n_{i}^{k}-\n_{i}^{k-1})(\n_{i}^{k}-\n_{i}^{k-1}-1).
\end{align}
We now explain how formula (\ref{eq:HLmonsum}) can be viewed within the framework of our construction.


\subsection{The lattice path construction}\label{ssect:mHLsetup}
Consider a square lattice on a rectangular domain of length $n$ and height $N$. Splitting the lattice into $n$ columns of height $N$ we label the columns from $1$ to $n$ from right to left. Splitting the lattice into $N$ rows of length $n$ we label the rows from $1$ to $N$ from top to bottom. Next, consider $M$ up-right lattice paths which start at the lower horizontal boundary and finish at the right vertical boundary. In this setting no path touches the left and the top boundaries of the lattice. Fix the starting positions of the $M$ paths to be given by $M$ numbers corresponding to the label of the column where they start. These numbers can be ordered $\l=(\l_1,\l_2,\dots,\l_M)$ such that $\l_1=n$ and $\l_i\geq \l_{i+1}$; hence $\l$ is a partition. Let the paths end at any position on the right boundary. Hence we have the following setting:
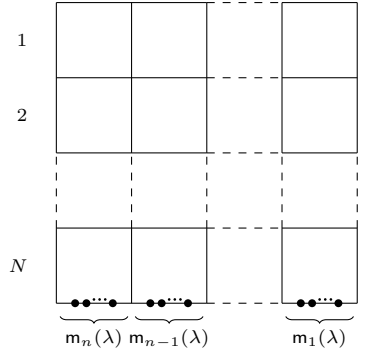
\begin{center}
\vspace{0.3cm}
\begin{tikzpicture}[scale=0.5,baseline=(current bounding box.center)]
\foreach\x in {0,...,4}{
\draw (0,2*\x) -- (4,2*\x);
}
\foreach\x in {0,...,4}{
\draw[dashed] (4,2*\x) -- (6,2*\x);
}
\foreach\x in {0,...,4}{
\draw (6,2*\x) -- (8,2*\x);
}
\foreach\y in {0,...,4}{
\draw (2*\y,0) -- (2*\y,2);
}
\foreach\y in {0,...,4}{
\draw[dashed] (2*\y,2) -- (2*\y,4);
}
\foreach\y in {0,...,4}{
\draw (2*\y,4) -- (2*\y,8);
}
\node[left] at (-0.5,7) {\tiny $1$};
\node[left] at (-0.5,5) {\tiny$2$};
\node[left] at (-0.5,1) {\tiny $N$};
\fill (0.5,0)  circle[radius=3pt];
\fill (0.8,0)  circle[radius=3pt];
\fill (1.5,0)  circle[radius=3pt];
\fill (1.0,0.1)  circle[radius=1pt];
\fill (1.15,0.1)  circle[radius=1pt];
\fill (1.3,0.1)  circle[radius=1pt];
\fill (1.5+2,0)  circle[radius=3pt];
\fill (0.5+2,0)  circle[radius=3pt];
\fill (0.8+2,0)  circle[radius=3pt];
\fill (1.0+2,0.1)  circle[radius=1pt];
\fill (1.15+2,0.1)  circle[radius=1pt];
\fill (1.3+2,0.1)  circle[radius=1pt];
\fill (1.5+6,0)  circle[radius=3pt];
\fill (0.5+6,0)  circle[radius=3pt];
\fill (0.8+6,0)  circle[radius=3pt];
\fill (1.0+6,0.1)  circle[radius=1pt];
\fill (1.15+6,0.1)  circle[radius=1pt];
\fill (1.3+6,0.1)  circle[radius=1pt];
\draw[decoration={brace,mirror,raise=5pt},decorate]
  (0+0.1,0) -- node[below=6pt] {\tiny $\cm_n(\l)$} (2-0.1,0);
  \draw[decoration={brace,mirror,raise=5pt},decorate]
  (2+0.1,0) -- node[below=6pt] {\tiny $\cm_{n-1}(\l)$} (4-0.1,0);
    \draw[decoration={brace,mirror,raise=5pt},decorate]
  (6+0.1,0) -- node[below=6pt] {\tiny $\cm_1(\l)$} (8-0.1,0);
\end{tikzpicture}~,
\vspace{0.3cm}
\end{center}
where dots on the lower boundary edges denote the starting positions of the paths and $\cm_k(\l)$ is the part multiplicity of $\l$; $\cm_k(\l)=\text{Card}\{i:\l_i=k\}$. We can write $\cm_k(\l)$ in terms of the dual partition $\l$ as $\cm_k(\l)=\l'_k-\l'_{k+1}$. At each face of the lattice an up-right path can go straight vertically, straight horizontally, or make a turn from the lower horizontal edge to the right vertical edge of the face or move from the left vertical edge to the upper horizontal edge of the face. We also suppose that the paths do not intersect; hence straight horizontal and straight vertical moves occurring in a single face should be treated as two turns. We impose no restriction on the number of paths per edge.  

A local configuration, i.e. a configuration of a face, is given by the data $(\r,\rt,\s,\st)$ assigned to the edges as follows:
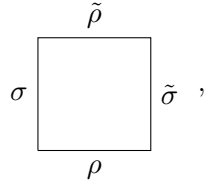
\begin{center}
\vspace{0.3cm}
\begin{tikzpicture}[scale=0.5,baseline=(current bounding box.center)]
\draw (0,0) -- (3,0) -- (3,3) -- (0,3) -- (0,0);
\node[left] at (0,1.5) {$\s$};
\node[right] at (3,1.5) {$\st$};
\node[below] at (1.5,0) {$\r$};
\node[above] at (1.5,3) {$\rt$};
\end{tikzpicture}~,
\vspace{0.3cm}
\end{center}
where $\r,\rt,\s,\st \in \mathbb{Z}_{\geq 0}$ count the numbers of paths touching the corresponding edge. 
\begin{ex}\label{ex:0}
The lattice path configuration with the edge data $\r=1,\rt=2, \s=3,\st=2$ is given by:
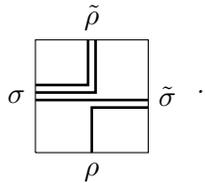
\begin{center}
\vspace{0.3cm}
\begin{tikzpicture}[scale=0.5,baseline=(current bounding box.center)]
\draw (0,0) -- (3,0) -- (3,3) -- (0,3) -- (0,0);
\node[left] at (0,1.5) {$\s$};
\node[right] at (3,1.5) {$\st$};
\node[below] at (1.5,0) {$\r$};
\node[above] at (1.5,3) {$\rt$};
 \begin{scope}[line width=1pt]
\draw (0,1.5+0.3) -- (1.5-0.1,1.5+0.3) -- (1.5-0.1,3);
\draw (0,1.5+0.1) -- (1.5+0.1,1.5+0.1) -- (1.5+0.1,3);;
\draw (0,1.5-0.1) -- (3,1.5-0.1);
\draw (1.5,0) -- (1.5,1.5-0.3) -- (3,1.5-0.3);
\end{scope}
\end{tikzpicture}~.
\vspace{0.3cm}
\end{center}
\end{ex}

Each local configuration is assigned the weight
\begin{align}\label{eq:LKir}
\ML^{\s,\st}_{\r,\rt}(x):=
t^{\frac{1}{2} \st (\st-1)}
\binom{\rt+\st}{\rt}_t
x^{\st},
\qquad 
\s+\r=\st+\rt,
\end{align}
where $x$ is called the spectral parameter. 
In the cases when $\s+\r\neq \st +\rt$, which breaks the continuity of the paths, we set $\ML^{\s,\st}_{\r,\rt}(x)=0$. The weights $\ML^{\s,\st}_{\r,\rt}(x)$ are integrable in the sense of the Yang--Baxter RLL equation \cite{BLZ} (see Section \ref{sec:proof}).

To each global configuration of paths with fixed boundary conditions given by $\l$ we assign the weight given by products of all local weights $\ML^{\s,\st}_{\r,\rt}(x_k)$, where $\s,\st,\r$ and $\rt$ are local edge occupation numbers and $k$ is the label of the row. The {\it{partition function}} $\MZ_\l(x;t)$ is defined as the sum of all global weights taken over all possible lattice path configurations with the boundary conditions given by $\l$.    


\subsection{Column partition functions}\label{ssect:mHLcolumns}
A configuration of a single column is given by the data $(\r,\rt,\s,\st)$, where $\s=(\s^{1},\dots,\s^N)$ and $\st=(\st^{1},\dots,\st^N)$ now are compositions of non-negative integers. The assignment of these numbers to the edges is given by 
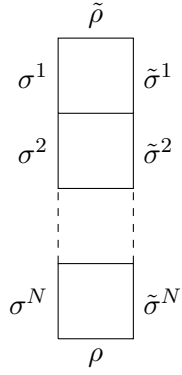
\begin{center}
\vspace{0.3cm}
\begin{tikzpicture}[scale=0.5,baseline=(current bounding box.center)]
\draw (0,0) -- (2,0) -- (2,2) -- (0,2) -- (0,0);
\draw (0,2+4)  -- (0,0+4) -- (2,0+4) -- (2,2+4) ;
\draw[dashed] (0,2) -- (0,2+2);
\draw[dashed] (2,2) -- (2,2+2);
\draw (0,0+6) -- (2,0+6) -- (2,2+6) -- (0,2+6) -- (0,0+6);
\node[left] at (0,1+6) {$\s^1$};
\node[right] at (2,1+6) {$\st^1$};
\node[left] at (0,1+4) {$\s^2$};
\node[right] at (2,1+4) {$\st^2$};
\node[left] at (0,1) {$\s^N$};
\node[right] at (2,1) {$\st^N$};
\node[below] at (1,0) {$\r$};
\node[above] at (1,8) {$\rt$};
\end{tikzpicture}~.
\vspace{0.3cm}
\end{center}
The number $\r$ tells us how many paths touch the bottom of the column while $\rt$ tells us how many paths touch the top of the column. Due to our choice of boundary conditions (up-right paths start at the bottom boundary according to a partition $\l$ and end anywhere at the right  boundary) the $k$-th column will have $\r=\cm_{k}(\l)$ and $\rt=0$. The compositions $\s$ and $\st$ must satisfy $|\st-\s|=\cm_{k}(\l)$, where $\st-\s$ is a composition whose parts are given by the differences of the corresponding parts of $\st$ and $\s$. 

The values of the horizontal edges $\r^k$, where $k=1,\dots,N+1$ is counted from top to bottom with $\r^1=\rt=0$ and $\r^{N+1}=\r=\cm_{k}(\l)$, are fixed by $\s$ and $\st$ due to the continuity of the paths and given by 
\begin{align}\label{eq:rhosig}
\r^k=\st^1+\dots +\st^{k-1}-(\s^1+\dots +\s^{k-1}).
\end{align}
Letting $x=(x_1,\dots x_N)$, the weight $\MQ^{i}_{\s,\st}(x;t)$ of the column with the label $i$ is factorized into the product of the weights of its faces given in (\ref{eq:LKir}), where we take all spectral parameters to be different:
\begin{align}\label{eq:QKirdef}
\MQ^{i}_{\s,\st}(x;t):=
\prod_{k=1}^{N}
\ML_{\r^{k+1},\r^k}^{\s^k,\st^k}(x_k).
\end{align}
Substituting the elements (\ref{eq:LKir}) we have
\begin{align}\label{eq:QKir}
\MQ^{i}_{\s,\st}(x;t)=
t^{\frac{1}{2} \sum_{k=1}^N \st^k (\st^k-1)}
\prod_{k=1}^{N}
\binom{\r^{k}+\st^{k}}{\r^{k}}_t
\times x_1^{\st^1} \dots x_N^{\st^N}.
\end{align}
In the view of (\ref{eq:rhosig}) it is convenient to introduce two partitions $\n=(\n^1,\dots,\n^{N})$ and $\nt=(\nt^1,\dots,\nt^{N})$ which are defined by $\n^k=\s^1+\dots +\s^{k}$ and $\nt^k=\st^1+\dots +\st^{k}$. Note also that $\nt^k\geq \n^k$, due to the boundary conditions and since the paths are up-right. 
It is therefore natural to set $\n^0=0$ and $\nt^0=0$. 
With this notation we have $\s^{k}=\n^{k}-\n^{k-1}$, $\st^{k}=\nt^{k}-\nt^{k-1}$ and $\r^{k}=\nt^{k-1}-\n^{k-1}$. The weight of the $i$-th column becomes
\begin{align}\label{eq:QKir}
\MQ^{i}_{\s,\st}(x;t)=
t^{\frac{1}{2} \sum_{k=1}^{N} (\nt^{k}-\nt^{k-1}) (\nt^{k}-\nt^{k-1}-1)}
\prod_{k=1}^{N-1}
\binom{\nt^{k+1}-\n^{k} }{\nt^{k}-\n^{k}}_t
\times 
\prod_{k=1}^{N} 
x_k^{\nt^{k}-\nt^{k-1}}.
\end{align}
%


\subsection{Lattice partition functions}\label{ssect:mHLlattice}
Let now $\{\s\}$ be a collection of compositions $\{\s\}=(\s_1,\dots,\s_n)$ with $\s_i=(\s_i^1,\dots,\s_i^N)$, $i=1,\dots n$ and $\s_{n+1}=0$. 
In order to compute the partition function $\MZ_\l(x;t)$ with the positions of the paths at the lower edges given by the partition $\l$ we need to join together $n$ columns $\MQ^{i}_{\s_{i+1},\s_{i}}$, where $|\s_{i}-\s_{i+1}|=\cm_i(\l)$ for $i=1,\dots,n$, and sum over all possible arrangements of the paths on the vertical edges, i.e. sum over all compatible $\s_{1},\dots,\s_{n}$:
\begin{align}\label{eq:ZKir}
\MZ_\l(x;t)=\sum_{\{\s\}}\prod_{i=1}^n \MQ^{i}_{\s_{i+1},\s_i}(x;t).
\end{align}
Let $\{\n\}$ be a collection of partitions $\{\n\}=(\n_1,\dots,\n_n)$ with $\n_i=(\n_i^1,\dots,\n_i^N)$, $i=1,\dots n$ which are given by
 $\n_i^k=\s_i^1+\dots +\s_i^{k}$. As we noted above $\n_i^k\geq \n_{i+1}^k$, therefore we can denote by $\n^k$ the partitions $(\n_1^k,\dots,\n_n^k)$, $k=1,\dots N$, letting also $\n_i^k=0$ for $i=n+1$. With this change of the summation indices we can write the partition function as the following sum of products of the column weights:
\begin{align}\label{eq:ZKir2}
\MZ_\l(x;t)=
\sum_{\{\n\}}
t^{c(\n)}
\prod_{i=1}^{n}
\prod_{k=1}^{N-1}
\binom{\n_i^{k+1}-\n_{i+1}^{k} }{\n_i^{k}-\n_{i+1}^{k}}_t
\times
\prod_{k=1}^{N} 
x_k^{|\n^{k}-\n^{k-1}|},
\end{align}
where $c(\n)$ appears in (\ref{eq:c}) and the summation $\{\n\}$ runs over unrestricted flags $\{\n\}=\{\emptyset =\n^{0} \subseteq \n^{1}\subseteq \dots \subseteq \n^{N}=\l'    \}$. Let us explain how this summation set should be understood in terms of lattice paths. Start with $\n^{N}=\l'$. Since all paths must end up at the right vertical boundary of the column with the label $1$, this means that $\n_1^N$ is equal to $\l_1'$ which is the length of $\l$. At the same time all particles in the columns $i>1$ must pass through the right boundary of the column with the label $2$; this means that $\n_2^N=\l'_2$ which is the length of the partition obtained from $\l$ by removing all parts equal to $1$. Continuing this logic we find that $\n^N=\l'$. The counting of $\n^{k}_i$ with $k<N$ is analogous, the difference is that now we consider all paths passing through the right edges of the first $k$ faces of the column $i$. 
\begin{ex}\label{ex:1}
Here is an example in the case of the $3\times 3$ square lattice. For each element $\n$ of a flag the numbers $\n^k_i$ count how many paths hit the thick vertical gray lines:
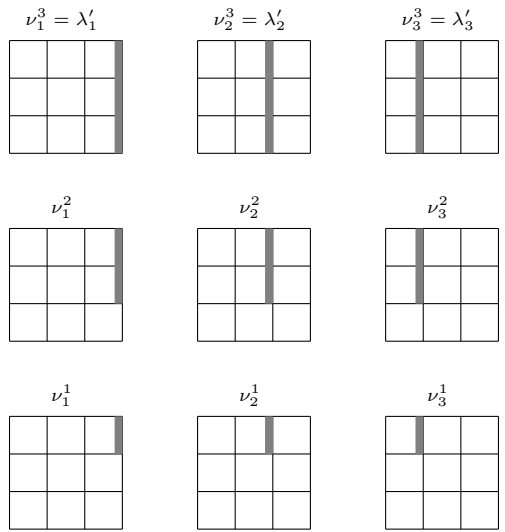
\begin{center}
\vspace{0.3cm}
\begin{tikzpicture}[scale=0.5,baseline=(current bounding box.center)]
\def\dw{5}
\def\dh{5}
\def\ddw{14}
\foreach\a in {0,...,2}{
	\foreach\b in {0,...,2}{
		\foreach\x in {0,...,3}{
		\draw (0+\dw*\a,\x-\dh*\b) -- (3+\dw*\a,\x-\dh*\b);
		\draw (\x+\dw*\a,0-\dh*\b) -- (\x+\dw*\a,3-\dh*\b);
						}
					}
				}
 \begin{scope}[color=gray,line width=3pt]
\foreach\x in {0,...,2}{
      \draw (3-0.1+\x*\dw-\x,0) -- (3-0.1+\x*\dw-\x,3);
				}
\foreach\x in {0,...,2}{
      \draw (3-0.1+\x*\dw-\x,1-\dh) -- (3-0.1+\x*\dw-\x,3-\dh);
      				}
\foreach\x in {0,...,2}{
      \draw (3-0.1+\x*\dw-\x,2-2*\dh) -- (3-0.1+\x*\dw-\x,3-2*\dh);
      				}
   \end{scope}
\node[above] at (1.4,3) {\tiny{$\n_1^3=\l'_1$}};
\node[above] at (1.4+\dw,3) {\tiny{$\n_2^3=\l'_2$}};
\node[above] at (1.4+2*\dw,3) {\tiny{$\n_3^3=\l'_3$}};
\node[above] at (1.4,3-\dh) {\tiny{$\n_1^2$}};
\node[above] at (1.4+\dw,3-\dh) {\tiny{$\n_2^2$}};
\node[above] at (1.4+2*\dw,3-\dh) {\tiny{$\n_3^2$}};%
\node[above] at (1.4,3-2*\dh) {\tiny{$\n_1^1$}};
\node[above] at (1.4+\dw,3-2*\dh) {\tiny{$\n_2^1$}};
\node[above] at (1.4+2*\dw,3-2*\dh) {\tiny{$\n_3^1$}};
\end{tikzpicture}~.
\vspace{0.3cm}
\end{center}
Take for example the following lattice path configuration:
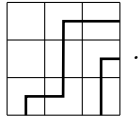
\begin{center}
\vspace{0.3cm}
\begin{tikzpicture}[scale=0.5,baseline=(current bounding box.center)]
		\foreach\x in {0,...,3}{
		\draw (0,\x) -- (3,\x);
		\draw (\x,0) -- (\x,3);
						}
 \begin{scope}[line width=1pt]
\draw (0.5,0) -- (0.5,0.5) -- (1.5-0.01,0.5) -- (1.5,2.5)  -- (3,2.5);
\draw (2.5,0) -- (2.5,1.5) -- (3,1.5);
\end{scope}
\end{tikzpicture}~.
\vspace{0.3cm}
\end{center}
By drawing gray lines on top of this configuration as shown above we can read the corresponding numbers $\n_i^k$ of the element $\n$
\begin{align*}
&\n_1^3=2, \qquad \n_2^3=1, \qquad \n_3^3=1,\\
&\n_1^2=2, \qquad \n_2^2=1, \qquad \n_3^2=0,\\
&\n_1^1=1, \qquad \n_2^1=1, \qquad \n_3^1=0.
\end{align*}
In this example the flag corresponds to the partition $\l=(3,1,0)$, since $\l'=\n^3$.
\end{ex}
Now we introduce a composition $\m=(\m_1,\dots,\m_N)$ and define the numbers $\m_k$ by $\m_k=|\n^{k}-\n^{k-1}|$. For a fixed ${\n}$ the value of  $\m_k$  tells us how many paths hit the right edges of every face in the $k$-th row. Therefore in (\ref{eq:ZKir2}) each $x_k$ appears as $x_k^{\m_k}$. 
We can now reverse the logic and fix a composition $\m=(\m_1,\dots,\m_N)$ and demand that the partitions $\n^k$ of the flag should be such that $|\n^{k}-\n^{k-1}|=\m_k$. This puts a restriction on the flags: $|\n^{k}|=\m_1+\dots +\m_k$. Therefore one sums over all possible $\m$ such that $|\m|=|\l|$ and finally over $\m$-restricted flags:
\begin{align}\label{eq:QKir3}
\MZ_\l(x;t)=
\sum_{\m: |\m|=|\l|}
x_1^{\m_1}\dots x_N^{\m_N}
\sum_{\{\n\}}
t^{c(\n)}
\prod_{i=1}^{n}
\prod_{k=1}^{N-1}
\binom{\n_i^{k+1}-\n_{i+1}^{k} }{\n_i^{k}-\n_{i+1}^{k}}_t.
\end{align}
The partition function (\ref{eq:QKir3}), in fact, equals the modified Hall--Littlewood polynomial $H_\l(x;t)$. We can rewrite it as a sum of monomial symmetric functions with coefficients equal to $\MP_{\l,\m}(t)$, thus recovering the combinatorial formula (\ref{eq:HLmonsum})-(\ref{eq:HLmon}). This follows from the fact that the coefficient
\begin{align}\label{eq:coefH}
\sum_{\{\n\}}
t^{c(\n)}
\prod_{i=1}^{n}
\prod_{k=1}^{N-1}
\binom{\n_i^{k+1}-\n_{i+1}^{k} }{\n_i^{k}-\n_{i+1}^{k}}_t,
\end{align}
which runs over $\m$-restricted flags is invariant under permutations of $\m$. This statement holds because $\MZ$ turns out to by symmetric in $x_1,\dots,x_N$\footnote{The symmetry in $x_1,\dots,x_N$ follows from the RLL equation satisfied by the weights (\ref{eq:LKir}) and the weights of the fused $R$-matrix (not given in this paper).}. 
Therefore the summation over $\m: |\m|=|\l|$ can be split into two summations: one over partitions $\m$ satisfying $\m: |\m|=|\l|$ and the second over permutations of $\m$. Since the coefficient (\ref{eq:coefH}) is invariant under the permutations of $\m$ we can compute the second of these two sums and find the monomial symmetric function $m_\m(x)$. The corresponding coefficients coincide with $\MP_{\l,\m}(t)$ in (\ref{eq:HLmon}).
\begin{ex}
Consider $\l=(3,1,0)$ and $\m=(2,1,1)$. The domain is given by a $3\times 3$ grid where we have two paths: one path starts from the lower edge of the first column and the other from the lower edge of the third column similarly as in Example \ref{ex:1}. The coefficient $\MP_{(3,1,0),(2,1,1)}(t)$ is a sum of three terms which come from three configurations
\begin{align}
\MP_{(3,1,0),(2,1,1)}(t)=
~
1\times
\begin{tikzpicture}[scale=0.5,baseline=(current bounding box.center)]
		\foreach\x in {0,...,3}{
		\draw (0,\x) -- (3,\x);
		\draw (\x,0) -- (\x,3);
						}
 \begin{scope}[line width=1pt]
\draw (0.5,0) -- (0.5,1.5) -- (1.5,1.5) -- (1.5,2.5)  -- (3,2.5);
\draw (2.5,0) -- (2.5,0.5) -- (3,0.5);
\end{scope}
\end{tikzpicture}~
+
1\times
\begin{tikzpicture}[scale=0.5,baseline=(current bounding box.center)]
		\foreach\x in {0,...,3}{
		\draw (0,\x) -- (3,\x);
		\draw (\x,0) -- (\x,3);
						}
 \begin{scope}[line width=1pt]
\draw (0.5,0) -- (0.5,0.5) -- (1.5,0.5) -- (1.5,2.5) -- (3,2.5);
\draw (2.5,0) -- (2.5,1.5) -- (3,1.5);
\end{scope}
\end{tikzpicture}~
+
t
\times
\begin{tikzpicture}[scale=0.5,baseline=(current bounding box.center)]
		\foreach\x in {0,...,3}{
		\draw (0,\x) -- (3,\x);
		\draw (\x,0) -- (\x,3);						}
 \begin{scope}[line width=1pt]
\draw (0.5,0) -- (0.5,0.5) -- (1.5,0.5) -- (1.5,1.5) -- (2.5-0.06,1.5) -- (2.5-0.06,2.5+0.06)  -- (3,2.5+0.06);
\draw (2.5+0.06,0) -- (2.5+0.06,2.5-0.06) -- (3,2.5-0.06);
\end{scope}
\end{tikzpicture}~.
\end{align}
We have $\MP_{(3,1,0),(2,1,1)}(t)=2+t$. One can verify that among all elements of the unrestricted flag of $\l$ we have only three $\m$-restricted elements, i.e. with 
$|\n^1|=2, |\n^2|=3 $  and $|\n^3|=4$.
As we mentioned the coefficient $\MP_{\l,\m}(t)$ is symmetric under the permutations of $\m$. Thus we can also choose $\m=(1,1,2)$ and consider configurations with the restriction $|\n^1|=1, |\n^2|=2 $  and $|\n^3|=3$. In this case the weights of the partitions $\n^k$ are smaller than before which leads to a formula with fewer elements in the flag. The coefficient is given by a sum of two configurations
\begin{align}
\MP_{(3,1,0),(2,1,1)}(t)=
~
1\times
\begin{tikzpicture}[scale=0.5,baseline=(current bounding box.center)]
		\foreach\x in {0,...,3}{
		\draw (0,\x) -- (3,\x);
		\draw (\x,0) -- (\x,3);
						}
 \begin{scope}[line width=1pt]
\draw (0.5,0) -- (0.5,0.5) -- (1.5,0.5) -- (1.5,1.5) -- (2.5,1.5) -- (2.5,2.5)  -- (3,2.5);
\draw (2.5,0) -- (2.5,0.5) -- (3,0.5);
\end{scope}
\end{tikzpicture}~
+
(1+t)
\times
\begin{tikzpicture}[scale=0.5,baseline=(current bounding box.center)]
		\foreach\x in {0,...,3}{
		\draw (0,\x) -- (3,\x);
		\draw (\x,0) -- (\x,3);						}
 \begin{scope}[line width=1pt]
\draw (0.5,0) -- (0.5,0.5) -- (2.5-0.06,0.5) -- (2.5-0.06,2.5)  -- (3,2.5);
\draw (2.5+0.06,0) -- (2.5+0.06,1.5) -- (3,1.5);
\end{scope}
\end{tikzpicture}~.
\end{align}
\end{ex}
For completeness the above discussion needs a proof (independent from matching it with (\ref{eq:HLmonsum})-(\ref{eq:HLmon})) that $\MZ_\l$ is equal to the modified Hall--Littlewood polynomial $H_\l(x;t)$. This can be done along the same lines as our construction of $H_\l(x;q,t)$ in Sections \ref{sec:modM} and \ref{sec:proof}. Alternatively, one can obtain the above construction of $H_\l(x;t)$ from our construction of $H_\l(x;q,t)$ in which $q$ is set to $0$. The formula for the coefficients $\MP_{\l,\m}(t)$ in (\ref{eq:HLmon}) can be recovered from $\MP_{\l,\m}(q,t)$ in (\ref{eq:HmonP-intro}) by setting $q=0$. From (\ref{Phi-def}) and (\ref{eq:Phiz1}) we have
\begin{align}\label{eq:Phi01}
\Phi_{\n|\nt}(1;t)=
\prod_{k=1}^{N-1}
\binom{\nt^{k+1}}{\nt^{k}}_t, \qquad
\Phi_{\n|\nt}(0;t)
=
\prod_{k=1}^{N-1}
\binom{\nt^{k+1}-\n^{k}}{\nt^{k}-\n^{k}}_t.
\end{align}
Using these expressions we can specialize $q=0$ in (\ref{eq:HmonP-intro}) and write the coefficient $\MP_{\l,\m}(0,t)$
\begin{align}\label{eq:HmonPq0}
\MP_{\l,\m}(0,t)=
\sum_{\{\n\}}
t^{\chi(\n)}
\prod_{k=1}^{N-1}
\prod_{i\leq j}
\binom{\n_{i,j}^{k+1}-\n_{i+1,j}^{k}}{\n_{i,j}^{k}-\n_{i+1,j}^{k}}_t.
\end{align}
Because $q=0$ many terms are absent from the summation set $\{\n\}$ and the non-vanishing terms correspond to $n$ flags. Define $\b_i^k$ by the formula
\begin{align}\label{eq:nubeta}
\n_{i,i}^k=\b_i^k-\sum_{j>i}\n_{i,j}^k,\qquad i=1,\dots,n,\quad k=1,\dots,N. 
\end{align}
The set $\{\b\}=\{\emptyset =\b^{0} \subseteq \b^{1}\subseteq \dots \subseteq \b^{N}=\l'    \}$ is a flag such that $|\b^{k}|=\m_1+\dots +\m_k$, $1\leq k \leq N$. This flag coincides with the flag in the summation in (\ref{eq:coefH}). 
Thus we split the summation set in (\ref{eq:HmonPq0}) into a sum over $\b$ and a sum over $\{\n\}_\b$ defined as the set of partitions $\n_{i,j}=(\n_{i,j}^{1},\dots,\n_{i,j}^N)$, $1\leq i< j\leq n$, satisfying 
\begin{align}
\label{eq:sumset1}
\n_{i,j}^{N}
&=\b_j^N-\b_{j+1}^N=\l'_j-\l'_{j+1},\\
\label{eq:sumset2}
\sum_{j>i}\n_{i,j}^k
&\leq 
\b_i^k.
\end{align}
Let $\chi'(\b,\n)$ be equal to $\chi(\n)$ in which $\n_{i,i}^k$ is expressed in terms of $\b_i^k$ according to (\ref{eq:nubeta}). Express (\ref{eq:HmonPq0}) in terms of $\b$
\begin{align}\label{eq:Psum0}
\MP_{\l,\m}(q,t)=
\sum_{\{\b\}}
\sum_{\{\n\}_\b}
\MB_\b(t),
\end{align}
where $\MB_\b(t)$ is defined by
\begin{align}
\label{eq:MB}
&\MB_\b(t):=
\prod_{k=1}^{N-1}
\binom{\b_{n}^{k+1}}{\b_{i}^k}_t
\sum_{\n\in D_\b}
t^{\chi'(\b,\n)}
\nonumber \\
\times
\prod_{k=1}^{N-1}
\prod_{i=1}^{n-1}
\binom{\b_{i}^{k+1}-\sum_{j=i+1}^n\n_{i,j}^{k+1}}{\b_{i}^k-\sum_{j=i+1}^n\n_{i,j}^{k}}_t
&
\binom{\n_{i,{i+1}}^{k+1}+\sum_{j=i+2}^n \n_{i+1,j}^k-\b_{i+1}^k}{\n_{i,i+1}^{k}+\sum_{j=i+2}^n \n_{i+1,j}^k-\b_{i+1}^k}_t
\prod_{j=i+2}^n
\binom{\n_{i,j}^{k+1}-\n_{i+1,j}^{k}}{\n_{i,j}^{k}-\n_{i+1,j}^{k}}_t.
\end{align}
Therefore the identification of $\MP_{\l,\m}(0,t)$ with $\MP_{\l,\m}(t)$ follows from the identity
\begin{align}\label{eq:sumnu}
\MB_\b(t)=
 t^{c(\b)} 
\prod_{k=1}^{N-1}\prod_{i=1}^{n}
\binom{\b_{i}^{k+1}-\b_{i+1}^{k}}{\b_{i}^{k}-\b_{i+1}^{k}}_t.
\end{align}
A brute force proof of this equation can be accomplished by an iterated application of the formula (\ref{eq:binom2}).

We would like to note that there are other ways of writing $H_{\l}(x;t)$ as a lattice path partition function. These other constructions would lead to different types of formulae (see for example Section \ref{sec:modM}, formula (\ref{eq:HZz}) with $q=0$).


\section{Construction of the modified Macdonald polynomials}\label{sec:modM}
In this section we give two lattice formulations of the modified Macodnald polynomials $H_{\l}(x;q,t)$. The logic is similar to the Hall--Littlewood case, however, the local weights and the boundary conditions are different. We present the construction in this section and postpone the proofs to Section \ref{sec:proof}.


\subsection{The lattice path construction}\label{ssect:setup}
Once again we consider a square lattice on a rectangular domain of length $n$ and height $N$, and the labelling of the columns and the rows is from right to left and from top to bottom respectively, as before. The $M=\ell(\l)$ up-right paths start at the lower boundary and end at the right boundary, however, this time they can reach the top boundary of a column and reappear at the bottom boundary in the same column and continue the up-right motion. The only restriction to this is that the paths are not allowed to wind in the columns where they started. 
This turns the lattice on a planar square domain into a lattice on a cylinder with a seam, where the seam represents the identified top and bottom boundaries. Each path is coloured and its colour is given by the label corresponding to the column of its starting position. Hence we have the following setting:
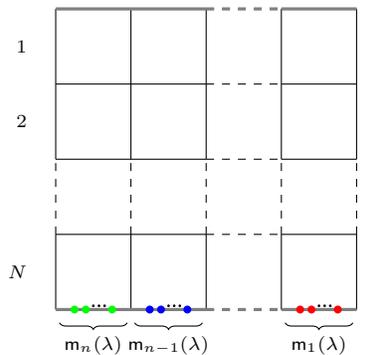
\begin{center}
\vspace{0.3cm}
\begin{tikzpicture}[scale=0.5,baseline=(current bounding box.center)]
\foreach\x in {0,...,4}{
\draw (0,2*\x) -- (4,2*\x);
}
\foreach\x in {0,...,4}{
\draw[dashed] (4,2*\x) -- (6,2*\x);
}
\foreach\x in {0,...,4}{
\draw (6,2*\x) -- (8,2*\x);
}
\foreach\y in {0,...,4}{
\draw (2*\y,0) -- (2*\y,2);
}
\foreach\y in {0,...,4}{
\draw[dashed] (2*\y,2) -- (2*\y,4);
}
\foreach\y in {0,...,4}{
\draw (2*\y,4) -- (2*\y,8);
}
\node[left] at (-0.5,7) {\tiny $1$};
\node[left] at (-0.5,5) {\tiny$2$};
\node[left] at (-0.5,1) {\tiny $N$};
\draw[very thick, gray] (0,0) -- (4,0);
\draw[very thick, dashed, gray] (4,0) -- (6,0);
\draw[very thick, gray] (6,0) -- (8,0);
\draw[very thick, gray] (0,8) -- (4,8);
\draw[very thick, dashed, gray] (4,8) -- (6,8);
\draw[very thick, gray] (6,8) -- (8,8);
\fill[green] (0.5,0)  circle[radius=3pt];
\fill[green] (0.8,0)  circle[radius=3pt];
\fill[green] (1.5,0)  circle[radius=3pt];
\fill (1.0,0.1)  circle[radius=1pt];
\fill (1.15,0.1)  circle[radius=1pt];
\fill (1.3,0.1)  circle[radius=1pt];
\fill[blue] (1.5+2,0)  circle[radius=3pt];
\fill[blue] (0.5+2,0)  circle[radius=3pt];
\fill[blue] (0.8+2,0)  circle[radius=3pt];
\fill (1.0+2,0.1)  circle[radius=1pt];
\fill (1.15+2,0.1)  circle[radius=1pt];
\fill (1.3+2,0.1)  circle[radius=1pt];
\fill[red] (1.5+6,0)  circle[radius=3pt];
\fill[red] (0.5+6,0)  circle[radius=3pt];
\fill[red] (0.8+6,0)  circle[radius=3pt];
\fill (1.0+6,0.1)  circle[radius=1pt];
\fill (1.15+6,0.1)  circle[radius=1pt];
\fill (1.3+6,0.1)  circle[radius=1pt];
\draw[decoration={brace,mirror,raise=5pt},decorate]
  (0+0.1,0) -- node[below=6pt] {\tiny $\cm_n(\l)$} (2-0.1,0);
  \draw[decoration={brace,mirror,raise=5pt},decorate]
  (2+0.1,0) -- node[below=6pt] {\tiny $\cm_{n-1}(\l)$} (4-0.1,0);
    \draw[decoration={brace,mirror,raise=5pt},decorate]
  (6+0.1,0) -- node[below=6pt] {\tiny $\cm_1(\l)$} (8-0.1,0);
\end{tikzpicture}~,
\vspace{0.3cm}
\end{center}
where the dots at the lower edges denote the starting positions of the paths, and the colours of the dots correspond to the colours of the associated paths. The gray lines indicate the identification of the boundaries which now permit the winding of the paths. 
The number of paths starting at a column $k$ is again given by $\cm_k(\l)$, the multiplicity of $k$ in $\l$. We impose no restriction on the number of paths per edge.  

A local configuration, i.e. a configuration of a face, is given by the data $(\r,\rt,\s,\st)$ assigned to the edges as follows:
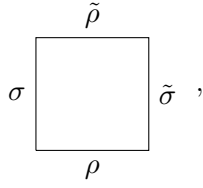
\begin{center}
\vspace{0.3cm}
\begin{tikzpicture}[scale=0.5,baseline=(current bounding box.center)]
\draw (0,0) -- (3,0) -- (3,3) -- (0,3) -- (0,0);
\node[left] at (0,1.5) {$\s$};
\node[right] at (3,1.5) {$\st$};
\node[below] at (1.5,0) {$\r$};
\node[above] at (1.5,3) {$\rt$};
\end{tikzpicture}~,
\vspace{0.3cm}
\end{center}
where $\r=(\r_{1},\dots,\r_n)$, $\rt=(\rt_{1},\dots,\rt_n)$, $\s=(\s_{1},\dots,\s_n)$ and $\st=(\st_{1},\dots,\st_n)$ are now compositions where each number with the index $j$ counts how many paths of colour $j$ are touching the corresponding edge. 
\begin{ex}\label{ex:N}
The lattice path configuration with three colours $j=1$ (red), $j=2$ (blue) and $j=3$ (green)  with the edge data $\r=(1,1,0),\rt=(0,1,1), \s=(1,1,1),\st=(2,1,0)$ is given by:
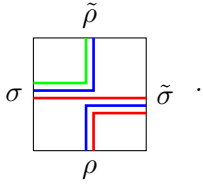
\begin{center}
\vspace{0.3cm}
\begin{tikzpicture}[scale=0.5,baseline=(current bounding box.center)]
\draw (0,0) -- (3,0) -- (3,3) -- (0,3) -- (0,0);
\node[left] at (0,1.5) {$\s$};
\node[right] at (3,1.5) {$\st$};
\node[below] at (1.5,0) {$\r$};
\node[above] at (1.5,3) {$\rt$};
 \begin{scope}[line width=1pt]
\draw[green] (0,1.5+0.3) -- (1.5-0.1,1.5+0.3) -- (1.5-0.1,3);
\draw[blue] (0,1.5+0.1) -- (1.5+0.1,1.5+0.1) -- (1.5+0.1,3);;
\draw[red] (0,1.5-0.1) -- (3,1.5-0.1);
\draw[blue] (1.5-0.1,0) -- (1.5-0.1,1.5-0.3) -- (3,1.5-0.3);
\draw[red] (1.5+0.1,0) -- (1.5+0.1,1.5-0.5) -- (3,1.5-0.5);
\end{scope}
\end{tikzpicture}~.
\vspace{0.3cm}
\end{center}
\end{ex}
For two compositions of integers $a=(a_1,\dots,a_n)$ and $b=(b_1,\dots,b_n)$ we will use the following notations:
\begin{align*}
a.b
=\sum_{i=1}^n a_i b_i, \qquad
\binom{a}{b}_t
=\prod_{i=1}^n \binom{a_i}{b_i}_t, \qquad
\sum_{a=0}^{b}
=\sum_{a_1=0}^{b_1}\dots \sum_{a_n=0}^{b_n}.
\end{align*}
The relevant local weights for constructing partition functions corresponding to the polynomials $W(x;q,t;z)$ are given by the following formula:
\begin{align}
\label{eq:Lone}
&\ML^{\s,\st}_{\r,\rt}(x,z)= 
\sum_{\k=0}^{\st}
 t^{\frac{1}{2}(\k.\k-|\k|)+\sum_{l<j}\st_l(\rt_j+\k_j)-\sum_{l<j}(\st_l-\k_l,\s_j)}\binom{\rt+\k}{\k}_t\binom{\r}{\st-\k}_t
  x^{|\k|}z^{|\st-\k|}.
\end{align}
The continuity of the paths of each colour must be preserved, so if $\s_j+\r_j\neq \st_j +\rt_j$ for any $j$ then we set $\ML^{\s,\st}_{\r,\rt}(x,z)=0$. The matrix in (\ref{eq:Lone}) can be obtained from the $R$-matrix associated to the quantum affine algebra $\Uq$ for symmetric tensor representations with arbitrary weights \cite{BosM,KunibaMMO}. In Section \ref{sec:proof} we show how (\ref{eq:Lone}) can be calculated starting from the fundamental $L$-matrix and performing fusion.  

According to Proposition \ref{prop:WH}, the construction of the modified Macdonald polynomials $H_{\l}$ can be carried out using the two special cases of the above weights
\begin{align}
\label{eq:MLx}
&\ML^{\s,\st}_{\r,\rt}(x,0)=
t^{\frac{1}{2}(|\st|^2-|\st|)+\sum_{l<j}\st_l\rt_j}
 \binom{\st+\rt}{\st}_t
 x^{|\st|} , \\
\label{eq:MLz}
&\ML^{\s,\st}_{\r,\rt}(0,z)=
t^{\sum_{l<j}\st_l (\rt_j-\s_j)}
\binom{\r}{\st}_t
z^{|\st|}.
\end{align}
The matrix $\ML^{\s,\st}_{\r,\rt}(0,z)$ has a special property which is due to the  binomial coefficient. It says that whenever $\r_i<\st_i$, for any $1\leq i\leq n$, then the matrix element vanishes. This $L$-matrix is closely related to the $R$-matrix introduced by Povolotsky \cite{Povo13}. It also essentially appears in the explicit formulas of the $R$-matrix and the $Q$-operators in the work of Mangazeev \cite{Mang14a,Mang14b}. Let us also remark that the above weights $\ML^{\s,\st}_{\r,\rt}(x,0)$, whose indices are compositions, can be expressed through the weights of the modified Hall--Littlewood construction $\ML^{\s_j,\st_j}_{\r_j,\rt_j}(x)$ given in (\ref{eq:LKir}), whose indices are integers:
\begin{align}\label{eq:LLkir}
\ML^{\s,\st}_{\r,\rt}(x,0)=
t^{\sum_{j<l}(\rt_j +\st_j)\st_l}
\prod_{j=1}^n 
\ML^{\s_j,\st_j}_{\r_j,\rt_j}(x).
\end{align}
For the weights (\ref{eq:MLx}) and (\ref{eq:MLz}) we will use the following notation
\begin{align}
\label{eq:Lx}
&L^{\s,\st}_{\r,\rt}(x):=
\ML^{\s,\st}_{\r,\rt}(x,0),\\
\label{eq:Lz}
&L'{}^{\s,\st}_{\r,\rt}(z):=
\ML^{\s,\st}_{\r,\rt}(0,z).
\end{align}
These two weights give two constructions which we will write in parallel.  

To each global configuration of paths with fixed boundary conditions given by $\l$ we assign the global weight given by products of all local weights $\ML^{\s,\st}_{\r,\rt}(x_k,z_k)$ multiplied by $q^{\omega}$ (with $q\in \mathbb{C}$ and $|q|<1$), where $\s,\st,\r$ and $\rt$ are local edge occupation states, $k$ is the label of the row and $\omega$ is given by: 
\begin{align}\label{eq:qomega}
\omega=\sum_{1\leq i<j\leq n} (j-i)s_{i,j},
\end{align}
where the numbers $s_{i,j}$ count how many times a path of colour $j$ winds in a column with the label $i$. The {\it{partition function}} $\MZ_\l(x;q,t;z)$ is defined as the sum of all global weights taken over all possible lattice path configurations. In the same way we define the {\it{partition functions}}  $\MZ_\l(x;q,t)$ and $\MZ'_\l(x;q,t)$ whose local weights are instead given by $L^{\s,\st}_{\r,\rt}(x_k)$ and $L'^{\s,\st}_{\r,\rt}(x_k)$, respectively. In contrast to the Hall--Littlewood case, the paths now are allowed to wind which gives rise to an infinite number of possible global configurations. 


\subsection{Column partition functions}\label{ssect:columns} 
In this subsection our goal is to derive two expressions for the column partition functions. One expression is based on the face weights (\ref{eq:Lx}) and the other on (\ref{eq:Lz}). The partition functions $\MZ_\l(x;q,t)$ and $\MZ'_\l(x;q,t)$ are given by sums of products of the two column partition functions. This is possible since the weight $q^{\omega}$ is factorized over the column index.  

Fix $i: 1\leq i\leq n$. 
A configuration of a single column with the label $i$ is given by the data $(\r,\rt,\s,\st)$ where $\r=(\r_{1},\dots,\r_n)$ and $\rt=(\rt_{1},\dots,\rt_n)$ are compositions and $\s=(\s^{1},\dots,\s^N)$ and $\st=(\st^{1},\dots,\st^N)$ now are sets of compositions $\s^{k}=(\s^{k}_1,\dots,\s^k_n)$ and $\st^k=(\st^{k}_1,\dots,\st^k_n)$, $k=1,\dots,N$. We also need to consider compositions $\s_{j}=(\s^{1}_j,\dots,\s^N_j)$ and $\st_j=(\st^{1}_j,\dots,\st^N_j)$, $j=1,\dots,n$. 
The number $\r_j$ tells us how many paths of colour $j$ touch the bottom edge of the column while $\rt_j$ tells us how many paths of colour $j$ touch the top edge of the column. 

Next, we need to impose the boundary conditions which allow the paths to reach the top boundary of a column and reappear at the bottom boundary in the same column and continue the up-right motion\footnote{This boundary condition does not apply to the paths of colour $i$ in the column $i$}. More precisely, we have the following: the $i$-th column will have $\r_j=0$ for $j< i$, $\r_i=\cm_{i}(\l)$ which means that all $i$-colour paths start at this column, and $\r_j\geq 0$ for $j> i$. On the other hand $\rt_j=0$ for all $j\leq i$ since the path of colour $i$ is not allowed to wind in the column with the same label and the paths of colours $j< i$ appear in the columns to the right from the column $i$.  The compositions $\s_j$ and $\st_j$ must satisfy: $|\st_j-\s_j|=0$ for $j> i$, $|\st_i|=\cm_{i}(\l)$ and $\s_j=\emptyset$ for $j\leq i$ and $\st_j=\emptyset$ for $j<i$, where $\emptyset$ is the composition with all entries equal to $0$. The assignment of these numbers to the edges is the same as before:
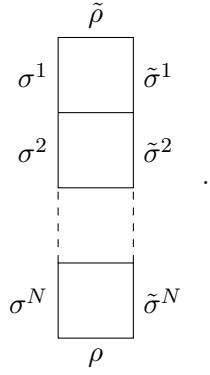
\begin{center}
\vspace{0.3cm}
\begin{tikzpicture}[scale=0.5,baseline=(current bounding box.center)]
\draw (0,0) -- (2,0) -- (2,2) -- (0,2) -- (0,0);
\draw (0,2+4)  -- (0,0+4) -- (2,0+4) -- (2,2+4) ;
\draw[dashed] (0,2) -- (0,2+2);
\draw[dashed] (2,2) -- (2,2+2);
\draw (0,0+6) -- (2,0+6) -- (2,2+6) -- (0,2+6) -- (0,0+6);
\node[left] at (0,1+6) {$\s^1$};
\node[right] at (2,1+6) {$\st^1$};
\node[left] at (0,1+4) {$\s^2$};
\node[right] at (2,1+4) {$\st^2$};
\node[left] at (0,1) {$\s^N$};
\node[right] at (2,1) {$\st^N$};
\node[below] at (1,0) {$\r$};
\node[above] at (1,8) {$\rt$};
\end{tikzpicture}~.
\vspace{0.3cm}
\end{center}
The values of the horizontal edges $\r^k_j$, where $k=1,\dots,N+1$ is counted from top to bottom, are fixed by $\s_j$ and $\st_j$ due to the continuity of the paths and given by 
\begin{align}\label{eq:rhosigj}
\r^k_j=\st^1_j+\dots +\st^{k-1}_j-(\s^1_j+\dots +\s^{k-1}_j)+s_j,
\end{align}
where $s_j$ are the indices which count the winding of the paths, so $s_j\geq 0$ for $j>i$ and $s_i=0$. 
Note that $\r^1_j$ for $j> i$ can be non zero since the paths with these colour labels can touch the top boundary and hence contribute to the value $\r^1_j$. For the values  $j\leq i$ we have $\r^1_j=0$.
Analogously, $\r^{N+1}_j$ can be non zero for $j> i$ and $\r^{N+1}_j=\r^{1}_j$, while for $j=i$ we have $\r^{N+1}_i=\cm_i(\l)$ and for $j<i$ we have $\r^{N+1}_j=0$.  

In view of (\ref{eq:rhosigj}) we introduce two sets $\n_j=(\n_j^1,\dots,\n_j^{N})$ and $\nt_j=(\nt_j^1,\dots,\nt_j^{N})$ which are defined by $\n_j^k=\s_j^1+\dots +\s_j^{k}$ and $\nt_j^k=\st_j^1+\dots +\st_j^{k}$, $j=1,\dots,n$. The sets of integers $\n_j$ and $\nt_j$ written in such order can be regarded as partitions for each $j$, and we define $\n_j^k=0$ and $\nt{}_j^k=0$ for $k\leq 0$. Note also that $\nt_j^k\geq \n_j^k$ only for $j=i$, while for $j<i$ this condition is broken due to the possibility for the paths to wind. 
With this notation we have $\s_j^{k}=\n_j^{k}-\n_j^{k-1}$, $\st_j^{k}=\nt_j^{k}-\nt_j^{k-1}$ and $\r_j^{k}=\nt_j^{k-1}-\n_j^{k-1}+s_j$.

The weight of the column $\MQ^{i}_{\n,\nt}(x;q,t)$ is given by $n-i$ infinite sums of the product of the weights of its faces $L^{\s,\st}_{\r,\rt}(x_k)$ in (\ref{eq:Lx}) while the weight of the column $\MQ'{}^{i}_{\n,\nt}(x;q,t)$ is given by $n-i$ infinite sums of the product of the weights of its faces $L'{}^{\s,\st}_{\r,\rt}(x_k)$ in (\ref{eq:Lz})
\begin{align}\label{eq:Qx0}
&\MQ^{i}_{\n,\nt}(x;q,t)
=
\sum_{s_{i+1},\dots,s_n}
q^{\sum_{j>i} (j-i) s_j}  
\prod_{k=1}^{N}
L_{\r^{k+1},\r^k}^{\s^k,\st^k}(x_k), \\
&
\label{eq:Qz0}
\MQ'{}^{i}_{\n,\nt}(x;q,t)=
\sum_{s_{i+1},\dots,s_n}
q^{\sum_{j>i} (j-i) s_j}
\prod_{k=1}^{N}
L'{}_{\r^{k+1},\r^k}^{\s^k,\st^k}(x_k).
\end{align}
The infinite summations occur due to windings of the paths. A path of colour $j$ appearing in column $i$, such that $i<j$, can wind infinitely many times which gives rise to the infinite number of possible configurations. Each infinite summation in $\MQ$ and $\MQ'$ can be expressed through the functions 
\begin{align*}
&\Phi_{\n_j|\nt_j}(z;t)
=
(z;t)_{\nt{}_j^{N}+1}
\sum_{s=0}^{\infty}
z^s
\prod_{k=1}^{N}
\binom{\nt_j^{k}-\n_j^{k-1}+s}{\nt_j^{k-1}-\n_j^{k-1}+s}_t, \\
&\Phi'_{\n_j|\nt_j}(z;t)
=
(z;t)_{\nt{}_j^{N}+1}
\sum_{s=0}^{\infty}
z^s
\prod_{k=1}^{N}
\binom{\nt_j^{k}-\n_j^{k}+s}{\nt_j^{k-1}-\n_j^{k}+s}_t,
\end{align*}
respectively. The function $\Phi'_{\n_j|\nt_j}(z;t)$ can be expressed through the function $\Phi_{\n_j|\nt_j}(z;t)$. To see this we note that the first $\n_j^1$ summands in $\Phi'$ are $0$ because the binomial factor with $k=0$ has a negative argument and thus equals to $0$. We introduce a new summation variable $s'=s+\n_j^1$. Defining the partition $r(\n_j)$, whose components are given by $r(\n_j)^k=\n_j^{k+1}-\n_j^1$  for $k<N$ and $r(\n_j)^N=\n_j^{N}$, we can write 
\begin{align}\label{eq:PhipPhi}
&\Phi'_{\n_j|\nt_j}(z;t)
=
z^{\n_j^1}\Phi_{r(\n_j)|\nt_j}(z;t).
\end{align}
The function $\Phi_{\n_j|\nt_j}(z;t)$ is in fact a finite polynomial in $z$ and $t$. It enjoys several interesting properties and can be written as a combinatorial sum (\ref{th:Fpos}) which reflects the positivity of its coefficients. We devote Section \ref{sec:phi} to the study of the polynomial $\Phi_{\n_j|\nt_j}(z;t)$. In order to write the final expression for the column weights in a uniform fashion we make use of the following identity:
\begin{align*}
\Phi_{\n_i|\nt_i}(1;t)=
\prod_{k=1}^{N-1} \binom{\nt_i^{k+1}}{\nt_i^{k}}_t.
\end{align*}
For a proof of this identity see Remark \ref{rmk:Phi1}. 

Next, we need to introduce two coefficients $\chi(\n,\nt)$ and $\chi'(\n,\nt)$:
\begin{align}\label{eq:chi1}
&\chi(\n,\nt)=
\sum_{k=1}^N
\sum_{j=i}^n
\Big{(}
\frac{1}{2}  (\nt_j^k-\nt_j^{k-1}) (\nt_j^k-\nt_j^{k-1}-1) + 
 \sum_{l>j} (\nt_j^k-\nt_j^{k-1})(\nt_l^{k}-\n_l^{k-1})
 \Big{)},\\
&\label{eq:chip1}
\chi'(\n,\nt)=
\sum_{k=1}^N
\sum_{i\leq j < l\leq n}
(\nt_j^k-\nt_j^{k-1})(\nt_l^{k-1}-\n_l^{k}),
\end{align}
and the normalization function:
\begin{align}\label{eq:NQ}
\MN^i_{\l}(q,t) =
\prod_{j=i+1}^n
(q^{j-i} t^{ \l'_i-\l'_j};t)_{\l'_j-\l'_{j+1}+1}.
\end{align}
Now we are ready to state the result of this subsection. The column weights can be expressed through $\Phi_{\n_j|\nt_j}(z;t)$ and $\Phi'_{\n_j|\nt_j}(z;t)$ as follows:
\begin{align}\label{eq:QQ2x}
&\MQ^{i}_{\n,\nt}(x;q,t)
=\frac{1}{\MN^i_{\l}(q,t) }
t^{\chi(\n,\nt)}
\prod_{j=i}^n
\Phi_{\n_j|\nt_j}(q^{j-i} 
t^{ \l'_i-\l'_j};t)
\times 
\prod_{k=1}^{N} 
x_k^{\sum_{j\geq i} (\nt_j^{k}-\nt_j^{k-1})}, \\
\label{eq:QQ2z}
&\MQ'^{i}_{\n,\nt}(x;q,t)
=\frac{1}{\MN^i_{\l}(q,t) }
t^{\chi'(\n,\nt)}
\prod_{j=i}^n
\Phi'_{\n_j|\nt_j}(q^{j-i} 
t^{ \l'_i-\l'_j};t)
\times 
\prod_{k=1}^{N} 
x_k^{\sum_{j\geq i} (\nt_j^{k}-\nt_j^{k-1})}.
\end{align}

In the rest of this subsection we show how to compute (\ref{eq:QQ2x}) and (\ref{eq:QQ2z}) starting from (\ref{eq:Qx0}) and (\ref{eq:Qz0}). Using the face weights (\ref{eq:Lx}) and (\ref{eq:Lz}) and taking into account the boundary conditions we can write the columns weights $\MQ^{i}_{\n,\nt}(x;q,t)$ and $\MQ'{}^{i}_{\n,\nt}(x;q,t)$ given in  (\ref{eq:Qx0}) and (\ref{eq:Qz0}) as 
\begin{align}\label{eq:Qx}
&\MQ^{i}_{\n,\nt}(x;q,t)
=
\sum_{s_{i+1},\dots,s_n}
q^{\sum_{j>i} (j-i) s_j}  
t^{ \sum_{k=1}^N \left( \frac{1}{2} |\st^k |(|\st^k|-1) + \sum_{i\leq l<j} \st_l^k \r^{k}_j\right) }
\prod_{k=1}^{N}
x_k^{|\st^k|}
\prod_{j=i}^{n}
\binom{\r_j^{k}+\st^{k}_j}{\r_j^{k}}_t, \\
&
\label{eq:Qz}
\MQ'{}^{i}_{\n,\nt}(x;q,t)=
\sum_{s_{i+1},\dots,s_n}
q^{\sum_{j>i} (j-i) s_j}
t^{ \sum_{k=1}^N  \sum_{i\leq l<j} \st_l^k (\r^{k}_j-\s^{k}_j)}
\prod_{k=1}^{N}
 x_k^{|\st^k|}
\prod_{j=i}^{n}
\binom{\r_j^{k+1}}{\st_j^{k}}_t,
\end{align}
where each summation over $s_{i+1},\dots,s_n$ runs from $0$ to $\infty$. Notice that the winding weight given by a power of $q$ corresponds to the $i$-th factor of $q^{\omega}$ with $\omega$ given in (\ref{eq:qomega}).
After rewriting the right hand sides of (\ref{eq:Qx})-(\ref{eq:Qz}) using the indices $\n$:
\begin{align*}
\s_j^{k}=\n_j^{k}-\n_j^{k-1},\qquad \st_j^{k}=\nt_j^{k}-\nt_j^{k-1},
\end{align*}
 and noticing that
the $n-i$ infinite summations are independent from each other and have the same form, the column weights become
\begin{align}
&\MQ^{i}_{\n,\nt}(x;q,t) 
=
t^{\chi(\n,\nt)}
\prod_{k=1}^{N-1} \binom{\nt_i^{k+1}}{\nt_i^{k}}_t
\times
\prod_{j>i}
\left(
\sum_{s}
(q^{j-i} 
t^{ \l'_i-\l'_j})^s
\prod_{k=1}^{N}
\binom{\nt_j^{k}-\n_j^{k-1}+s}{\nt_j^{k-1}-\n_j^{k-1}+s}_t
 x_k^{\nt_j^{k}-\nt_j^{k-1}}
 \right), \label{eq:Qx1}  \\
&
\MQ'{}^{i}_{\n,\nt}(x;q,t) 
=
t^{\chi'(\n,\nt)}
\prod_{k=1}^{N-1} 
\binom{\nt_i^{k+1}}{\nt_i^{k}}_t
\times
\prod_{j>i}
\left(
\sum_{s}
(q^{j-i} 
t^{ \l'_i-\l'_j})^s
\prod_{k=1}^{N}
\binom{\nt_j^{k}-\n_j^{k}+s}{\nt_j^{k-1}-\n_j^{k}+s}_t
 x_k^{\nt_j^{k}-\nt_j^{k-1}}
 \right). \label{eq:Qz1}
\end{align}
In (\ref{eq:Qx})-(\ref{eq:Qz}) the product over $j$ runs from $i$ to $n$, while in the above expressions (\ref{eq:Qx1}) and (\ref{eq:Qz1}) we wrote separately the factor with $j=i$. This factor equals $\prod_{k=1}^{N-1} 
\binom{\nt_i^{k+1}}{\nt_i^{k}}_t$; it has a simplified form due to the fact that in the column with label $i$ we have $\n_i^k=0$ (since $\s_i^k=0$). In the expressions (\ref{eq:Qx1}) and (\ref{eq:Qz1}) we also inserted the coefficients $\chi(\n,\nt)$ and $\chi'(\n,\nt)$, introduced in (\ref{eq:chi1})-(\ref{eq:chip1}). The appearance of $\chi(\n,\nt)$ is due to the exponent of $t$ in (\ref{eq:Qx}), more precisely we have:
\begin{align*}
&\frac{1}{2}\sum_{k=1}^N |\st^k |(|\st^k|-1) + \sum_{k=1}^N\sum_{i\leq l<j} \st_l^k \r^{k}_j
\\
=& \frac{1}{2}\sum_{k=1}^N |\nt^k-\nt^{k-1}| (|\nt^k-\nt^{k-1} |-1) 
 +  
\sum_{k=1}^N  \sum_{i\leq l<j} (\nt_l^k-\nt_l^{k-1})(\nt_j^{k-1}-\n_j^{k-1}+s_j) \\
 =& \chi(\n,\nt) +\sum_{k=1}^N \sum_{i\leq l<j} (\nt_l^k-\nt_l^{k-1})s_j 
=\chi(\n,\nt) +\sum_{i<j}(\l'_i-\l'_j)s_j,
\end{align*}
where in the last line the sum over $k$ telescopes and gives  $\nt_l^{N}= \cm_{l}(\l)=\l'_{l}-{\l'_{l+1}}$, and hence the sum over $l$ also telescopes. The coefficient $\chi'(\n,\nt)$ can be shown to appear analogously. In (\ref{eq:Qx1}) each factor of the product over $j>i$ equals the function $\Phi$ and in (\ref{eq:Qz1}) each factor of the product over $j>i$ equals the function $\Phi'$. This yields the two formulae (\ref{eq:QQ2x}) and (\ref{eq:QQ2z}).


\subsection{Lattice partition functions}\label{ssect:lattice}
Let now $\{\s\}$ be a set $\{\s\}=(\s_1,\dots,\s_n)$ with $\s_i=(\s_i^1,\dots,\s_i^N)$, $i=1,\dots, n$, where each $\s_i^k$ is a composition $\s_i^k=(\s_{i,1}^k,\dots,\s_{i,n}^k)$, $i=1,\dots, n$ and $k=1,\dots, N$. We also consider compositions $\s_{i,j}=(\s_{i,j}^1,\dots,\s_{i,j}^N)$, $i=1,\dots, n$ and $j=1,\dots, n$ and we set $\s_{i,j}^k=0$ for 
$i>j$ for all $k$.
On the lattice the indices of $\s$ are understood according to
\begin{align}
\label{sigma}
\s_{column,colour}^{row}= \s_{i,j}^k.
\end{align}
Let $\{\n\}$ be a set $\{\n\}=(\n_1,\dots,\n_n)$ with $\n_i=(\n_i^1,\dots,\n_i^N)$, $i=1,\dots, n$ where each $\n_i^k$ is a composition $\n_i^k=(\n_{i,1}^k,\dots,\n_{i,n}^k)$, $i=1,\dots, n$ and $k=1,\dots, N$. The numbers $\n_{i,j}^k$ are given by $\n_{i,j}^k=\s_{i,j}^1+\dots +\s_{i,j}^{k}$ and therefore $\n_{i,j}=(\n_{i,j}^1,\dots,\n_{i,j}^N)$ can be regarded as partitions.
In order to compute the partition function $\MZ_\l(x;q,t)$ with the positions of the paths at the lower edges given by the partition $\l$ we need to join together $n$ columns $\MQ^{i}_{\n_{i+1},\n_{i}}(x;q,t)$ and sum over all possible arrangements of the paths on the vertical edges, i.e. sum over all compatible $\n_{1},\dots,\n_{n}$:
\begin{align}\label{eq:Z}
&\MZ_\l(x;q,t):=\MN_{\l}(q,t) \sum_{\{\n\}}\prod_{i=1}^n \MQ^{i}_{\n_{i+1},\n_i}(x;q,t),\\
&
\label{eq:Zp}
\MZ'_\l(x;q,t):=\MN_{\l}(q,t)\sum_{\{\n\}}\prod_{i=1}^n \MQ'{}^{i}_{\n_{i+1},\n_i}(x;q,t).
\end{align}
The overall factor
\begin{align}\label{eq:NZ}
\MN_{\l}(q,t) = \prod_{i=1}^n \MN^i_{\l}(q,t)=
\prod_{1\leq i<j\leq n}
\prod_{\l'_{j+1}\leq \ell\leq \l'_j}(1-q^{j-i} t^{\l'_i-\ell})
\end{align}
is introduced in order to clear the denominators, which makes both $\MZ_\l(x;q,t)$ and $\MZ'_\l(x;q,t)$ polynomials in $(x;q,t)$. By inserting (\ref{eq:QQ2x}) in (\ref{eq:Z}) and (\ref{eq:QQ2z}) in (\ref{eq:Zp}) we can write the  partition functions as the following sums of products of the column weights:
\begin{align}\label{eq:Z1}
&\MZ_\l(x;q,t)=
\sum_{\{\n\}}
t^{\chi(\n)}
\prod_{j\geq i}
\Phi_{\n_{i+1,j}|\n_{i,j}}(q^{j-i} 
t^{ \l'_i-\l'_j};t)
\times 
\prod_{k=1}^{N} 
x_k^{\sum_{j\geq i} (\n_{i,j}^{k}-\n_{i,j}^{k-1})}, \\
\label{eq:Z1p}
&\MZ'_\l(x;q,t)=
\sum_{\{\n\}}
t^{\chi'(\n)}
\prod_{j\geq i}
\Phi'_{\n_{i+1,j}|\n_{i,j}}(q^{j-i} 
t^{ \l'_i-\l'_j};t)
\times 
\prod_{k=1}^{N} 
x_k^{\sum_{j\geq i} (\n_{i,j}^{k}-\n_{i,j}^{k-1})}.
\end{align}
The summation set $\{\n\}$ is given by $n(n+1)/2$ partitions $\n_{i,j}=(\n_{i,j}^1,\dots,\n_{i,j}^N)$  whose top parts are $\n_{i,j}^{N}=\l'_j-\l'_{j+1}$. The new exponents $\chi(\n)$ and $\chi'(\n)$ are given by
\begin{align}
\label{eq:chi}
&\chi(\n)=
\sum_{i=1}^n
\chi(\n_{i+1},\n_i)=\sum_{k=1}^N  \sum_{1\leq i\leq j\leq n}
\Big{(}
\frac{1}{2} (\n_{i,j}^k-\n_{i,j}^{k-1})(\n_{i,j}^k-\n_{i,j}^{k-1} -1) + 
 \sum_{l>j} (\n_{i,j}^k-\n_{i,j}^{k-1})(\n_{i,l}^{k}-\n_{i+1,l}^{k-1})
 \Big{)},\\
 \label{eq:chip}
&\chi'(\n)=
\sum_{i=1}^n\chi'(\n_{i+1},\n_i)=\sum_{k=1}^N
\sum_{1\leq i\leq j <l\leq n}
 (\n_{i,j}^k-\n_{i,j}^{k-1})(\n_{i,l}^{k-1}-\n_{i+1,l}^{k}),
\end{align}
where $\chi(\n_{i+1},\n_i)$ and $\chi'(\n_{i+1},\n_i)$ are given in (\ref{eq:chi1}) and (\ref{eq:chip1}).
Let the integer matrices $\n^{k}$ be chosen such that $\sum_{j\geq i} (\n_{i,j}^{k}-\n_{i,j}^{k-1})=\m_k$. We can write our partition functions by summing over all possible $\m$ such that $|\m|=|\l|$ and finally over $\m$-restricted sets $\{\n\}$:
\begin{align}\label{eq:Z2}
&\MZ_\l(x;q,t)=
\sum_{\m: |\m|=|\l|}
x_1^{\m_1}\dots x_N^{\m_N}
\sum_{\{\n\}}
t^{\chi(\n)}
\prod_{j\geq i}
\Phi_{\n_{i+1,j}|\n_{i,j}}(q^{j-i} 
t^{ \l'_i-\l'_j};t),\\
&\label{eq:Z2p}
\MZ'_\l(x;q,t)=
\sum_{\m: |\m|=|\l|}
x_1^{\m_1}\dots x_N^{\m_N}
\sum_{\{\n\}}
t^{\chi'(\n)}
\prod_{j\geq i}
\Phi'_{\n_{i+1,j}|\n_{i,j}}(q^{j-i} 
t^{ \l'_i-\l'_j};t).
\end{align}
The next crucial step follows from the fact that the coefficients
\begin{align*}
&\sum_{\{\n\}}
t^{\chi(\n)}
\prod_{j\geq i}
\Phi_{\n_{i+1,j}|\n_{i,j}}(q^{j-i} 
t^{ \l'_i-\l'_j};t),\\
&
\sum_{\{\n\}}
t^{\chi'(\n)}
\prod_{j\geq i}
\Phi'_{\n_{i+1,j}|\n_{i,j}}(q^{j-i} 
t^{ \l'_i-\l'_j};t),
\end{align*}
in which the summations run over $\m$-restricted sets $\{\n\}$, are both invariant under permutations of the parts of $\m$. This follows from the fact that $\MZ_\l$ and $\MZ'_\l$ are symmetric polynomials in $x_i$. This is explained in the construction in Section \ref{sec:proof}. 
Now let the summation over $\m: |\m|=|\l|$ be split into two sub-summations: one over partitions $\m$ satisfying $|\m|=|\l|$ and the second over permutations of $\m$. Since the above coefficients are invariant under permutations of $\m$ we can compute the second of these two sums and find the monomial symmetric function $m_\m(x)$:
\begin{align}\label{eq:Z3}
&\MZ_\l(x;q,t)=
\sum_{\m}
m_\m(x)
\sum_{\{\n\}}
t^{\chi(\n)}
\prod_{j\geq i}
\Phi_{\n_{i+1,j}|\n_{i,j}}(q^{j-i} 
t^{ \l'_i-\l'_j};t),\\
\label{eq:Z3}
&\MZ'_\l(x;q,t)=
\sum_{\m}
m_\m(x)
\sum_{\{\n\}}
t^{\chi'(\n)}
\prod_{j\geq i}
\Phi'_{\n_{i+1,j}|\n_{i,j}}(q^{j-i} 
t^{ \l'_i-\l'_j};t).
\end{align}
The summation set $\{\n\}$ can be described as the set running over partitions $\n_{i,j}$ satisfying the conditions:
\begin{align}
\label{eq:constr1}
\n_{i,j}^{N}
=\l'_j-\l'_{j+1},
&
\qquad
\forall\ 1\leq i \leq j \leq n,
\\
\label{eq:constr2}
\n_{i+1,i}^{k}
=
0,
&
\qquad
\forall\ 1 \leq i \leq n,\ \ 1 \leq k \leq N,
\\
\label{eq:constr3}
\sum_{1 \leq i \leq j \leq n} \n_{i,j}^{k}
=
\m_1+\dots+\m_k,
&
\qquad
\forall\ 1\leq k \leq N.
\end{align}
Another relevant summation set is the set $\{\n\}'$ which differs from $\{\n\}$ only up to $\l\leftrightarrow \l'$. This set contains partitions $\n_{i,j}$ with $1\leq i\leq j\leq \l'_1$ satisfying
\begin{align}
\label{eq:dconstr1}
\n_{i,j}^{N}
=\l_j-\l_{j+1},
&
\qquad
\forall\ 1\leq i \leq j \leq \l'_1,
\\
\label{eq:dconstr2}
\n_{i+1,i}^{k}
=
0,
&
\qquad
\forall\ 1 \leq i \leq \l'_1,\ \ 1 \leq k \leq N,
\\
\label{eq:dconstr3}
\sum_{1 \leq i \leq j \leq \l'} \n_{i,j}^{k}
=
\m_1+\dots+\m_k,
&
\qquad
\forall\ 1\leq k \leq N.
\end{align}

Let us recall the face weights $\ML^{\s,\st}_{\r,\rt}(x,z)$ in (\ref{eq:Lone}). 
In Section \ref{sec:proof}, Proposition \ref{prop:WZ} we demonstrate that the above lattice construction implemented with the face weights $\ML^{\s,\st}_{\r,\rt}(x,z)$ given in (\ref{eq:Lone}) leads to a partition function $\MZ_{\l}(x;q,t;z)$ for which we have
\begin{align}\label{eq:ZW}
\MZ_{\l}(x;q,t;z)= W_{\l}(x;q,t;z). 
\end{align}
The partition functions $\MZ_{\l}(x;q,t)$ and $\MZ'_{\l}(x;q,t)$ are special cases of $\MZ_{\l}(x;q,t;z)$: 
\begin{align}
\label{eq:Zx0}
&\MZ_{\l}(x;q,t)=\MZ_{\l}(x;q,t;0), \\
\label{eq:Zz0}
&\MZ'_{\l}(x;q,t)=\MZ_{\l}(0;q,t;x), 
\end{align}
since they are constructed with the local weights (\ref{eq:MLx}) and (\ref{eq:MLz}), respectively. Using (\ref{eq:ZW}), the two reductions (\ref{eq:Zx0}) and (\ref{eq:Zz0}) and the two relations between $W_\l$ and $H_\l$ in (\ref{eq:WredH1}) and (\ref{eq:WredH2}) allows us to obtain two formulae for the modified Macdonald polynomials $H_\l(x;q,t)$.
\begin{thm}\label{thm:H}
Let $\{\n\}$ be a set of partitions $\n_{i,j}$, $1\leq i\leq j\leq n=\l_1$ of length $N$ satisfying (\ref{eq:constr1})-(\ref{eq:constr3}) and $\{\n\}'$ be a set of partitions $\n_{i,j}$, $1\leq i\leq j\leq \l_1'$ of length $N$ satisfying (\ref{eq:dconstr1})-(\ref{eq:dconstr3}). We have
\begin{align}\label{eq:HZx}
&H_\l(x;q,t)=
\sum_{\m}
m_\m(x)
\sum_{\{\n\}}
t^{\chi(\n)}
\prod_{j\geq i}
\Phi_{\n_{i+1,j}|\n_{i,j}}(q^{j-i} 
t^{ \l'_i-\l'_j};t),\\
\label{eq:HZz}
&
H_\l(x;q,t)=
\sum_{\m}
m_\m(x)
\sum_{\{\n\}'}
q^{\chi'(\n)}
\prod_{j\geq i}
\Phi'_{\n_{i+1,j}|\n_{i,j}}(t^{j-i} 
q^{ \l_i-\l_j};q).
\end{align}
\end{thm}
The formula (\ref{eq:HmonP-intro}) which is stated in the introduction corresponds to (\ref{eq:HZx}). By (\ref{eq:HZx}) and (\ref{eq:HZz}) we find two formulae for the monomial coefficients 
\begin{align}\label{eq:Hcoef}
&\MP_{\l,\m}(q,t)=
\sum_{\{\n\}}
t^{\chi(\n)}
\prod_{j\geq i}
\Phi_{\n_{i+1,j}|\n_{i,j}}(q^{j-i} 
t^{ \l'_i-\l'_j};t),\\
\label{eq:Hcoefp}
&\MP_{\l,\m}(q,t)=
\sum_{\{\n\}'}
q^{\chi'(\n)}
\prod_{j\geq i}
\Phi'_{\n_{i+1,j}|\n_{i,j}}(t^{j-i} 
q^{ \l_i-\l_j};q).
\end{align}
The two constructions which lead to these formulae have very different face weights and their duality is not obvious. For example, from (\ref{eq:HinvH}) we can write 
\begin{align}\label{eq:HHp}
&\MP_{\l,\m}(q,t)=
q^{n(\l)}t^{n(\l')}
\MP_{\l',\m}(t^{-1},q^{-1}),
\end{align}
which, when used in (\ref{eq:Hcoefp}) and then matching it to (\ref{eq:Hcoef}), leads to the identity
\begin{align}\label{eq:HHp}
&\sum_{\{\n\}}
t^{\chi(\n)}
\prod_{j\geq i}
\Phi_{\n_{i+1,j}|\n_{i,j}}(q^{j-i} 
t^{ \l'_i-\l'_j};t)=
q^{n(\l)}t^{n(\l')}
\sum_{\{\n\}}
t^{-\chi'(\n)}
\prod_{j\geq i}
\Phi'_{\n_{i+1,j}|\n_{i,j}}(q^{i-j} 
t^{\l'_j-\l'_i};t^{-1}).
\end{align}
The summands, however, are not equal termwise, and the equality is not readily deduced.

\subsection{Example}
We consider an example of the combinatorial formula (\ref{eq:HZx}) in detail for $\l=(2,2,1)$ in two variables $(x_1,x_2)$. The monomial expansion reads
\begin{align*}
H_{(2,2,1)}(x_1,x_2)
&=
\left(q^2 t^4+q t^4+2 q t^3+q t^2+t^4+t^3+2 t^2+t\right) m_{(3,2)}\\
&+\left(q t^4+q t^3+t^4+t^3+t^2\right) m_{(4,1)}
+t^4 m_{(5,0)},
\end{align*}
where $m_\m=m_\m(x_1,x_2)$. 
Let us see how (\ref{eq:HZx}) recovers this expression.  We need to compute three coefficients  of the monomial functions
\begin{align*}
H_{(2,2,1)}(x_1,x_2)
=
&m_{(3,2)} \MP_{(2,2,1),(3,2)}(q,t)
+
m_{(4,1)} \MP_{(2,2,1),(4,1)}(q,t)
+
m_{(5,0)} \MP_{(2,2,1),(5,0)}(q,t).
\end{align*}
For each coefficient $\MP_{\l,\m}(q,t)$ we compute the summations sets $\{\n\}$ which depend on $\l=(2,2,1)$ and the three different partitions $\m$.
The summation set $\{\n\}$, described in (\ref{eq:constr1})-(\ref{eq:constr3}), has elements which are triples of length-two partitions: 
\begin{align*}
\n_{1,1}=(\n_{1,1}^1,\n_{1,1}^2),\qquad
\n_{1,2}=(\n_{1,2}^1,\n_{1,2}^2),\qquad
\n_{2,2}=(\n_{2,2}^1,\n_{2,2}^2).
\end{align*}
Using (\ref{eq:constr1}) with $\l'=(3,2)$ we compute the top terms $\n_{i,j}^2$ of these partitions
\begin{align}
\label{eq221:cond1}
\n_{1,1}^2=1,
\qquad
\n_{1,2}^2=2,
\qquad
\n_{2,2}^2=2.
\end{align}
where $\l'_3$ is set to $0$. Next, for each of the three coefficients we need to impose the restriction (\ref{eq:constr3}) which will give the remaining parts $\n_{i,j}^1$. 
\begin{enumerate}
\item Compute $\MP_{(2,2,1),(3,2)}(q,t)$. We have $\m=(3,2)$ and the condition (\ref{eq:constr3}) gives
\begin{align*}
\n_{1,1}^1+\n_{1,2}^1+\n_{2,2}^1=3.
\end{align*}
Given the bounds (\ref{eq221:cond1}) for the top parts we get five different elements in the set $\{\n\}$. Below we also give the corresponding lattice path configurations\footnote{In order to match the data given by the partitions $\n_{i,j}$ and the lattice paths recall (\ref{sigma}) and the relation between the indices $\s$ and $\n$.}
\begin{align*}
&
\MC_1=\left(
\begin{array}{c}
 \nu_{1,1} =(0,1) \\
 \nu_{1,2}=(1,2),~~ \nu_{2,2}=(2,2) \\
\end{array}
\right)
=
\begin{tikzpicture}[scale=0.6,baseline=(current bounding box.center)]
		\foreach\x in {0,...,2}{
		\draw (0,\x) -- (2,\x);
		\draw (\x,0) -- (\x,2);						}
 \begin{scope}[line width=1pt]
\draw[blue] (0.4,0) -- (0.4,1.6) -- (1.4, 1.6) -- (1.4, 2);
\draw[blue] (1.4,0) -- (1.4,0.6) -- (2, 0.6);
\draw[blue] (0.6,0) -- (0.6,1.4) -- (2, 1.4) ;
\draw[red] (1.6,0) -- (1.6,0.4) -- (2, 0.4) ;
\end{scope}
\end{tikzpicture}~,\\
&
\MC_2=
\left(
\begin{array}{c}
 \nu_{1,1} =(0,1) \\
 \nu_{1,2}=(2,2),~~
 \nu_{2,2}=(1,2) \\
\end{array}
\right)
=
\begin{tikzpicture}[scale=0.6,baseline=(current bounding box.center)]
		\foreach\x in {0,...,2}{
		\draw (0,\x) -- (2,\x);
		\draw (\x,0) -- (\x,2);						}
 \begin{scope}[line width=1pt]
\draw[blue] (0.4,0) -- (0.4,1.6) -- (2, 1.6) ;
\draw[blue] (0.6,0) -- (0.6,0.6) -- (1.4, 0.6) -- (1.4,1.4) -- (2,1.4) ;
\draw[red] (1.6,0) -- (1.6,0.4) -- (2, 0.4) ;
\end{scope}
\end{tikzpicture}~,\\
&
\MC_3= \left(
\begin{array}{c}
 \nu_{1,1} =(1,1) \\
 \nu_{1,2}=(0,2),~~
 \nu_{2,2}=(2,2) \\
\end{array}
\right)
=
\begin{tikzpicture}[scale=0.6,baseline=(current bounding box.center)]
		\foreach\x in {0,...,2}{
		\draw (0,\x) -- (2,\x);
		\draw (\x,0) -- (\x,2);						}
 \begin{scope}[line width=1pt]
\draw[blue] (0.4,0) -- (0.4,1.6) -- (1.4, 1.6) -- (1.4, 2);
\draw[blue] (1.4,0) -- (1.4,0.6) -- (2, 0.6);
\draw[blue] (0.6,0) -- (0.6,1.4) -- (1.5, 1.4) -- (1.5,2);
\draw[blue] (1.5,0) -- (1.5,0.5) -- (2, 0.5);
\draw[red] (1.6,0) -- (1.6,1.4) -- (2, 1.4) ;
\end{scope}
\end{tikzpicture}~,\\
&
\MC_4=
\left(
\begin{array}{c}
 \nu_{1,1} =(1,1) \\
 \nu_{1,2}=(1,2),~~ 
 \nu_{2,2}=(1,2) \\
\end{array}
\right)
=
\begin{tikzpicture}[scale=0.6,baseline=(current bounding box.center)]
		\foreach\x in {0,...,2}{
		\draw (0,\x) -- (2,\x);
		\draw (\x,0) -- (\x,2);						}
 \begin{scope}[line width=1pt]
\draw[blue] (0.4,0) -- (0.4,1.6) -- (2, 1.6) ;
\draw[blue] (0.6,0) -- (0.6,0.5) -- (2, 0.5) ;
\draw[red] (1.6,0) -- (1.6,1.4) -- (2, 1.4) ;
\end{scope}
\end{tikzpicture}~,\\
&
\MC_5=
\left(
\begin{array}{c}
 \nu_{1,1} =(1,1) \\
 \nu_{1,2}=(2,2), ~~\nu_{2,2} =(0,2) \\
\end{array}
\right)=
\begin{tikzpicture}[scale=0.6,baseline=(current bounding box.center)]
		\foreach\x in {0,...,2}{
		\draw (0,\x) -- (2,\x);
		\draw (\x,0) -- (\x,2);						}
 \begin{scope}[line width=1pt]
\draw[blue] (0.4,0) -- (0.4,0.6) -- (1.4, 0.6) -- (1.4, 1.6) -- (2, 1.6)  ;
\draw[blue] (0.5,0) -- (0.5,0.5) -- (1.5, 0.5) -- (1.5, 1.5) -- (2, 1.5);
\draw[red]  (1.6, 0) -- (1.6, 1.4) -- (2, 1.4);
\end{scope}
\end{tikzpicture}~.
\end{align*}
The coefficient is given by the sum of the five corresponding terms which we write in the above order:
\begin{align*}
\MP_{(2,2,1),(3,2)}(q,t)
&=
t^{\chi(\MC_1)}\times 
\Phi_{(0,0)|(0,1)}(1;t) ~ \Phi_{(0,0)|(2,2)}(1;t) ~\Phi_{(2,2)|(1,2)}(q t;t)
{\dg{~=t\times q t (1+t)}}\\
&+
t^{\chi(\MC_2)}\times  
\Phi_{(0,0)|(0,1)}(1;t) ~\Phi_{(0,0)|(1,2)}(1;t) ~\Phi_{(1,2)|(2,2)}(q t;t)
{\dg{~=t^2\times (1+t)}}\\
&+
t^{\chi(\MC_3)}\times  \Phi_{(0,0)|(1,1)}(1;t) ~\Phi_{(0,0)|(2,2)}(1;t)~ \Phi_{(2,2)|(0,2)}(q t;t)
{\dg{~=t^2\times q^2 t^2 }}\\
&+
t^{\chi(\MC_4)}\times  \Phi_{(0,0)|(1,1)}(1;t) ~\Phi_{(0,0)|(1,2)}(1;t)~ \Phi_{(1,2)|(1,2)}(q t;t)
{\dg{~=t\times (1+t)(1+q t^2) }}\\
&+
t^{\chi(\MC_5)}\times  \Phi_{(0,0)|(0,2)}(1;t) ~\Phi_{(0,0)|(1,1)}(1;t)~ \Phi_{(0,2)|(2,2)}(q t;t)
{\dg{~=t^4\times 1}},
\end{align*}
where in green we give the explicit values of each term. 
The powers of $t$ are computed with (\ref{eq:chi}) and the values of different polynomials $\Phi$ are
\begin{align*}
&\Phi _{(0,0)|(0,1)}(1)=\Phi _{(0,0)|(0,2)}(1)=\Phi _{(0,0)|(1,1)}(1)=1,\\
&
\Phi _{(0,0)|(2,2)}(1)=\Phi _{(0,2)|(2,2)}(q t)=\Phi _{(1,2)|(2,2)}(q t)=1,\\
&
\Phi_{(0,0)|(1,2)}(1)=1+t,\qquad 
   \Phi _{(1,2)|(1,2)}(q t)=1+q t^2,\\
   &
   \Phi _{(2,2)|(0,2)}(q t)=q^2 t^2,
   \qquad
   \Phi _{(2,2)|(1,2)}(q t)=q t (1+t).
\end{align*}
These polynomials can be computed by either of the three formulae (\ref{eq:FF}), (\ref{eq:Fpos}) or (\ref{eq:FFpf}). 
\item Compute $\MP_{(2,2,1),(4,1)}(q,t)$. We have $\m=(4,1)$ and the condition (\ref{eq:constr3}) gives
\begin{align*}
\n_{1,1}^1+\n_{1,2}^1+\n_{2,2}^1=4.
\end{align*}
Given the bounds (\ref{eq221:cond1}) for the top parts we get three different elements in the set $\{\n\}$
\begin{align*}
&
\MC_1=\left(
\begin{array}{c}
 \nu_{1,1} =(0,1) \\
 \nu_{1,2}=(2,2),~~ \nu_{2,2}=(2,2) \\
\end{array}
\right)=
\begin{tikzpicture}[scale=0.6,baseline=(current bounding box.center)]
		\foreach\x in {0,...,2}{
		\draw (0,\x) -- (2,\x);
		\draw (\x,0) -- (\x,2);						}
 \begin{scope}[line width=1pt]
\draw[blue] (0.4,0) -- (0.4,1.6) -- (2, 1.6) ;
\draw[blue] (0.6,0) -- (0.6,1.4) -- (1.4,1.4) -- (2,1.4) ;
\draw[red] (1.6,0) -- (1.6,0.4) -- (2, 0.4) ;
\end{scope}
\end{tikzpicture}~,\\
&
\MC_2=
\left(
\begin{array}{c}
 \nu_{1,1} =(1,1) \\
 \nu_{1,2}=(1,2),~~
 \nu_{2,2}=(2,2) \\
\end{array}
\right)=
\begin{tikzpicture}[scale=0.6,baseline=(current bounding box.center)]
		\foreach\x in {0,...,2}{
		\draw (0,\x) -- (2,\x);
		\draw (\x,0) -- (\x,2);						}
 \begin{scope}[line width=1pt]
\draw[blue] (0.4,0) -- (0.4,1.6) -- (1.4, 1.6) -- (1.4, 2);
\draw[blue] (1.4,0) -- (1.4,0.6) -- (2, 0.6);
\draw[blue] (0.6,0) -- (0.6,1.5) -- (2, 1.5) ;
\draw[red] (1.6,0) -- (1.6,1.4) -- (2, 1.4) ;
\end{scope}
\end{tikzpicture}~,\\
&
\MC_3= \left(
\begin{array}{c}
 \nu_{1,1} =(1,1) \\
 \nu_{1,2}=(2,2),~~
 \nu_{2,2}=(1,2)
\end{array}
\right)
=
\begin{tikzpicture}[scale=0.6,baseline=(current bounding box.center)]
		\foreach\x in {0,...,2}{
		\draw (0,\x) -- (2,\x);
		\draw (\x,0) -- (\x,2);						}
 \begin{scope}[line width=1pt]
\draw[blue] (0.4,0) -- (0.4,1.6) -- (1.4, 1.6) -- (2, 1.6)  ;
\draw[blue] (0.5,0) -- (0.5,0.5) -- (1.5, 0.5) -- (1.5, 1.5) -- (2, 1.5);
\draw[red]  (1.6, 0) -- (1.6, 1.4) -- (2, 1.4);
\end{scope}
\end{tikzpicture}~.
\end{align*}
The coefficient is given by the sum of the three corresponding terms which we write in the above order:
\begin{align*}
\MP_{(2,2,1),(4,1)}(q,t)
&=
t^{\chi(\MC_1)}\times 
\Phi_{(0,0)|(0,1)}(1;t) ~ \Phi_{(0,0)|(2,2)}(1;t) ~\Phi_{(2,2)|(2,2)}(q t;t)
{\dg{~=t^2 \times 1}}\\
&+
t^{\chi(\MC_2)}\times  
\Phi_{(0,0)|(1,1)}(1;t) ~\Phi_{(0,0)|(2,2)}(1;t) ~\Phi_{(2,2)|(1,2)}(q t;t)
{\dg{~=t^2\times q t(1+t)}}\\
&+
t^{\chi(\MC_3)}\times  \Phi_{(0,0)|(1,1)}(1;t) ~\Phi_{(0,0)|(1,2)}(1;t)~ \Phi_{(1,2)|(2,2)}(q t;t)
{\dg{~=t^3\times (1+t) }}.
\end{align*}
All polynomials $\Phi$ appearing above are known from the previous computation except from $\Phi_{(2,2)|(2,2)}(q t;t)=1$. 
\item Compute $\MP_{(2,2,1),(5,0)}(q,t)$. We have $\m=(5,0)$ and the condition (\ref{eq:constr3}) gives
\begin{align*}
\n_{1,1}^1+\n_{1,2}^1+\n_{2,2}^1=5.
\end{align*}
Given the bounds (\ref{eq221:cond1}) for the top parts we get one element in the set $\{\n\}$
\begin{align*}
\MC_1=\left(
\begin{array}{c}
 \nu_{1,1} =(1,1) \\
 \nu_{1,2}=(2,2),~~ \nu_{2,2}=(2,2) \\
\end{array}
\right)
=
\begin{tikzpicture}[scale=0.6,baseline=(current bounding box.center)]
		\foreach\x in {0,...,2}{
		\draw (0,\x) -- (2,\x);
		\draw (\x,0) -- (\x,2);						}
 \begin{scope}[line width=1pt]
\draw[blue] (0.4,0) -- (0.4,1.6)  -- (2, 1.6)  ;
\draw[blue] (0.5,0) -- (0.5,1.5) -- (2, 1.5);
\draw[red]  (1.6, 0) -- (1.6, 1.4) -- (2, 1.4);
\end{scope}
\end{tikzpicture}~.
\end{align*}
The coefficient is given by a single term:
\begin{align*}
\MP_{(2,2,1),(5,0)}(q,t)
&=
t^{\chi(\MC_1)}\times 
\Phi_{(0,0)|(1,1)}(1;t) ~ \Phi_{(0,0)|(2,2)}(1;t) ~\Phi_{(2,2)|(2,2)}(q t;t)
{\dg{~=t^4 \times 1}}.
\end{align*}
\end{enumerate}


\section{Proof of the construction of the modified Macdonald polynomials}\label{sec:proof}\label{sec:proof}

The purpose of this section is to prove one of the central relations of the paper  (\ref{eq:ZW}), which equates the combinatorial quantity $\MZ_{\l}(x;q,t;z)$ on the one hand with the family of polynomials $W_{\l}(x;q,t;z)$ on the other.

The proof goes through several stages. In the first step, we recall the matrix product construction of (ordinary) Macdonald polynomials, obtained in \cite{CantinidGW}. The construction of \cite{CantinidGW} makes use of a family of infinite dimensional matrices which we term {\it row operators}, which satisfy a set of exchange relations known as the {\it Faddeev--Zamolodchikov} algebra. By taking a trace over appropriate products of such operators, one obtains a combinatorial expression for a family of non-symmetric functions $f_{\l}$, which we refer to as {\it ASEP polynomials}. These polynomials are related to the celebrated {\it non-symmetric Macdonald polynomials} $E_{\l}$ \cite{Opdam,Cher95a,Cher95b}, and just like the latter, by symmetrizing appropriately over the family $f_{\l}$ one recovers the symmetric Macdonald polynomials. We will review these facts in Sections \ref{ssec:Hecke}--\ref{ssec:MPACdGW}. In the second stage, we reformulate the previous matrix product in terms of operators which we call {\it column operators}. 
In the third and final step, we apply the evaluation homomorphism (\ref{eq:ev}) to the previously obtained matrix product expressions. As we have already seen in (\ref{eq:evW}), this converts symmetric Macdonald polynomials into the polynomials $W_{\l}(x;q,t;z)$, and thus it leads us directly to the right hand side of (\ref{eq:ZW}). The evaluation homomorphism can be expressed via the principal specializations which at the level of the matrix product turn out to be equivalent to the well-known fusion technique of integrable lattice models \cite{KulishRS,KirillovR}. In Sections \ref{ssec:Fusion}--\ref{ssec:rec} we perform this fusion and calculate the matrix elements (\ref{eq:Lone}), 
leading ultimately to the combinatorial quantity $\MZ_{\l}(x;q,t;z)$, the desired left hand side of (\ref{eq:ZW}).

\subsection{Hecke algebra and its polynomial representation}\label{ssec:Hecke}

Let us begin by recalling the Hecke algebra of type $A_{N-1}$. It has generators $\{T_i\}_{1 \leq i \leq N-1}$ satisfying the relations
\begin{equation}
\begin{aligned}
\label{eq:hecke}
(T_i-t)(T_i+1)=0,& \qquad T_iT_{i+1}T_i=T_{i+1}T_iT_{i+1},
\\
T_iT_j = T_j T_i,& \qquad \forall\ i,j \ \text{such that}\ \ |i-j| > 1.
\end{aligned}
\end{equation}
A common realization of this algebra is on the ring of polynomials $\mathbb{C}[x_1,\dots,x_N]$, and is known as the {\it polynomial representation} of (\ref{eq:hecke}). Namely, if one defines
\begin{align}
\label{eq:T_i}
T_i
:=
t
-
\left(
\frac{t x_i - x_{i+1}}{x_i - x_{i+1}}
\right)
(1-s_i),
\qquad
1 \leq i \leq N-1,
\end{align}
where $s_i$ denotes the transposition operator which exchanges the variables $x_i$ and $x_{i+1}$, then one can verify that each of the relations (\ref{eq:hecke}) hold. Furthermore, $T_i$ acts injectively on $\mathbb{C}[x_1,\dots,x_N]$, in the sense that it preserves polynomiality.

The generators (\ref{eq:T_i}) are invertible, and have inverses which closely resemble the original operators:
\begin{align*}
T_i^{-1}
=
t^{-1}
\left[
1
-
\left(
\frac{t x_i - x_{i+1}}{x_i - x_{i+1}}
\right)
(1-s_i)
\right],
\qquad
1 \leq i \leq N-1.
\end{align*}
In \cite{CantinidGW} a different convention was used for the Hecke algebra relations. The two conventions can be matched by replacing $T_i\rightarrow t^{1/2} T_i$ everywhere above.

\subsection{Non-symmetric Macdonald polynomials}

Throughout what follows, let $\mathbb{C}_{q,t}[x_1,\dots,x_N] := \mathbb{C}[x_1,\dots,x_N] \otimes \mathbb{Q}(q,t)$. Let us extend the Hecke algebra (\ref{eq:hecke}) by a generator $\omega$ which acts cyclically on polynomials: 
\begin{align}
\label{eq:omega}
(\omega g)(x_1,\ldots,x_N) := g(qx_N,x_1,\ldots,x_{N-1}),
\end{align}
where $g \in \mathbb{C}_{q,t}[x_1,\dots,x_N]$ denotes an arbitrary polynomial. The resulting algebraic structure is the affine Hecke algebra of type $A_{N-1}$. It has an Abelian subalgebra generated by the Cherednik--Dunkl operators $\{Y_i\}_{1 \leq i \leq N}$, where
\begin{align}
\label{eq:Yi}
Y_i := T_i\cdots T_{N-1} \omega T_{1}^{-1} \cdots T_{i-1}^{-1},
\qquad
1 \leq i \leq N.
\end{align}
These operators mutually commute and can be jointly diagonalized. The non-symmetric Macdonald polynomials $E_{\l} \equiv E_{\l}(x_1,\dots,x_N;q,t)$, indexed by compositions $\l \in \mathbb{N}^N$, are the unique family of polynomials which satisfy the properties
\begin{align}
\label{eq:monic}
E_{\l} &= x^{\l} + \sum_{\nu \not= \l} c_{\l,\nu}(q,t) x^{\nu},
\quad
x^{\l} := \prod_{i=1}^{N} x_i^{\l_i},
\quad
c_{\l,\nu}(q,t) \in \mathbb{Q}(q,t),
\\
Y_i E_{\l} &= y_i(\l;q,t)E_{\l},
\quad
\forall\ 1 \leq i \leq N,
\label{eq:eigYi}
\end{align}
with eigenvalues given by
\begin{align}
\label{rhomu}
y_i(\l;q,t)= q^{\l_i} t^{\rho_i(\l)},
\end{align}
where $\r_i(\l)$ is defined as follows. For a composition $\l=(\l_1,\dots,\l_N)$ we define $\l^+$ to be the unique partition obtained by ordering its components in a non-increasing fashion and removing all parts $\l_i=0$. Let $w_+$ be the smallest word such that $\lambda = w_+\cdot\lambda^+$. The permutation $w_+^{-1}$ is obtained  by labeling each entry from $\lambda$ with a number from 1 to $N$, from the biggest entry to the smallest and from the left to the right. For instance $\lambda=(3,0,4,4,2) \Rightarrow w_+^{-1}=(3,5,1,2,4)$ and so $w_+=(3,4,1,5,2)$ and $\lambda^+=(4,4,3,2,0)$. The anti-dominant weight for this example is $\delta=(0,2,3,4,4)$. Then
\begin{align}
\rho(\lambda) := w_+\cdot\rho+\frac{1}{2}(N,N-1,\dots,1),
\end{align}
where $\rho=\tfrac12(N-1,N-3,\ldots,-(N-1))$. For the example above, $N=5$ and therefore $\rho=(2,1,0,-1,-2)$ and $\rho(\lambda)=1/2(5,0,7,4,-1)$.

The non-symmetric Macdonald polynomials comprise a basis for $\mathbb{C}_{q,t}[x_1,\dots,x_N]$, and one of their most important features is the following fact:
\begin{prop}
Let $\m$ be a partition and $\mathcal{R}_{\m} \subset \mathbb{C}_{q,t}[x_1,\dots,x_N]$ denote the space of linear combinations of the polynomials $E_{\l}$ whose composition index $\l$ is permutable to $\m$:
\begin{align*}
\mathcal{R}_{\m} 
=
{\rm Span}_{\mathbb{C}}\{E_{\l}\}_{\l : \l^{+} = \mu}.
\end{align*}
Then up to an overall multiplicative constant, $P_{\m}(x;q,t)$ is the unique polynomial in $\mathcal{R}_{\m}$ which also lives in $\LL_{N,\mathbb{F}}$.
\end{prop}

\subsection{ASEP polynomials and the matrix product expression}

A further set of non-symmetric polynomials, which also comprise a basis of $\mathbb{C}_{q,t}[x_1,\dots,x_N]$, were studied in \cite{KasataniT}. We refer to them as {\it ASEP polynomials}, and denote them by $f_{\l} \equiv f_{\l}(x_1,\dots,x_N;q,t)$. They are the unique family of polynomials which satisfy
\begin{align}
\label{f_init}
f_{\delta}(x;q,t) = E_{\delta}(x;q,t),
\quad \forall\ \delta = (\delta_1 \leq \cdots \leq \delta_N),
\\
\label{f_exch}
f_{s_i \l}(x;q,t) = T^{-1}_i f_{\l}(x;q,t),
\quad \text{when}\ \ \l_i < \l_{i+1},
\end{align}
where $s_i \l = (\l_1,\dots,\l_{i+1},\l_i,\dots,\l_N)$. Clearly by repeated use of (\ref{f_exch}), one is able to construct $f_{\l}$ for any composition, starting from $f_{\l^{-}} = E_{\l^{-}}$. Furthermore, because of the Hecke algebra relations (\ref{eq:hecke}), $f_{\l}$ is independent of the order in which one performs the operations (\ref{f_exch}), making the definition unambiguous. The polynomials $f_{\l}(x;q,t)$  can also be defined as solutions to the (reduced) quantum Knizhnik-Zamolodchikov ($q$KZ) equations
\begin{align}
\label{eq:qKZ1}
&T_i f_{\dots,\l_i,\l_{i+1},\dots}= t f_{\dots,\l_{i},\l_{i+1},\dots}
\qquad
\l_i = \l_{i+1},\\
\label{eq:qKZ2}
& T_i f_{\dots,\l_i,\l_{i+1},\dots} = f_{\dots,\l_{i+1},\l_i,\dots}
\qquad
\l_i > \l_{i+1}, \\
\label{eq:qKZ3}
& \omega 
f_{\l_N,\l_1,\dots,\l_{N-1}} = q^{\l_N}  f_{\l_1,\dots,\l_N}.
\end{align}
The approach of \cite{CantinidGW} relies on this definition of $f_{\l}(x;q,t)$. It is straightforward to extend this definition to the cases of compositions $\l$.  The connection of $f_\l$ with Macdonald polynomials follows from the following proposition.
\begin{prop}
For a partition $\l$ we have
\begin{align}\label{eq:Psumf}
\sum_{\l: \l^+ =\m}
f_{\l}(x;q,t)
=
P_{\mu}(x;q,t).
\end{align}
\end{prop}
By this proposition, if one is able to construct the solution to the reduced $q$KZ equation (\ref{eq:qKZ1})-(\ref{eq:qKZ3}) then it also leads to a construction of Macdonald polynomials. The key idea of \cite{CantinidGW} is to use integrability tools of the quantum group $U_{t^{1/2}}(\widehat{sl_{n+1}})$ in order to construct partition functions which are equal to $f_{\l}(x;q,t)$ for each $\l$ and then sum over $\l$ as in (\ref{eq:Psumf}).


\subsection{Matrix product expression for ASEP polynomials}\label{ssec:MPACdGW}
Let us turn to the matrix product construction of \cite{CantinidGW}. We choose here a slightly different convention, in what follows we present the ``transposed'' version of the construction which appears in \cite{CantinidGW}. This means that our key objects like the $R$-matrix and the $L$-matrix are related to those in \cite{CantinidGW} by a transposition.
 
Suppose we have an algebra consisting of operators $A_i(x)$, $i=0,\dots,n$ and $S$ which satisfy the exchange relations
\begin{align}
\label{eq:ER1}
&A_i(x)A_i(y) = A_i(y)A_i(x),\\
\label{eq:ER2}
&
t A_j(x)A_i(y) - \frac{t x - y}{x-y} \left( A_j (x)A_i(y) - A_j (y)A_i(x)\right)  = A_i(x)A_j (y),\\
\label{eq:ER3}
&S A_i(q x) = q^i A_i(x)S,
\end{align}
for all $0 \leq i < j \leq n$. Then we can use these operators to construct polynomials $f_\m(x;q,t)$.  
\begin{prop}
The following equality holds
\begin{align}
\label{eq:trace}
\Omega_{\l^+}(q,t) f_{\l}(x_1,\dots,x_N) = \Tr\left[A_{\l_1}(x_1)\cdots A_{\l_N}(x_N) S\right],
\end{align}
provided the right hand side is not vanishing identically and $\Omega_{\l^+}(q,t)$ is some factor depending on a particular realization of the operators $A_i(x)$.   
\end{prop}
This proposition is proved by showing the equivalence of (\ref{eq:qKZ1})--(\ref{eq:qKZ3}) and (\ref{eq:ER1})--(\ref{eq:ER3}), see  \cite{CantinidGW} for details.
Solutions to the relations (\ref{eq:ER1})-(\ref{eq:ER2}) can be recovered from the Yang--Baxter algebra corresponding to the quantum group $U_{t^{1/2}}(\widehat{sl_{n+1}})$ \cite{Dr87,Jimbo86a}, or rather a twisted version of it \cite{Resh90}. For models based on $U_{t^{1/2}}(\widehat{sl_{n+1}})$, the fundamental $R$-matrix acting on the tensor product space $V(x)\otimes V(y)$ of two $n+1$ dimensional spaces can be expressed in the form
\begin{align}\label{eq:RcheckE}
\check{R}^{(n)} (x,y)
=
\sum_{i=0}^{n}
E^{(ii)}
\otimes
E^{(ii)}
+
\frac{x - y}{t x - y}
\sum_{0 \leq i < j \leq n}
\Big(
t E^{(ji)}
\otimes
E^{(ij)}
+
E^{(ij)}
\otimes
E^{(ji)}
\Big)\nonumber
\\
\frac{t-1}{t x - y}
\sum_{0 \leq i < j \leq n}
\Big(
x
E^{(ii)}
\otimes
E^{(jj)}
+
y
E^{(jj)}
\otimes
E^{(ii)}
\Big),
\end{align}
where $E^{(ab)}$ with $0\leq a,b\leq n$ denotes the elementary $(n+1) \times (n+1)$ matrix with a single non-zero entry 1 at position $(a,b)$. The intertwining equation, or the Yang--Baxter algebra, for such models is given by
\begin{align}
\left [L(x)\otimes L(y)\right] \cdot \check{R}(x,y)
= 
\check{R}(x,y)\cdot \left [L(y)\otimes L(x)\right] ,
\label{eq:YBAdef}
\end{align}
where we have suppressed the superscript $(n)$, and $L(x)=L^{(n)}(x)$ is an $(n+1) \times (n+1)$ operator-valued matrix in $V(x)$. This algebra is well-studied and many solutions for $L(x)$ are known. The $L$-matrix used in \cite{CantinidGW} has operator-valued entries given by $n$ copies of the $t$-boson algebras. The $t$-boson algebra $\mathcal{A}$ is generated by three elements $\{a,a^{\dag},k\}$. We require $n$ copies of such algebras $\mathcal{A}^{\otimes n}$, and denote $\mathcal{A}_i=1\otimes \dots \otimes  \mathcal{A}\otimes \dots\otimes 1$ with $\mathcal{A}$ at the $i$-th position. Thus the set of  generators that we need to use is $\{a_i,a^{\dag}_i,k_i\}_{i=1}^n$. The algebra elements satisfy the following relations:
\begin{align}\label{eq:qboson}
&a_i k_i=t k_i a_i,\qquad 
a^{\dag}_i k_i=t^{-1} k_i a^{\dag}_i, \nonumber \\
&
a_i a_i^{\dag}=1-t k_i,\qquad 
a_i^{\dag}a_i =1- k_i,
\end{align}
and any element with the index $i$ commutes with any element with the index $j$ for $i\neq j$. For each copy of the $t$-boson algebra we assume the Fock representation $\mathcal{F}=\text{Span}\{\ket{m}\}_{m=0}^{\infty}$
\begin{align}\label{eq:Fock}
a^{\dag}\ket{m}=\ket{m+1},\qquad 
a\ket{m}=(1-t^m) \ket{m-1},\qquad 
k\ket{m}=t^m \ket{m},\qquad
\end{align}
and the dual Fock representation $\mathcal{F}^*=\text{Span}\{\bra{m}\}_{m=0}^{\infty}$ is given by
\begin{align}\label{eq:Fockdual}
\bra{m}a^{\dag}=\bra{m-1}(1-t^m)
\qquad
 \bra{m}a=\bra{m+1}, \qquad 
\bra{m}k=t^m\bra{m}.
\end{align}
The algebra $\mathcal{A}_i$ acts on the space $1\otimes \dots \otimes  \mathcal{F}\otimes \dots\otimes 1$ with $\mathcal{F}$ standing at the $i$-th position. By $\ket{I_1,\dots,I_n}$ with $I_i\in\mathbb{N}$ we will denote vectors in $\MF^{\otimes n}$ and similarly for the dual vectors.
If $I=(I_1,\dots,I_n)$, then we write $\ket{I}=\ket{I_1,\dots,I_n}$ and similarly for the dual vectors.

The $L$ matrix has the form \cite{CantinidGW}
\begin{align}
L^{(n)}_{ji}(x)
=
\left\{
\begin{array}{ll}
x \prod_{m=i+1}^{n}
k_m,
&
i=j
\\
\\
x a_j a^\dag_i
\prod_{m=i+1}^{n} k_m,
&
i>j
\\
\\
0,
&
i<j
\end{array}
\right.
\label{eq:Losc1}
\end{align}
for all $1 \leq i,j \leq n$, and
\begin{align}
L^{(n)}_{j0}
=
a_j,\
1 \leq j \leq n,
\quad\quad
L^{(n)}_{0i}(x)
=
x a^\dag_i \prod_{m=i+1}^{n}
k_m,\
1 \leq i \leq n,
\quad\quad
L^{(n)}_{00}
=
1.
\label{eq:Losc2}
\end{align}
Let $I=(I_1,\dots,I_n)$ and $K=(K_1,\dots,K_n)$ be two compositions. Using the Fock representation we can write the elements of the $L$-matrix explicitly by sandwiching them between a state $\bra{I}$ and a state $\ket{K}$
\begin{align}\label{eq:Lelements}
L_{I,K}^{j,i}(x)= 
\bra{I}
L^{(n)}_{j,i}(x) \ket{K}.
\end{align}
We set $I_{a,b}=\sum_{j=a}^b I_j$.  
The matrix elements of $L_{I,K}^{j,i}(x)$ for $I+e_{j}\neq K+e_{i}$ are equal to $0$ and otherwise are given by
\begin{align}\label{eq:LmatCdGW0}
L_{I,K}^{j,i}(x)
=
\left\{
\begin{array}{ll}
x t^{I_{i+1,n}},
&
i = j ,
\\ \\
x (1-t^{I_{i}}) t^{I_{i+1,n}},
&
i > j ,
\\ \\
0,
&
i < j ,
\end{array}
\right.
\end{align}
for all $1\leq i,j \leq n$, and 
\begin{align}
&
L_{I,K}^{0,0}(x)= 1,		\label{eq:LmatCdGW1} \\
&L_{I,K}^{0,i}(x)
= x (1-t^{I_{i}})t^{I_{i+1,n}}, \label{eq:LmatCdGW2}\\
&L_{I,K}^{j,0}(x)
= 1. 	\label{eq:LmatCdGW3}
\end{align}

Introduce $\mathbb{A}^{(n)}(x)$ an $(n+1)$-dimensional operator valued vector given by
\begin{align}
\mathbb{A}^{(n)}(x) = (A_0(x),\ldots, A_n(x)).
\end{align}
The exchange relations (\ref{eq:ER1})-(\ref{eq:ER3}) are equivalent to the Zamolodchikov--Faddeev (ZF) algebra \cite{ZZ1979,Fad1980}:
\begin{align}
\left [\mathbb{A}(x)\otimes \mathbb{A}(y)\right] \cdot \check{R}(x,y )
&= 
\left [\mathbb{A}(y)\otimes \mathbb{A}(x)\right] ,
\label{eq:ZFdef}\\
\label{eq:SZF}
S \mathbb{A}(qx) 
&=q^{\sum_i iE^{(ii)}}\mathbb{A}(x) S,
\end{align}
where we suppress $(n)$ and $\mathbb{A}(x)=\mathbb{A}^{(n)}(x)$  and $S=S^{(n)}$. 
Solutions of (\ref{eq:ZFdef}) can be constructed by rank-reducing the Yang--Baxter algebra (\ref{eq:YBAdef}) in the following way. Assume a solution of the modified $RLL$ relation (\ref{eq:YBAdef})
\begin{align}
\left [\tilde{L}(x)\otimes \tilde{L}(y)\right] \cdot \check{R}^{(n)}(x,y)
= 
\check{R}^{(n-1)}(x,y) \cdot\left [\tilde{L}(y)\otimes \tilde{L}(x)\right] ,
\label{eq:YBAdef_lowrank}
\end{align}
in terms of an $n\times (n+1)$ operator-valued matrix $\tilde{L}(x)=\tilde{L}^{(n)}(x)$, and an operator $s=s^{(n)}$ that satisfies
\begin{align}
s\tilde{L}(qx) =q^{\sum_i iE^{(ii)}}\tilde{L}(x) sq^{-\sum_i iE^{(ii)}},
\end{align}
which in components reads
\begin{align}
\label{eq:twist-comp}
s\tilde{L}_{ij}(qx) =q^{j-i}\tilde{L}_{ij}(x) s.
\end{align}
Then a solution to (\ref{eq:ER1}) and (\ref{eq:ER2}) can be constructed as follows:
\begin{align}
\label{eq:nestedMPA}
\mathbb{A}^{(n)}(x) &=\tilde{L}^{(1)}(x) \cdots \tilde{L}^{(n-1)}(x)  \cdot \tilde{L}^{(n)}(x),
\\
S^{(n)} &= s^{(1)}\cdots s^{(n-1)} \cdot s^{(n)},
\label{eq:nestedTwist}
\end{align}
provided that the operator entries of $\tilde{L}^{(a)}(x)$ commute with those of $\tilde{L}^{(b)}(y)$, for all $a \not=b$. The usual way to ensure this commutativity is to demand that the operator valued entries of $\tilde{L}^{(a)}$ act on some vector space $F_a$ while $\tilde{L}^{(b)}$ act on a different vector space $F_b$. One can show that solutions to (\ref{eq:YBAdef_lowrank}) can be obtained from the Yang--Baxter algebra (\ref{eq:YBAdef}) by trivialising the representation of a single $t$-boson algebra. This reduces the rank of $L^{(n)}(x)$ by one and thus gives rise to $\tilde{L}^{(n)}(x)$ (see  \cite{CantinidGW} for more details). The operators $s^{(j)}$ can be explicitly realized as 
\begin{align}
\label{eq:stwist}
s^{(j)}=\prod_{i=1}^{j-1} k_i^{i u},
\end{align}
 where $u$ is defined by $t^u=q$. 
 
Due to the nested structure of (\ref{eq:nestedMPA}) we need to use $n(n-1)/2$ copies of $t$-boson algebras. Let us denote by $\mathcal{A}_{a,b}$, with $1\leq a<b\leq n$, the algebra of $t$-boson operators with the index $b$ entering the matrix $\tilde{L}^{(n-a+1)}(x)$ in (\ref{eq:nestedMPA}) and $s^{(n-a+1)}$ in (\ref{eq:nestedTwist}). Thus the trace in (\ref{eq:trace}) runs over the tensor product of $\mathcal{F}_{a,b}$, for all $1\leq a<b\leq n$. In this realization of the operators $A_i(x)$ the normalization factor appearing in (\ref{eq:trace}) is given by 
\begin{align}\label{eq:norm}
\Omega_{\l^+}(q,t)
=
\frac{1}{
\prod_{1 \leq i<j \leq n}
\left(
1-q^{j-i} t^{\lambda'_i-\lambda'_j} 
\right)}.
\end{align}

The rank reduction method above can be reformulated in terms of the matrices $L^{(n)}(x)$. This new reformulation appeared as a special case in another work on matrix product formulae for symmetric polynomials \cite{GdGW}. In this approach we replace every reduced matrix $\tilde{L}^{(a)}(x)$ with $L^{(n)}(x)$. In the formula (\ref{eq:trace}) we have $n(n-1)/2$ $t$-boson algebras which are evaluated by a trace over their Fock spaces. However, with all matrices $L^{(n)}(x)$ being of rank $n$ we now have $n^2$ copies of $t$-boson algebras $\mathcal{A}_{a,b}$, $1\leq a,b\leq n$. Hence in the new formula we need to evaluate additional $n(n-1)/2+n$ algebras $\mathcal{A}_{a,b}$. We divide the set of $n^2$ algebras into three groups. The first group consists of $n(n-1)/2$ algebras $\mathcal{A}_{a,b}$ which have indices $1\leq a<b\leq n$; the operators of these algebras are taken in the Fock representation and traced as in (\ref{eq:trace}). The second group consists of $n(n-1)/2$ copies $\mathcal{A}_{a,b}$ with $1\leq b<a\leq n$; the operators of these algebras are redundant and need to be sandwiched between the vacuum states of their Fock representations. In the third group we have $n$ copies $\mathcal{A}_{a,a}$ with $a=1,\dots, n$; the corresponding operators need to be sandwiched between the vacuum states and special states defined by the composition $\l$. The algebras $\mathcal{A}_{a,b}$ act on $n^2$ Fock spaces which we can split into three parts, $\MF^{\otimes n^2}=\MF^-\otimes \MF^0\otimes \MF^+$, according to the formulae
\begin{align}
\label{eq:Focksplit}
\MF^-=\bigotimes_{1\leq b< a\leq n} \mathcal{F}_{a,b},\qquad
\MF^0=\bigotimes_{1\leq a \leq n} \mathcal{F}_{a,a}, \qquad
\MF^+=\bigotimes_{1\leq a< b \leq n} \mathcal{F}_{a,b}.
\end{align}
Define also more refined tensor products
\begin{align}
\label{eq:Focksplit2}
\mathcal{F}_a=\bigotimes_{1\leq b \leq n}\mathcal{F}_{a,b},\qquad 
\MF_a^-=\bigotimes_{1\leq b< a} \mathcal{F}_{a,b},\qquad
\MF_a^+=\bigotimes_{a< b \leq n} \mathcal{F}_{a,b}.
\end{align}
The corresponding dual Fock spaces are defined analogously.  The bra and ket vectors belonging to these spaces will acquire the corresponding superscripts.  A special role is played by the dual vector associated with a composition $\l$ with $\text{max}(\l)=n$ 
\begin{align}
\label{eq:lambdavector}
\bra{{\sf m}(\l)}=\bigotimes_{a=1}^n  \bra{{\sf m}_a(\l)}.
\end{align}

Now we can define the operators $A_l(x)$ and $S^{(n)}$:
\begin{align}
\label{eq:nestedMPAfull}
A_l^{(n)}(x) &=\sum_{\{i_2,\dots,i_{n}\}}L^{(n)}_{0,i_{n}}(x) \dots L^{(n)}_{i_{3},i_{2}}(x)  L^{(n)}_{i_{2},l}(x),
\\
S^{(n)} &= s_n^{(n)}\cdots s_2^{(n)} \cdot s_1^{(n)},
\label{eq:nestedTwist2}
\end{align}
where the sum runs over all possible values of the indices $i_a$ and the operators $s_a^{(n)}$ acts nontrivially only on $\MF_a^{+}$ and are given by 
\begin{align}
s_a^{(n)}=\prod_{a<j\leq n}k_j^{(j-a)u}.
\end{align}

Since all matrices above are of  rank $n$ we will omit the superscript $(n)$ in the following. 
\begin{prop}
For a partition $\l$ in an $n\times N$ box the following equality holds:
\begin{align}
\label{eq:trace2}
\Omega'_{\l^+}(q,t) f_{\l}(x_1,\dots,x_N) =
\bra{0}^{\MF^-}\otimes \bra{{\sf m}(\l)}^{\MF^0}\Tr_{\MF^+}\left[A_{\l_1}(x_1)\cdots A_{\l_N}(x_N) S\right]\ket{0}^{\MF^{0}}\otimes\ket{0}^{\MF^{-}},
\end{align}
where the normalisation is given by 
\begin{align}
\label{eq:normprime}
\Omega'_{\l^+}(q,t)= \frac{(t;t)_{{\sf m}(\l)}}{
\prod_{1 \leq i<j \leq n}
\left(
1-q^{j-i} t^{\lambda'_i-\lambda'_j} 
\right)}.
\end{align}
\end{prop}
The new normalisation factor can be written as $\Omega'_{\l^+}(q,t)=(t;t)_{{\sf m}(\l)}\Omega_{\l^+}(q,t)$ with $\Omega_{\l^+}(q,t)$ given in (\ref{eq:norm}). The occurrence of $\Omega_{\l^+}(q,t)$ follows the same logic as in \cite{CantinidGW} while the extra term $(t;t)_{{\sf m}(\l)}$ appears due to the presence of the evaluation in the space $\MF^0$. More precisely, the operator inside the trace in (\ref{eq:trace2}) must be such that its action on $\ket{0}^{\MF^{0}}$ produces the state $\ket{{\sf m}(\l)}^{\MF^{0}}$, otherwise the right hand side of (\ref{eq:trace2}) vanishes. Therefore we need to take into account the extra factor:
\begin{align*}
\bra{{\sf m}(\l)}^{\MF^0}\ket{{\sf m}(\l)}^{\MF^{0}}=\prod_{i=1}^n \langle{{\sf m}_i(\l)}|{{\sf m}_i(\l)}\rangle
=\prod_{i=1}^n (t;t)_{{\sf m}_i(\l)},
\end{align*}
where $\langle{{\sf m}_i(\l)}|{{\sf m}_i(\l)}\rangle$ is computed according to (\ref{eq:Fock})-(\ref{eq:Fockdual}).
For further discussion of the correspondence of (\ref{eq:trace}) and (\ref{eq:trace2}) we refer to \cite{GdGW}.  Below we use the second approach (\ref{eq:trace2}).


\subsection{Column operators}\label{ssec:columnsCdGW}
From (\ref{eq:nestedMPAfull}) we see that in (\ref{eq:trace2}) the operator part on the right hand side consists of a product of $n\times N$ matrices $L(x)$ together with the operator $S$. This product can be computed in different orders. According to the order given by the expression (\ref{eq:trace2}) one first computes the product over $L(x_i)$ with fixed $i$ which results in the row operator $A_l(x_i)$ as in (\ref{eq:nestedMPAfull}). After that the row operators are multiplied together. We would like to rewrite this product in a different order. Recall that $L$-matrices act on $n$ copies of Fock spaces $L(x)\in \text{End}(V(x)\otimes \mathcal{F}_i)$. 
Let us rewrite the operator part of (\ref{eq:trace2}) using (\ref{eq:nestedMPAfull}) and split the product of $L$-matrices into column operators instead 
\begin{align}\label{eq:AtoQ}
A_{\l_1}(x_1)\cdots A_{\l_N}(x_N) S=
\sum_{\{i\}}
&
\left(\prod_{a=1}^n L_{i_{n-a+2}^1,i_{n-a+1}^1}(x_1)\right)_{1,\dots,n}\cdots \left(\prod_{a=1}^n L_{i_{n-a+2}^N,i_{n-a+1}^N}(x_N)\right)_{1,\dots,n} S
\nonumber\\
=
\sum_{\{i\}}
&
\left(\prod_{l=1}^N L_{i_{n+1}^l,i_{n}^l}(x_l)\right)_{n}\cdots \left(\prod_{l=1}^N L_{i_{2}^l,i_{1}^l}(x_l)\right)_1 S
\nonumber\\
=
\sum_{\{i\}}
&
\prod_{a=1}^n\left(\prod_{l=1}^N L_{i_{a+1}^l,i_{a}^l}(x_l)\right)_{a} s_{a}^{(n)},
\end{align}
where $i_{1}^{l}=\l_l$, $i_{n+1}^{l}=0$ for all $l$, all other indices $i_{a}^{l}$ are summed over and the subscripts of products of $L$-matrices specify on which collection of Fock spaces $\MF_a$ the elements of the corresponding $L$-matrices act. In the first line a single product acts in all $n$ Fock spaces $\MF_1,\dots,\MF_n$, while in the second line the products are reorganized in such a way that a single string of $L$-matrices acts on one copy of $\MF_a$. Note that the trace over $\MF^-$ and the evaluation in $\MF^0$ in formula (\ref{eq:trace2}) can be similarly factorized over spaces $\MF_a$. Therefore we can write the trace formula as a product of $n$ matrices. Fix a composition $\l$ and introduce the column operators $Q^a(x_1,\dots,x_N)$:
\begin{align}
\label{eq:columnMac}
Q_{i,i'}^a(x_1,\dots,x_N)
&:=
\bra{0}^{\MF_{a}^-}\bra{{\sf m}_a(\l)}^{\MF_{a,a}}\Tr_{\MF^+_{a}} M_{i,i'}(x_1,\dots,x_N)\ket{0}^{\MF_{a,a}}\ket{0}^{\MF^-_{a}},\\
\label{eq:Monod}
M_{i,i'}(x_1,\dots,x_N)
&:=L_{i^1,i'{}^1}(x_1) \cdots  L_{i^N,i'{}^N}(x_N).
\end{align}
Note that  in (\ref{eq:Monod}) the entries of each $L$-matrix act in the same Fock space $\MF_a$, and $Q_{i,i'}^a(x_1,\dots,x_N)$ depends only on the $a$-th component of the multiplicity vector ${\sf m}(\l)$.
With these definitions we can rewrite (\ref{eq:AtoQ}) in terms of column operators $Q^a$. It follows that $f_\l(x;q,t)$ in (\ref{eq:trace2}), and also $P_\l(x;q,t)$ by (\ref{eq:Psumf}), can be written as sums of products of column operators $Q^a$. 
\begin{prop}
For a partition  $\l$ in an $n\times N$ box the following equalities hold
\begin{align}
\label{eq:fprodQ}
\Omega'_{\l^+}(q,t) f_{\l}(x_1,\dots,x_N) 
&=
\sum_{\{i_2,\dots,i_{n}\}}
\prod_{a=1}^n Q_{i_{a+1},i_a}^{a}(x_1,\dots,x_N), \\
\label{eq:PsumQf}
P_{\mu}(x_1,\dots,x_N;q,t)
&=
\frac{1}{\Omega'_{\m}(q,t)}
\sum_{i_1: i_1^+ = \mu}
\sum_{\{i_2,\dots,i_{n}\}}
\prod_{a=1}^n Q_{i_{a+1},i_a}^{a}(x_1,\dots,x_N),
\end{align}
where $Q_{i_{a+1},i_{a}}^a(x_1,\dots,x_N)$ is defined in (\ref{eq:columnMac}), $i_1=\l$ and $i_{n+1}=\emptyset$. 
\end{prop}
%


\subsection{Fusion}\label{ssec:Fusion}
In this section we perform the fusion of the lattice partition function on the right hand side of (\ref{eq:PsumQf}) and link this to the evaluation homomorphism thus establishing the desired matrix product construction for polynomials $W_\m(x;q,t;z)$. 

Consider two operations:
\begin{enumerate} 
\item For $i=1,\dots,N$ replace each $x_i$ with the alphabet $(x_i^{(1)},\dots,x_i^{(\MJ_i)})$, for some integer $\MJ_i$. 
\item Within each alphabet perform the specialization $x_i^{(j)}\rightarrow  t^{j-1}x_i$, for $j=1\dots,\MJ_i$. 
\end{enumerate} 
We define the operation $\e_{(x_i,\MJ_i)}$ accordingly:
\begin{align*}
\e_{(x_i,\MJ_i)}g(x_1,\dots,x_N)=
g(x_1,\dots,x_{i-1},x_i, t x_i,\dots, t^{\MJ_i-1}x_i,x_{i+1},\dots,x_N),
\end{align*}
where $g(x_1,\dots,x_N)$ is a symmetric function defined for any $N\in \mathbb{N}$.
If $\MJ=(\MJ_1,\dots,\MJ_N)$ is a set of positive integers and $x=(x_1,\dots,x_N)$ is an alphabet then we will write $\e_{(x,\MJ)}$ instead of $\e_{(x_N,\MJ_N)}\dots \e_{(x_1,\MJ_1)}$:
\begin{align}\label{eq:epsilon}
\e_{(x,\MJ)}g(x_1,\dots,x_N)=
g(x_1,t x_1,\dots, t^{\MJ_1-1} x_1,\dots,x_N,t x_N,\dots, t^{\MJ_N-1} x_N).
\end{align}
When acting with $\e_{(x,\MJ)}$, with $x=(x_1,\dots,x_N)$ and $\MJ=(\MJ_1,\dots,\MJ_N)$, on the right hand side of (\ref{eq:PsumQf}) we will encounter matrix elements of $M_{i,i'}$ of the form 
\begin{align}\label{eq:fusedM0}
M_{i,i'}(x_1,t x_1 ,\dots,t^{(\MJ_1-1)} x_{1},\dots,x_N,t x_N ,\dots,t^{(\MJ_N-1)} x_{N} ).
\end{align}
The compositions $i$ and $i'$ are of size $|\MJ|=\MJ_1+\dots + \MJ_N$. It is convenient to split these compositions into blocks of lengths $\MJ_l$, $l=1,\dots,N$,  as follows: $i^l=(i_1^l,\dots,i_{\MJ_l}^{l})$, so that $i=\cup_{l=1}^N i^l$ and similarly for $i'$. This allows us to factorize the matrix element $M_{i,i'}$ into blocks $M_{i^l,i'{}^l}$.  
Let $\l=(\l_1,\dots,\l_n)$ and $\m=(\m_1,\dots,\m_n)$ be two non-negative compositions. Following \cite{KulishRS,KirillovR} we define the fused matrix elements $\MM_{\l,\m}(x,\MJ)$ by 
\begin{align}\label{eq:fusedM}
\MM_{\l,\m}(x,\MJ):=
\sum_{\substack{i:{\sf m}(i)=\l\\ i':{\sf m}(i')=\m}}
\prod_{l=1}^N t^{-\MI(i^l_{\MJ_l},\dots,i_1^l)}
M_{i^l,i'{}^l}(x_l,t x_l,\dots,t^{(\MJ_l-1)} x_{l}),
\end{align}
where we introduced the {\it inversion numbers}:
\begin{align}
\MI(i_1,\dots,i_{\MJ}):
=
\#\{(j,k) | 1\leq j<k \leq {\MJ}, i_j < i_k \}.		\label{eq:invnum}
\end{align} 			
We define the coefficient $C_{\MJ}(\l)$ by summing the inversion numbers with a fixed multiplicity vector
\begin{align}\label{eq:invnorm}
C_{\MJ}(\l)
:=
\sum_{i: {\sf m}(i) = \l}
t^{-\MI(i)}
&=
\prod_{m=1}^{n} t^{-\l_m(\l_0+\cdots+\l_{m-1})}
\binom{\l_0 + \cdots + \l_m }{ \l_m}_t \\
&= t^{-\sum_{0\leq i<\ell\leq n} \l_i \l_\ell} \frac{(t,t)_{\MJ}}{(t,t)_{\l_0}(t,t)_{\l_1}\dots(t,t)_{\l_n}}, \nonumber
\end{align}
and $\l_0 := {\MJ}-\sum_{m=1}^{n} \l_m$. The products of matrices $M_{i,i'}$ with the specialization (\ref{eq:fusedM0}) can be expressed as products of matrices $\MM_{\l,\m}$. For $N=1$ we define the $\MJ$-{\it fused vertex}
\begin{align}\label{eq:fusedL}
\ML_{\l,\m}(x,\MJ):=
\frac{1}{C_\MJ(\l)}
\sum_{\substack{i:{\sf m}(i)=\l\\ i':{\sf m}(i')=\m}}
 t^{-\MI(i_{\MJ},\dots,i_1)}
M_{i,i'}(x,t x ,\dots, t^{(\MJ-1)}x).
\end{align}
This matrix is called the fused $L$-matrix. We can write products of matrices $M_{i,i'}$ in terms of products of matrices $\ML_{\l,\m}(x,\MJ)$. 

\begin{prop}
Let $\l=(\l_1\dots,\l_n)$ and $\n=(\n_1\dots,\n_n)$ be two non-negative compositions, then we have
\begin{align}\label{eq:fusedMprod}
\sum_{\m}\ML_{\l,\m}(x,\MJ)\ML_{\m,\n}(x,\MJ)
=&
\frac{1}{C_{\MJ}(\l)}
\sum_{\m}\sum_{\substack{i:{\sf m}(i)=\l \\ i':{\sf m}(i')=\m \\ i'':{\sf m}(i'')=\n}}
t^{-\MI(i_{\MJ},\dots,i_1)}\\
\times&
M_{i,i'}(x,t x ,\dots, t^{(\MJ-1)}x)
M_{i',i''}(x,t x ,\dots, t^{(\MJ-1)}x).
\nonumber
\end{align}
\end{prop}
\begin{proof}
The essential ingredient of the proof is a relation which shows that $\ML_{\l,\m}(x,\MJ)$ in (\ref{eq:fusedL}) can be rewritten as a single sum over $i$. The sum over $i'$ can be calculated and leads to an overall factor. By  (\ref{eq:Monod}) the matrix elements of $M$ are given by products of matrix elements of $L$. For any two permutations on $\MJ$ letters $s$ and $s'$ we have
\begin{align*}
&t^{-\MI(s'(\ell_{\MJ},\dots,\ell_1))}
\sum_{j: {\sf m}(j) = \l}
t^{-\MI(j_{\MJ},\dots, j_1)}
L^{j_1,s_1(\ell)}_{I,I^1}(x)
L^{j_2,s_2(\ell)}_{I^1,I^2}(t x)
\dots
L^{j_{\MJ},s_{\MJ}(\ell)}_{I^{{\MJ}-1},I^{\MJ}}(t^{{\MJ}-1} x)\\
&=
t^{-\MI(s(\ell_{\MJ},\dots,\ell_1))}
\sum_{j: {\sf m}(j) = \l}
t^{-\MI(j_{\MJ},\dots, j_1)}
L^{j_1,s'_1(\ell)}_{I,I^1}(x)
L^{j_2,s'_2(\ell)}_{I^1,I^2}(t x)
\dots
L^{j_{\MJ},s'_{\MJ}(\ell)}_{I^{{\MJ}-1},I^{\MJ}}(t^{{\MJ}-1}x).
\end{align*}
The proof of this statement is given in Proposition \ref{prop:texch}. We can rewrite this equation in terms of matrices $M_{i,i'}$:
\begin{align}\label{eq:texchangeM}
&t^{-\MI(s'(\ell_{\MJ},\dots,\ell_1))}
\sum_{j: {\sf m}(j) = \l}
t^{-\MI(j_{\ell_J},\dots, j_1)}
M_{j,s(\ell)}(x,t x ,\dots, t^{(\MJ-1)}x)
\\
&=
t^{-\MI(s(\ell_{\MJ},\dots,\ell_1))}
\sum_{j: {\sf m}(j) = \l}
t^{-\MI(j_{\ell_J},\dots, j_1)}
M_{j,s'(\ell)}(x,t x ,\dots, t^{(\MJ-1)}x).\nonumber
\end{align}
We can take $s$ to be the identity permutation and $s'$ to be a permutation which orders the set $(\ell_{\MJ},\dots,\ell_1)$ such that $\MI(s'(\ell))=1$. Inserting this into (\ref{eq:fusedL}) and using (\ref{eq:invnorm}) we get
\begin{align}
\ML_{\l,\m}(x,\MJ)
=&
\frac{1}{C_{\MJ}(\l)}
\sum_{i:{\sf m}(i)=\l }
t^{-\MI(i_{\MJ},\dots,i_1)}
M_{i,\tilde{i}}(x,t x ,\dots, t^{(\MJ-1)}x)
\sum_{i':{\sf m}(i')=\m}t^{-\MI(i'_{\MJ},\dots,i'_1)}
 \nonumber\\
=&
\frac{C_{\MJ}(\m)}{C_{\MJ}(\l)}
\sum_{i:{\sf m}(i)=\l }
t^{-\MI(i_{\MJ},\dots,i_1)}
M_{i,\tilde{i}}(x,t x ,\dots, t^{(\MJ-1)}x)
, \label{eq:fusedM2}
\end{align}
where the composition $\tilde{i}$ is ordered in a non-increasing fashion, so that $\MI(\tilde{i})=1$, and ${\sf m}(\tilde{i})=\m$. Applying (\ref{eq:texchangeM}) to the right hand side of (\ref{eq:fusedMprod}) and taking into account (\ref{eq:fusedM2}) we find that the sums separate and can be written as a product of two matrices $\ML$
\begin{align*}
&\sum_{\m}\sum_{i:{\sf m}(i)=\l}
\frac{C_{\MJ}(\m)}{C_{\MJ}(\l)}
t^{-\MI(i_{\MJ},\dots,i_1)}M_{i,\tilde{i}}(x,t x ,\dots, t^{(\MJ-1)}x)\\
&
\frac{1}{C_{\MJ}(\m)}
\sum_{\substack{i':{\sf m}(i')=\m \\ i'':{\sf m}(i'')=\n}}
t^{-\MI(i'_{\MJ},\dots,i'_1)}
M_{i',i''{}}(x,t x ,\dots, t^{(\MJ-1)}x)
=
\sum_{\m}
\ML_{\l,\m}(x,\MJ)
\ML_{\m,\n}(x,\MJ).
\end{align*}
\end{proof}

We can find an explicit form of the matrix elements of $\ML_{\l,\m}(x,\MJ)$. In terms of the elements of the $L$-matrix (\ref{eq:LmatCdGW0})-(\ref{eq:LmatCdGW3}), we have
\begin{align}\label{eq:Lfuseddef}
\ML^{\l,\m}_{\l',\m'}(x,\MJ)=
\frac{1}{C_{\MJ}(\l)}
\sum_{\substack{j: {\sf m}(j) =\l\\
\ell:{\sf m}(\ell)=\m}}
t^{-\MI(j_{\MJ},\dots,j_1)}
\prod_{k=1}^{\MJ}L^{j_k,\ell_k}_{\l',\m'}(t^{k-1} x).
\end{align}
Introduce the matrix $L^{\l,\m}_{\l',\m'}(x,z)$:
\begin{align}
\label{eq:Lone1}
&L^{\l,\m}_{\l',\m'}(x,z):=
\sum_{\k=0}^{\m}
 t^{\frac{1}{2}(\k.\k-|\k|)+\sum_{j<l}\m_j(\m'_l+\k_l)-\sum_{j<l}(\m_j-\k_j,\l_l)}\binom{\m'+\k}{\k}_t\binom{\l'}{\m-\k}_t
  x^{|\k|}z^{|\m-\k|}.
\end{align}
These are precisely the matrix elements (\ref{eq:Lone}) that we used for the construction in Section \ref{sec:modM}\footnote{In Section \ref{sec:modM} we denoted this matrix $\ML^{\l,\m}_{\l',\m'}(x,z)$.}. 
Again we assume the condition $\m'+\m=\l'+\l$ to hold, and otherwise $L^{\l,\m}_{\l',\m'}(x,z):=0$. 
\begin{prop}
The following identification of the matrix elements holds:
\begin{align}\label{eq:LLfused}
L^{\l,\m}_{\l',\m'}(x,-t^{\MJ} x)=
\ML^{\l,\m}_{\l',\m'}(x,\MJ).
\end{align}
\end{prop}
\begin{proof}
In this proof we define indices $\tilde{\m}=(\m_1,\dots,\m_{n-1})$ and $\tilde{\l}=(\l_1,\dots,\l_{n-1})$ and similarly $\tilde{\m}'=(\m_1',\dots,\m_{n-1}')$ and $\tilde{\l}'=(\l_1',\dots,\l_{n-1}')$. From the definition (\ref{eq:Lone1}) we find that the matrix elements $L_{\l',\m'}^{\l,\m}(x,z)$ satisfy the recurrence relation 
\begin{align}
\label{eq:Lxzrec}
&L^{\l,\m}_{\l',\m'}(x,z)=
b^{\l_n,\m_n}_{\l'_n,\m'_n}(x,z;|\m|)
L^{\tilde{\l},\tilde{\m}}_{\tilde{\l}{}',\tilde{\m}{}'}(t^{\l_n} x ,z),\\
&
\label{eq:biz}
b^{\l_n,\m_n}_{\l'_n,\m'_n}(x,z;w)=
\sum_{k=0}^{\m_n}
  t^{\frac{1}{2}(k^2-k)+(w-\m_n)(\m'_n+k-\l_n)}\binom{\m'_n+k}{k}_t\binom{\l'_n}{\m_n-k}_t
  x^{k} z^{\m_n-k}.
\end{align}
This recurrence relation together with the initial condition $L^{\l,\emptyset}_{\l',\m'}(x,z)=1$ completely determines $L^{\l,\m}_{\l',\m'}(x,z)$. The initial condition $\ML^{\l,\emptyset}_{\l',\m'}(x,\MJ)=1$ holds by (\ref{eq:Lfuseddef}), (\ref{eq:LmatCdGW1}), (\ref{eq:LmatCdGW3}) and (\ref{eq:invnorm}).
The rest of the proof consists of showing that (\ref{eq:Lxzrec}) is also satisfied by $\ML^{\l,\m}_{\l',\m'}(x,\MJ)$, i.e.
\begin{align}
\label{eq:Lrec}
&\ML^{\l,\m}_{\l',\m'}(x,\MJ)=
b^{\l_n,\m_n}_{\l'_n,\m'_n}(x,-t^{\MJ} x;|\m|)
\ML^{\tilde{\l},\tilde{\m}}_{\tilde{\l}{}',\tilde{\m}{}'}(t^{\l_n} x,\MJ-\l_n).
\end{align}
We prove (\ref{eq:Lrec}) constructively in Section \ref{ssec:rec}. 
\end{proof}


Consider the matrix product formula (\ref{eq:PsumQf}) and express $Q^a_{i,i'}$ through $M_{i,i'}$ using (\ref{eq:columnMac}). On the right hand side of (\ref{eq:PsumQf}) we have
\begin{align}\label{eq:rhsP}
\Omega'_{\m}(q,t){}^{-1}
\sum_{i_1: i_1^+ = \mu}
\sum_{\{i_2,\dots,i_{n}\}}
\prod_{a=1}^n 
\bra{0}^{\MF_{a}^-}\bra{{\sf m}_a(\l)}^{\MF_{a,a}}\Tr_{\MF^+_{a}} 
M_{i_{a+1},i_{a}}(x_1,\dots,x_N)
\ket{0}^{\MF_{a,a}}\ket{0}^{\MF^-_{a}}.
\end{align}
Now we can apply fusion to the sum of products of matrices $M_{i_{a+1},i_{a}}$ 
\begin{align}\label{eq:sumprodM}
\sum_{i_1: i_1^+ = \mu}
\sum_{\{i_2,\dots,i_{n}\}}
\prod_{a=1}^n 
M_{i_{a+1},i_{a}}(x_1,\dots,x_N),
\end{align}
and rewrite (\ref{eq:rhsP}) in terms of the fused matrices $\ML$. 
This requires an additional superscript for the index $i$ similarly as in (\ref{eq:fusedM}), i.e. $i_a^l=(i_{a,1}^l,\dots,i_{a,\MJ_l}^{l})$, so that $i_a=\cup_{l=1}^N i_a^l$. For convenience we also introduce a trivial summation $i_{n+1}:{\sf m}(i_{n+1})=0^n$. We fuse (\ref{eq:sumprodM}) by replacing it with the following expression:
\begin{align}\label{eq:fusion0}
&\sum_{i_{n+1}: i_{n+1}^+ = \emptyset}
\sum_{i_1: i_1^+ = \mu}
\sum_{\{i_2,\dots,i_{n}\}}
\prod_{l=1}^N
t^{-\MI(0^{\MJ_l})}
\prod_{a=1}^n 
M_{i^l_{a+1},i_{a}^l}(x_l,t x_l,\dots,t^{(\MJ_l-1)} x_{l}).
\end{align}
The summations over $\{i_2,\dots,i_{n}\}$ can be written similarly as in the right hand side of (\ref{eq:fusedMprod}). We need to introduce an additional set of summation indices $\s_{i}^{k}=(\s_{i,1}^{k},\dots,\s_{i,n}^{k})$, $\s_{i,j}^{k}\in \mathbb{N}$, $i=1,\dots,n+1$ and $k=1,\dots,N+1$, and we set $\s_{n+1,j}^{k}=0$, for all $j$ and $k$. With these summation indices we write (\ref{eq:fusion0}) as
\begin{align}\label{eq:fusion1}
&
\sum_{\{\s\}_{\m}}
\prod_{l=1}^N 
\sum_{i^l_{n+1}: {\sf m}(i^l_{n+1}) = \s_{n+1}^{l}}
\sum_{i^l_{n}: {\sf m}(i^l_{n}) = \s_{n}^{l}}
\dots
\sum_{i^l_{1}: {\sf m}(i^l_{1}) = \s_{1}^{l}}
t^{-\MI(0^{\MJ_l})}
\prod_{a=1}^n 
M_{i^l_{a+1},i_{a}^l}(x_l,t x_l,\dots,t^{(\MJ_l-1)} x_{l}),
\end{align}
where the summation $\{\s\}_{\m}$ runs over all $\s$ such that $\s_{1}^1{^+}+\dots+\s_{1}^N{}^+=\mu$. In the product over $l$ we can treat each term separately. Let us take the term with $l=1$, use (\ref{eq:fusedM2}) and apply repeatedly (\ref{eq:fusedMprod}), in which the second $\ML$-matrix on the left hand side is substituted with its definitions (\ref{eq:fusedL}). We find
\begin{align}\label{eq:fusion2}
&
\sum_{\{\s\}_{\m}}
\sum_{i^1_{n+1}: {\sf m}(i^1_{n+1}) = \s_{n+1}^{1}}
\sum_{i^1_{n}: {\sf m}(i^1_{n}{}) = \s_{n}^{1}}
\dots
\sum_{i^1_{1}: {\sf m}(i^1_{1}{}) = \s_{1}^{1}}
t^{-\MI(0^{\MJ_1})}
\prod_{a=1}^n 
M_{i^1_{a+1},i_{a}^1}(x_1,t x_1,\dots,t^{(\MJ_1-1)} x_{1})
\nonumber \\
=
&
\sum_{\{\s\}_{\m}}
\ML_{\s_{n+1}^1,\s_{n}^1}(x_1,\MJ_1)
\sum_{i^1_{n}: {\sf m}(i^1_{n}{}) = \s_{n}^{1}}
\frac{1}{C_{\MJ_1}(\s_n^1)}
\nonumber\\
\times &
\sum_{i^1_{n-1}: {\sf m}(i^1_{n-1}{}) = \s_{n-1}^{1}}
\dots
\sum_{i^1_{1}: {\sf m}(i^1_{1}{}) = \s_{1}^{1}}
t^{-\MI(i^1_{n,\MJ_1},\dots,i^1_{n,1})}
\prod_{a=1}^{n-1} 
M_{i^1_{a+1},i_{a}^1}(x_1,t x_1,\dots,t^{(\MJ_1-1)} x_{1})
\nonumber\\
=
&
\sum_{\{\s\}_{\m}}
\ML_{\s_{n+1}^1,\s_{n}^1}(x_1,\MJ_1)
\ML_{\s_{n}^1,\s_{n-1}^1}(x_1,\MJ_1)
\dots
\ML_{\s_{2}^1,\s_{1}^1}(x_1,\MJ_1).
\end{align}
The calculation in (\ref{eq:fusion2})  can be repeated for all factors in the product over $l$ in (\ref{eq:fusion1}). Thus (\ref{eq:fusion0}) becomes
\begin{align}\label{eq:fusion3}
&
\sum_{\{\s\}_{\m}}
\prod_{l=1}^N 
\ML_{\s_{n+1}^l,\s_{n}^l}(x_l,\MJ_l)
\dots
\prod_{l=1}^N 
\ML_{\s_{2}^l,\s_{1}^l}(x_l,\MJ_l).
\end{align}
To perform fusion of the expression (\ref{eq:rhsP}) we replace the products of sums of matrices $M$ (\ref{eq:sumprodM}) with that in
(\ref{eq:fusion0}). As we just showed (\ref{eq:fusion0}) is equal to (\ref{eq:fusion3}), thus  fusion of the expression (\ref{eq:rhsP}) is given by
\begin{align}\label{eq:rhsP2}
\Omega'_{\m}(q,t){}^{-1}
\sum_{\{\s\}_{\m}}
\prod_{a=1}^n 
\bra{0}^{\MF_{a}^-}\bra{{\sf m}_a(\l)}^{\MF_{a,a}}\Tr_{\MF^+_{a}} 
\prod_{l=1}^N 
\ML_{\s_{a+1}^l,\s_{a}^l}(x_l,\MJ_l)
\ket{0}^{\MF_{a,a}}\ket{0}^{\MF^-_{a}}.
\end{align}
Finally, we introduce the fused column operators $\MQ^a_{\s,\r}(x)=\MQ^a_{\s,\r}(x_1,\dots,x_N)$:
\begin{align}\label{eq:fusedQ}
\MQ^a_{\s,\r}(x_1,\dots,x_N)=
\bra{0}^{\MF_{a}^-}\bra{{\sf m}_a(\l)}^{\MF_{a,a}}\Tr_{\MF^+_{a}} 
\prod_{l=1}^N 
\ML_{\s^l,\r^l}(x_l,\MJ_l)
\ket{0}^{\MF_{a,a}}\ket{0}^{\MF^-_{a}},
\end{align}
where $\s^{l}=(\s^{l}_1,\dots,\s^{l}_n)$ and $\r^{l}=(\r^{l}_1,\dots,\r^{l}_n)$. Fusion of the right hand side of (\ref{eq:PsumQf}) is now complete and we arrive at the following result:
\begin{align}\label{eq:fusedP}
\e_{x,\MJ}P_{\mu}(x_1,\dots,x_N;q,t)=
\Omega'_{\m}(q,t){}^{-1}
\sum_{\{\s\}_{\m}}
\prod_{a=1}^n 
\MQ_{i_{a+1},i_a}^{a}(x_1,\dots,x_N).
\end{align}

In order to formulate the main result of this section we define the partition function $\MZ_\m(x;q,t;z)$.
\begin{defn}
\begin{align}\label{eq:Zxz}
\MZ_\m(x;q,t;z):=\MN_{\m}(q,t) \sum_{\{\s\}_\m}\prod_{i=1}^n \MQ^{i}_{\s_{i+1},\s_i}(x;t),
\end{align}
where the overall factor $\MN_{\m}(q,t)$ reads
\begin{align}\label{eq:NZxz}
\MN_{\m}(q,t) =
\prod_{1\leq i<j\leq n}
\prod_{\m'_{j+1}\leq \ell\leq \m'_j}(1-q^{j-i} t^{\m'_i-\ell}).
\end{align}
\end{defn}
The normalisation factor is chosen such that $\MZ_\m$ has coefficients in $\mathbb{N}[q,t]$. This is demonstrated in Section \ref{sec:modM}.
\begin{prop}\label{prop:WZ}
We have the following identification:
\begin{align}\label{eq:WeqZ}
W_{\mu}(x;q,t;z)=
\MZ_\m(x;q,t;z).
\end{align}
\end{prop}
\begin{proof}
We will derive this formula from (\ref{eq:fusedP}). To see what happens with $P_{\mu}(x;q,t)$ under the action of $\e_{(x,\MJ)}$ we need to recall that $P_{\mu}(x;q,t)$ can be expanded in the basis of power sum functions $p_r(x_1,\dots,x_N)$ and perform the two operations on $p_r(x_1,\dots,x_N)$. In fact, it is enough to consider $p_r(x_i)$ due to the property (\ref{eq:psumadd}). We have
\begin{align*}
\e_{(x_i,\MJ_i)}p_r(x_i)
=
x_i^r \sum_{j=0}^{\MJ_i-1} t^{j r} =
\frac{x_i^r- t^{\MJ_i r}x_i^r}{1-t^r}.
\end{align*}
Introduce a new variable $z_i$ by the formula:
\begin{align}
\label{eq:z}
z_i=-t^{\MJ_i} x_i.
\end{align}
This relation should be thought of as analytic continuation of $\MJ$ to arbitrary complex values. 
We can express the operation $\e_{(x_i,\MJ_i)}$ as 
\begin{align*}
\e_{(x_i,\MJ_i)}p_r(x_i)
=
\frac{x_i^r-(-1)^r z_i^r}{1-t^r}.
\end{align*}
This is precisely the action of $\pi_{x,z}^x$ as defined in (\ref{eq:evhom}). Therefore, by (\ref{eq:evW}), on the left side of (\ref{eq:fusedP}), we obtain
\begin{align}\label{eq:proofMPAlhs}
c_\m(q,t)^{-1}W_\m(x;q,t;z). 
\end{align}
In terms of $\MZ_\m(x;q,t;z)$ the right hand side of (\ref{eq:fusedP}) reads
\begin{align}\label{eq:proofMPArhs}
\frac{1}{\Omega'_{\m}(q,t)\MN_{\m}(q,t)}
\MZ_\m(x;q,t;z).
\end{align}
Equating (\ref{eq:proofMPAlhs}) with (\ref{eq:proofMPArhs}) we get
\begin{align}\label{eq:proofMP0}
W_\m(x;q,t;z)=
\frac{c_\m(q,t)}{\Omega'_{\m}(q,t)\MN_{\m}(q,t)}
\MZ_\m(x;q,t;z).
\end{align}
With (\ref{eq:bcc}), (\ref{eq:normprime}) and (\ref{eq:NZxz}) we check that the factor in front of $\MZ_\m$ is equal to $1$:
\begin{align}\label{eq:prefactor}
\frac{c_\m(q,t)}{\Omega'_{\m}(q,t)\MN_{\m}(q,t)}
=
\frac{\prod_{1 \leq i<j \leq N}
\left(
1-q^{j-i} t^{\m'_i-\m'_j} 
\right)
\cdot
\prod_{i=1}^{N}
\prod_{j=1}^{\m_i}
(1-q^{\m_i-j} t^{\m_j'-i+1})
}{\prod_{i=1}^N\prod_{j=0}^{\m'_i-\m'_{i+1}}(1-t^{1+j})
\cdot
\prod_{1\leq i<j\leq N}
\prod_{\m'_{j+1}\leq \ell\leq \m'_j}(1-q^{j-i} t^{\m'_i-\ell})}
=1.
\end{align}
\end{proof}
%


\subsection{Calculation of the matrix $\ML(x,\MJ)$}\label{ssec:rec}
In this section we show that the matrix elements $\ML^{\l,\m}_{\l',\m'}(x,\MJ)$ satisfy the recurrence relation in (\ref{eq:Lrec}). 
In what follows we will work mainly with components of the $R$-matrix and the $L$-matrix. Let us prepare a few definitions and identities. Define the $R$-matrix $R(z)=R(z,1)$:
\begin{align}
\label{eq:Rcomponents}
R(z)=\sum_{0\leq i_a,i_b,j_a,j_b\leq n}
\Big[ R_{ab}(z) \Big]^{i_a j_a}_{i_b j_b}
E^{(i_a j_a)}\otimes E^{(i_b j_b)}.
\end{align}
The components of the $R$-matrix can be written in the form:
\begin{align}
\label{eq:fundRcompact}
\Big[ R_{ab}(z) \Big]^{i_a j_a}_{i_b j_b}
=
\frac{t^{\th(j_a<i_b)}z^{\th(j_a<i_a)} (1-t^{\th(j_a=i_b)}z^{\th(j_a=i_a)})}{1-t z},
\end{align}
where the indices obey the ``conservation property''  $i_a+i_b=j_a+j_b$, otherwise the matrix elements are equal to $0$.
The $R$-matrix $\check{R}(z)$ is obtained from $R(z)$ by acting with the permutation operator from the left, thus the  elements of $\check{R}(z)$ are given by (cf. (\ref{eq:RcheckE}))
\begin{align}
\label{eq:Rcheck}
\Big[\check{R}_{ab}(z) \Big]^{i_a j_a}_{i_b j_b}=
\Big[ R_{ab}(z) \Big]^{i_b j_a}_{i_a j_b}.
\end{align}
The RLL equation (\ref{eq:YBAdef}) in components reads
\begin{align}
\label{eq:RcheckLL}
\sum_{j_1,j_2 = 0}^{n}
\Big[ \check{R}_{12}(y/x) \Big]^{i_1 j_1}_{i_2 j_2}
L^{j_1,\ell_1}_{I,I'}(x)
L^{j_2,\ell_2}_{I',I''}(y)
=
\sum_{j_1,j_2 = 0}^{n}
L^{i_1,j_1}_{I,I'}(y)
L^{i_2,j_2}_{I',I''}(x )
\Big[ \check{R}_{12}(y/x) \Big]^{j_1 \ell_1}_{j_2 \ell_2}.
\end{align}
One can directly check that the matrix elements of $\check{R}$ satisfy the following identities:
\begin{align}
\label{eq:t-exch1}
& \sum_{s_1,s_2=0}^n t^{-\MI(s_1,s_2)}\Big[ \check{R}(t) \Big]^{s_1 i_1}_{s_2 i_2} = t^{-\MI(i_2,i_1)},\\
\label{eq:t-exch2}
&\Big[ \check{R}_{bc}(t) \Big]^{i_2 k_2}_{i_1 k_1}  = t^{\MI(k_2,k_1)-\MI(k_1,k_2)}  \Big[ \check{R}_{bc}(t) \Big]^{i_2 k_1}_{i_1 k_2}.
\end{align}

Let $j=(j_1,\dots,j_N)$, $\ell=(\ell_1,\dots,\ell_N)$, $I=(I_1,\dots,I_n)$, $J=(J_1,\dots,J_n)$, $K=(K_1,\dots,K_n)$ and $L=(L_1,\dots,L_n)$ be compositions of non-negative integers. Let in addition $I^j=(I^j_1,\dots,I^j_n)$ be a sequence of non-negative integer compositions, $j=0,\dots,N$, with $I^0=I$ and $I^N=K$. The $N$-fused vertex (\ref{eq:fusedL}) can be written in components as
\begin{align}\label{eq:fusedv}
\ML^{J,L}_{I,K}(x,N)=
\frac{1}{C_N(J)}
\times
\sum_{\substack{j: {\sf m}(j) = J \\ \ell: {\sf m}(\ell) = L}}
t^{-\MI(j_N,\dots, j_1)}
L^{j_1,\ell_1}_{I,I^1}(x)
L^{j_2,\ell_2}_{I^1,I^2}(t x)
\dots
L^{j_N,\ell_N}_{I^{N-1},K}(t^{N-1} x).
\end{align}
Let us introduce a hybrid $\ML$-matrix $\MTL$:
\begin{align}\label{eq:hybridL}
\MTL^{J,\ell}_{I,K}(x,N)=
\frac{1}{C_N(J)}
\times
\sum_{j: {\sf m}(j) = J}
t^{-\MI(j_N,\dots, j_1)}
L^{j_1,\ell_1}_{I,I^1}(x)
L^{j_2,\ell_2}_{I^1,I^2}(t x)
\dots
L^{j_N,\ell_N}_{I^{N-1},I^{K}}(t^{N-1}x ).
\end{align}
The matrices $\ML$ and $\MTL$ are related by
\begin{align}\label{eq:LLtilde}
\ML^{J,L}_{I,K}(x,N)=
\sum_{ \ell: {\sf m}(\ell) = L}
\MTL^{J,\ell}_{I,K}(x,N).
\end{align}
Let us show that this summation can be computed. 
Take the Yang--Baxter equation (\ref{eq:RcheckLL}) and specialize the spectral parameters $y=t x$
\begin{align*}
\sum_{j_1,j_2 = 0}^{n}
\Big[ \check{R}_{12}(t) \Big]^{i_1 j_1}_{i_2 j_2}
L^{j_1,\ell_1}_{I,I'}(x)
L^{j_2,\ell_2}_{I',I''}(t x)
=
\sum_{j_1,j_2 = 0}^{n}
L^{i_1,j_1}_{I,I'}(t x)
L^{i_2,j_2}_{I',I''}(x )
\Big[ \check{R}_{12}(t) \Big]^{j_1 \ell_1}_{j_2 \ell_2}.
\end{align*}
Multiplying both sides by $t^{-\MI(i_1,i_2)}$ and summing over $\{i_1,i_2\}$ we find
\begin{align*}
&\sum_{j_1,j_2 = 0}^{n}\sum_{i_1,i_2 = 0}^{n}t^{-\MI(i_1,i_2)}\Big[ \check{R}_{12}(t) \Big]^{i_1 j_1}_{i_2 j_2}
L^{j_1,\ell_1}_{I,I'}(x)
L^{j_2,\ell_2}_{I',I''}(t x)\\
=&
\sum_{i_1,i_2 = 0}^{n}\sum_{j_1,j_2 = 0}^{n}t^{-\MI(i_1,i_2)}
L^{i_1,j_1}_{I,I'}(t x)
L^{i_2,j_2}_{I',I''}(x )
\Big[ \check{R}_{12}(t) \Big]^{j_1 \ell_1}_{j_2 \ell_2}.
\end{align*}
The sum over $\{i_1,i_2\}$ on the left hand side can be computed using (\ref{eq:t-exch1}). As a result it produces a new factor of $t$
\begin{align*}
&\sum_{j_1,j_2 = 0}^{n}t^{-\MI(j_2,j_1)}
L^{j_1,\ell_1}_{I,I'}(x)
L^{j_2,\ell_2}_{I',I''}(t x)\\
=&
\sum_{i_1,i_2 = 0}^{n}\sum_{j_1,j_2 = 0}^{n}t^{-\MI(i_1,i_2)}
L^{i_1,j_1}_{I,I'}(t x)
L^{i_2,j_2}_{I',I''}(x )
\Big[ \check{R}_{12}(t) \Big]^{j_1 \ell_1}_{j_2 \ell_2}.
\end{align*}
On the left hand side we recognise the matrix $\MTL^{J,\ell}_{I,K}(x,N)$ with $N=2$ according to (\ref{eq:hybridL}). This leads to a new formula for $\MTL^{J,\ell}_{I,K}(x,2)$
\begin{align*}
\MTL^{J,(\ell_1,\ell_2)}_{I,I''}(x,2)
=
\sum_{i_1,i_2 = 0}^{n}\sum_{j_1,j_2 = 0}^{n}t^{-\MI(i_1,i_2)}
L^{i_1,j_1}_{I,I'}(t x)
L^{i_2,j_2}_{I',I''}(x )
\Big[ \check{R}_{12}(t) \Big]^{j_1 \ell_1}_{j_2 \ell_2}.
\end{align*}
The dependence on $\{\ell_1,\ell_2\}$ simplifies and from (\ref{eq:t-exch2}) we deduce that  
\begin{align}\label{eq:Lexch}
\MTL^{J,(\ell_1,\ell_2)}_{I,I''}(x,2)=
t^{\MI(\ell_1,\ell_2)-\MI(\ell_2,\ell_1)}
\MTL^{J,(\ell_2,\ell_1)}_{I,I''}(x,2).
\end{align}
This allows us to order the indices $\{\ell_1,\dots,\ell_N\}$ in $\MTL^{J,\ell}_{I,I''}(x,N)$ at the cost of a simple $t$-factor.
\begin{prop}\label{prop:texch}
Let $s$ and $s'$ be any two permutations on $N$ letters, then the $N$-vertex satisfies
\begin{align}
\label{eq:prop3}
t^{\MI(s(\ell_N,\dots,\ell_1))}
\MTL^{J,s(\ell)}_{I,K}(x,N)=
t^{\MI(s'(\ell_N,\dots,\ell_1))}
\MTL^{J,s'(\ell)}_{I,K}(x,N).
\end{align}
\end{prop}
\begin{proof}
The vanishing properties of the $R$-matrix allow us to reduce this statement to the one when $s$ and $s'$ are transpositions and $N=2$. The latter is precisely given by  (\ref{eq:Lexch}).
\end{proof}
Using this result the formula for the matrix $\ML$ can be simplified. Consider (\ref{eq:LLtilde}) and let $\ell^-$ denote the permutation of $\ell$ in which the order of the components is non-decreasing, then we compute
\begin{align}\label{eq:LLtilde2}
\ML^{J,L}_{I,K}(x,N)=
\sum_{ \ell: {\sf m}(\ell) = L}
\MTL^{J,\ell}_{I,K}(x,N)=
\MTL^{J,\ell^-}_{I,K}(x,N)
\sum_{ \ell: {\sf m}(\ell) = L}
 t^{-\MI(\ell_N,\dots,\ell_1)} =
 C_N(L)
 \MTL^{J,\ell^-}_{I,K}(x,N).
\end{align}
Another useful ordering of $\ell$ is $\ellt^-=(n,\dots,n,\ellt_1^-,\dots,\ellt^-_{N-L_n})$, where $\ellt_1^-,\dots,\ellt^-_{N-L_n}$ are non-decreasing, thus we can write
\begin{align}\label{eq:LLtilde3}
\ML^{J,L}_{I,K}(x,N)=
t^{L_n(N-L_{n})}
C_N(L)
 \MTL^{J,\ellt^-}_{I,K}(x,N).
\end{align}
When the indices $\ell$ in $\MTL$ are ordered according to $\ell^-$ or $\ellt^{-}$ we can write $\ML$ in terms of blocks of matrices $L$ whose indices $\ell_i$ are equal to a fixed value. The product of $L$ matrices of the form
\begin{align*}
\prod_{s=s'}^{s''}
L^{j_s,i}_{I^{s-1},I^s}(t^{s-1} x),
\end{align*}
where $s'$, $s''$ and $i$ are some integers, will be called the {\it $i$-block}. Let $\It, \Jt, \Kt$ and $\Lt$ be compositions of length $n$ whose last parts vanish $\It_n=\Jt_n=\Kt_n=\Lt_n=0$ and otherwise they coincide with $I, J, K$ and $L$, respectively. The matrix elements $\ML^{\Jt,\Lt}_{\It,\Kt}(x,N)$ coincide with the matrix elements $\ML^{(J_1,\dots,J_{n-1}),(L_1,\dots,L_{n-1})}_{(I_1,\dots,I_{n-1}),(K_1,\dots,K_{n-1})}(x,N)$.
\begin{prop}
The matrix elements of $\ML$ satisfy the recurrence relation:
\begin{align}\label{eq:LNrec}
\ML^{J,L}_{I,K}(x,N)=
c^{J,L}_{I,K}(x,N)
 \ML^{\Jt,\Lt}_{\It,\Kt}(t^{J_n} x,N-J_n),
\end{align}
where the coefficients $c^{J,L}_{I,K}(x,N)$ are given by
\begin{align}\label{eq:cb}
c^{J,L}_{I,K}(x,N)=b^{J_n,L_n}_{I_n,K_n}(x,-t^N x;|L|).
\end{align}
\end{prop}
\begin{proof}
Consider the ordering $\ell^-$ in $\MTL$, then by (\ref{eq:LLtilde2}) and (\ref{eq:hybridL}) we have
\begin{align}\label{eq:Lord1}
\ML^{J,L}_{I,K}(x,N)=
&\frac{C_N(L)}{C_N(J)}
\sum_{j :{\sf m}(j) = J}
t^{-\MI(j_N,\dots,j_1)}
\times			\nonumber 	\\
&
\prod_{s=1}^{L_0}
L^{j_s,0}_{I^{s-1},I^s}(t^{s-1} x)
\times
\prod_{i=1}^{n}
\prod_{s=\sum_{m=0}^{i-1}L_m+1}^{\sum_{m=0}^i L_{m}}
L^{j_s,i}_{I^{s-1},I^s}(t^{s-1} x). 
\end{align}
On the other hand, if we consider the ordering $\ellt^-$ in $\MTL$, then by (\ref{eq:LLtilde3}) and (\ref{eq:hybridL}) we have
\begin{align}\label{eq:Lord2}
\ML^{J,L}_{I,K}(x,N)=
&\frac{t^{L_n(N-L_{n})}
C_N(L)}{C_N(J)}
\sum_{j :{\sf m}(j) = J}
t^{-\MI(j_N,\dots,j_1)}
\times			\nonumber		\\
&
\prod_{s=1}^{L_n}
L^{j_s,n}_{I^{s-1},I^s}(t^{s-1} x)
\prod_{s=L_n+1}^{L_0+L_n}
L^{j_s,0}_{I^{s-1},I^s}(t^{s-1} x)
\times
\prod_{i=1}^{n-1}
\prod_{s=\sum_{m=0}^{i-1} L_{m} +L_n+1}^{\sum_{m=0}^i L_{m} +L_n}
L^{j_s,i}_{I^{s-1},I^s}(t^{s-1} x). 
\end{align}
The idea of the proof is to express the right hand side of (\ref{eq:Lord2}) as a sum  of the matrix elements $\ML^{\Jt,\Lt}_{\It,\Kt}(t^{J_n} x,N-J_n)$ written in the form (\ref{eq:Lord1}). 
In the summand of (\ref{eq:Lord2}) we distinguish three components: the $0$-block, the $n$-block and $n-1$ different $i$-blocks with $0<i<n$. There are in total $J_n$ cases when the indices $j_s$ are such that $j_s=n$ and all of these cases must be present in the $0$-block and in the $n$-block due to the vanishing property in (\ref{eq:LmatCdGW0}).
We split the summation over $j$ into two summations as follows. Let $j'=(j'_{J_n+1},\dots,j'_{N})$ be the set obtained from $j=(j_1,\dots,j_N)$ by removing all instances $j_i=n$, $1\leq i\leq N$, and preserving the order of the remaining labels $j_i$. Let $a=(a_1,\dots,a_{L_0+L_n})$ be another set obtained from $j$  by setting $a_i=\theta(j_i=n)$. The set $a$ records the positions of $n$ in the $0$-block and in the $n$-block. For the two new sets we have ${\sf m}(j')=\Jt$ and ${\sf m}(a)=J-\Jt$. Clearly, the correspondence between $j$ and $j'\cup a$ is bijective. Under this division of the set $j$ the inversion numbers can be expressed as follows
\begin{align*}
\MI(j_N,\dots,j_1)=\MI(j'_{N},\dots,j'_{J_n+1})+\MI(a_{L_0+L_n},\dots,a_1)+J_n \sum_{m=1}^{n-1}L_m.
\end{align*}
Finally, we note that $L^{j_s,0}_{I^{s-1},I^s}(x)=1$ by (\ref{eq:LmatCdGW3}), therefore we can write (\ref{eq:Lord2}) as 
\begin{align}\label{eq:Lord3}
&\ML^{J,L}_{I,K}(x,N)=
\frac{t^{L_n(N-L_{n})-J_n \sum_{m=1}^{n-1}L_m}
C_N(L)}{C_N(J)}
\sum_{a :{\sf m}(a) = J-\Jt}
t^{-\MI(a_{L_0+L_n},\dots,a_1)}
\prod_{s=1}^{L_n}
L^{a_s,n}_{I^{s-1},I^s}(t^{s-1} x)
			\nonumber		\\
&
\sum_{j' :{\sf m}(j') = \Jt}
t^{-\MI(j'_{N},\dots,j'_{J_n+1})}
\prod_{s=J_n+1}^{L_0+L_n}
L^{j'_s,0}_{I^{s-1},I^s}(t^{s-1} x)
\times 
\prod_{i=1}^{n-1}
\prod_{s=\sum_{m=0}^{i-1} L_{m} +L_n+1}^{\sum_{m=0}^i L_{m}+L_n}
L^{j'_s,i}_{I^{s-1},I^s}(t^{s-1} x),
\end{align}
where we introduced a product of $L^{j'_s,0}_{I^{s-1},I^s}$ in the second line and removed the product  of $L^{j_s,0}_{I^{s-1},I^s}$ which appeared in (\ref{eq:Lord2}). In the second line in (\ref{eq:Lord3}) we can recognize $\ML^{\Jt,\Lt}_{I,K}(t^{J_n} x,N-J_n)$ written in the form (\ref{eq:Lord1}). In fact we can replace $\ML_{I,K}^{\Jt,\Lt}$ with $\ML_{\It,\Kt}^{\Jt,\Lt}$ times a factor of $t$. By (\ref{eq:LmatCdGW0}) and (\ref{eq:LmatCdGW2}) we find that $t^{I^{s-1}_n}=t^{K_n}$ is an overall factor of every term $L^{j'_s,i}_{I^{s-1},I^s}(t^{s-1} x)$ for $i>0$ in the second line in (\ref{eq:Lord3}), thus 
\begin{align}\label{eq:IKtilde}
\ML_{I,K}^{\Jt,\Lt}(t^{J_n} x,N-J_n)=
t^{K_n \sum_{m=1}^{n-1}L_m}
\ML^{\Jt,\Lt}_{\It,\Kt}(t^{J_n} x,N-J_n).
\end{align}
After inserting $\ML_{I,K}^{\Jt,\Lt}(t^{J_n} x,N-J_n)$ in (\ref{eq:Lord3}) we get the recurrence relation
\begin{align}\label{eq:Lrec0}
&\ML^{J,L}_{I,K}(x,N)=
c^{J,L}_{I,K}(x,N)
\ML^{\Jt,\Lt}_{\It,\Kt}(t^{J_n} x,N-J_n),
\end{align}
where the coefficients $c^{J,L}_{I,K}(x,N)$ are given by
\begin{align}\label{eq:Lreccoef}
c^{J,L}_{I,K}(x,N)
=&
t^{L_n(N-L_{n})+(K_n-J_n)\sum_{m=1}^{n-1}L_m}
\frac{C_{N-J_n}(\Jt)C_N(L)}{C_N(J)C_{N-J_n}(\Lt)}
			\nonumber \\
\times &
\sum_{a :{\sf m}(a) = J-\Jt}
t^{-\MI(a_{L_0+L_n},\dots,a_1)}
\prod_{s=1}^{L_n}
L^{a_s,n}_{I^{s-1},I^s}(t^{s-1} x).
\end{align}
Now we need to demonstrate the validity of (\ref{eq:cb}). We first focus on the summation over $a$. The dependence on $a_i$ for $i>L_n$ enters (\ref{eq:Lreccoef}) only through the inversion number $\MI(a_{L_0+L_n},\dots,a_1)$. Divide the set $a$ into two parts $(a_1,\dots,a_{L_n})$ and $(a_{L_{n}+1},\dots,a_{L_0+L_n})$. Let the first part contain $k$ instances of $a_i=n$, $i=1,\dots,L_n$, for some integer $0\leq k\leq L_n$, and  the second part contain $k'=J_n-k$ instances of $a_i=n$, $i=L_n+1,\dots,L_0+L_n$. We can now sum over $k$ and over two sets $(a_1,\dots,a_{L_n})$ and $(a_{L_{n}+1},\dots,a_{L_0+L_n})$. The inversion numbers split as follows:
\begin{align*}
\MI(a_{L_0+L_n},\dots,a_1)=\MI(a_{L_n},\dots,a_1)+\MI(a_{L_{n}+L_0},\dots,a_{L_n+1})+k(L_0-J_n+k).
\end{align*}
The sum over $(a_{L_{n}+1},\dots,a_{L_0+L_n})$ can be readily computed with (\ref{eq:invnorm}), and we find
\begin{align}\label{eq:Lreccoef1}
c^{J,L}_{I,K}(x,N)
= &
t^{L_n(N-L_{n})+(K_n-J_n)\sum_{m=1}^{n-1}L_m}
\frac{C_{N-J_n}(\Jt)C_N(L)}{C_N(J)C_{N-J_n}(\Lt)}
			\nonumber \\
\times &
\sum_{k=0}^{L_n}
t^{-J_n(L_0-J_n+k)}
\binom{L_0}{J_n-k}_t
\sum_{a' :{\sf m}_n(a')=k}
t^{-\MI(a'_{L_n},\dots,a'_1)}
\prod_{s=1}^{L_n}
L^{a'_s,n}_{I^{s-1},I^s}(t^{s-1} x),
\end{align}
where $a'=(a'_1,\dots,a'_{L_n})$. 
By (\ref{eq:Lord1}) the summation over $a'$ in (\ref{eq:Lreccoef1}) is proportional to $\ML^{J,(0,\dots,0,L_n)}_{I,K}(x)$, which can be computed using basic summation identities giving the following result:
\begin{align}\label{eq:block}
\ML^{J,(0,\dots,0,L_n)}_{I,K}(x,L_n)
&=
 t^{L_n(L_n-1)/2}\frac{(t;t)_{I_n}}{(t;t)_{K_n}}x^{L_n}=
 t^{L_n(L_n-1)/2}
(t;t)_{L_n-J_n}
\binom{I_n}{L_n-J_n}
x^{L_n}.
\end{align}
Using (\ref{eq:block}) we find that the summation over $a'$ in (\ref{eq:Lreccoef1}) can be simplified
\begin{align*}
\sum_{a' :{\sf m}_n(a')=k}
t^{-\MI(a'_{L_n},\dots,a'_1)}
\prod_{s=1}^{L_n}
L^{a'_s,n}_{I^{s-1},I^s}(t^{s-1} x)=
t^{L_n(L_n-1)/2-k(L_n-k)}
(t;t)_{L_n-k}
\binom{L_n}{k}_t
\binom{I_n}{L_n-k}_t
x^{L_n}
.
\end{align*}
Substituting this into (\ref{eq:Lreccoef1}), taking into account (\ref{eq:invnorm}) and simplifying we get
\begin{align}\label{eq:Lreccoef2}
c^{J,L}_{I,K}(x,N)
=&
t^{L_n(L_{n}-1)/2+I_{n}\sum_{i=1}^{n-1}{L_i}+L_n J_n}
\frac{(t;t)_{J_n}(t;t)_{L_0-(J_n-L_n)}}{(t;t)_{L_0}(t;t)_{L_n}}
x^{L_n}
			\nonumber \\
\times &
\sum_{k=J_n-L_0}^{J_n}
t^{-k(J_n+L_n-k)}
(t;t)_{L_n-k}
\binom{L_0}{J_n-k}_t
\binom{L_n}{k}_t
\binom{I_n}{L_n-k}_t.
\end{align}
We can rewrite the first two binomials in the sum in terms of the Pochhammer symbols, then we can cancel a number of factors and write the results as
\begin{align}\label{eq:Lreccoef3}
c^{J,L}_{I,K}(x,N)
=&
t^{L_n(L_{n}-1)/2+I_{n}\sum_{i=1}^{n-1}{L_i}+L_n J_n}
x^{L_n}
			\nonumber \\
\times &
\sum_{k=J_n-L_0}^{J_n}
t^{-k(J_n+L_n-k)}
\frac{(t;t)_{L_0-(J_n-L_n)}}{(t;t)_{L_0-(J_n-k)}}
\binom{J_n}{k}_t
\binom{I_n}{L_n-k}_t.
\end{align}
The ratio of the two Pochhammer symbols can be written as a single Pochhammer symbol which can be then expanded into a sum of binomial coefficients
\begin{align*}
\frac{(t;t)_{L_0-(J_n-L_n)}}{(t;t)_{L_0-(J_n-k)}}=
\prod_{l=0}^{L_n-k-1}(1-t^{L_0-(J_n-k)+1+l})=
\sum_{l=0}^{L_n-k}(-1)^l t^{l(l+1)/2+l(L_0-(J_n-k))}
\binom{L_n-k}{l}_t. 
\end{align*}
We insert this into (\ref{eq:Lreccoef3}) and recombine the binomial coefficients. The sum over $k$ can be computed using (\ref{eq:binom1}) and a slightly modified version of (\ref{eq:binom2}) 
\begin{align*}
&
\sum_{k=J_n-L_0}^{J_n}
t^{-k(J_n+L_n-k)}
\sum_{l=0}^{L_n-k}(-1)^l t^{l(l+1)/2+l(L_0-(J_n-k))}
\binom{J_n}{k}_t
\binom{L_n-k}{l}_t
\binom{I_n}{L_n-k}_t\\
=&
\sum_{l=0}^{L_n}(-1)^l t^{l(l+1)/2+l(L_0-J_n)}
\binom{I_n}{l}_t
\sum_{k=J_n-L_0}^{J_n}
t^{k(k+l-J_n-L_n)}
\binom{J_n}{k}_t
\binom{I_n-l}{L_n-l-k}_t
\\
=&
\sum_{l=0}^{L_n}
(-1)^l 
t^{l(l+1)/2+l L_0-J_n L_n}
\binom{I_n}{l}_t
\binom{I_n+J_n-l}{L_n-l}_t.
\end{align*}
Insert this back into (\ref{eq:Lreccoef3}) and change the summation order with $l\rightarrow L_n-l$
\begin{align}\label{eq:Lreccoef4}
c^{J,L}_{I,K}(x,N)
=&
t^{L_n(L_{n}-1)/2+I_{n}\sum_{i=1}^{n-1}{L_i}}
\sum_{l=0}^{L_n}
(-1)^l 
t^{l(l+1)/2+l L_0}
\binom{I_n}{l}_t
\binom{I_n+J_n-l}{L_n-l}_t
x^{L_n}
\nonumber \\
=&
\sum_{l=0}^{L_n}
(-1)^{L_n-l}
t^{(L_n-l)N}
t^{l(l-1)/2+(|L|-L_n)(K_n+l-J_n)}
\binom{K_n+l}{l}_t
\binom{I_n}{L_n-l}_t
x^{L_n},
\end{align}
where we also used the conservation property $I_n-L_n=K_n-J_n$.  Comparing this with (\ref{eq:biz}) we find
\begin{align*}
c^{J,L}_{I,K}(x,N)=b^{J_n,L_n}_{I_n,K_n}(x,-t^N x;|L|).
\end{align*}
\end{proof}


\section{Positivity of the basic hypergeometric series ${}_{N}\phi_{N-1}$}\label{sec:phi}
In this section we focus on the basic hypergeometric series ${}_{N} \phi _{N-1}$ which is given by the infinite series of products of ratios of $t$-deformed Pochhammer symbols
\begin{align}\label{eq:hgfintro}
\pFq[5]{N}{N-1}{t^{\m_1+1}\text{,}\dots\text{,}t^{\m_{N-1}+1}\text{,}t^{\m_{N}+1}}{t^{\l_1+1}\text{,}\dots\text{,}t^{\l_{N-1}+1}}{z}=
\sum_{s=0}^{\infty}
z^s
\prod_{j=1}^{N-1}\frac{(t^{\m_j};t)_s}{(t^{\l_j};t)_s}
\cdot
\frac{(t^{\m_{N}};t)_s}{(t;t)_s}.
\end{align}
We consider the situation when the parameters $\m_1,\dots,\m_{N}$ and $\l_1,\dots,\l_{N-1}$ are non-negative integers satisfying
\begin{align}
\label{eq:pos1}
\m_i &\geq \l_i, \qquad \text{for}~i=1,\dots,N-1,	\\
\label{eq:pos2}
\m_{i+1} &\geq \l_{i},	\qquad \text{for}~i=1,\dots,N-1,
\end{align}
 which we call the {\it positive regime}. It is clear that if one of these two properties is satisfied, and with $|z|,|t|<1$, then the infinite series in $s$ can be summed using the geometric series, which gives a factorized denominator and a numerator which is a polynomial in $z$ and $t$. However, when both conditions are imposed we find a positivity property which states that the numerator is a polynomial in $z$ and $t$ with non-negative integer coefficients. We prove this result by computing the infinite summation using properties of the $t$-binomial coefficients. Our computation provides a manifestly positive summation formula for the numerator of the function ${}_{N} \phi _{N-1}$ in the positive regime. 

In Section \ref{ssec:notation} we introduce basic notations and identities of Pochhammer symbols and binomial coefficients. In Section \ref{ssec:positivity} we state the main positivity theorem and then provide the details of the derivation of our formula in Sections \ref{ssec:proofpos1} and \ref{spec:proofFi}.

\subsection{Basic identities}\label{ssec:notation}
It is well known that the Gauss binomial coefficient (\ref{gauss}) is a polynomial in $t$ with coefficients in $\mathbb{N}$. 
For non-negative integers $a,b$ and $c$, we will often use the following simple identities involving Gauss binomial coefficients:
\begin{align}
\label{eq:binom0}
&\binom{a}{b}_t  =\binom{a}{a-b}_t  , \\
\label{eq:binom1}
&\binom{a}{b}_t \binom{b}{c}_t =\binom{a}{c}_t \binom{a-c}{b-c}_t, \\
\label{eq:binom2}
&\binom{a+b}{c}_t = \sum_{j=0}^{b}t^{(b-j)(c-j)}\binom{a}{c-j}_t\binom{b}{j}_t,
\end{align}
as well as the identity
\begin{align}
\label{eq:wt}
&(w t^{a};t)_{b}=\frac{(w;t)_{a+b}}{(w;t)_{a}},
\end{align}
for Pochhammer symbols. We will also use the binomial formula 
\begin{align}
\label{eq:infbin}
\sum_{s=0}^{\infty}w^{s} 
\binom{s+a}{s}_t
=\frac{1}{(w;t)_{a+1}}.
\end{align}
Due to (\ref{eq:binom0}), one has
\begin{align*}
 \sum_{j=0}^{a} c_{j} \binom{a}{j}_t = \sum_{j=0}^{a} c_{a-j} \binom{a}{j}_t.
\end{align*}
This will be referred to as {\it reversing the sum} over $j$ ($c_j$ being arbitrary coefficients).


\subsection{Positivity of the basic hypergeometric series}\label{ssec:positivity}
Let us recall again the basic hypergeometric series ${}_{N} \phi _{N-1}$
\begin{align}
\label{eq:HGF}
\pFq[5]{N}{N-1}{a_1\text{,}\dots\text{,}a_{N-1}\text{,}a_{N}}{b_1\text{,}\dots\text{,}b_{N-1}}{z}=
\sum_{s=0}^{\infty}
z^s
\prod_{j=1}^{N-1}\frac{(a_j;t)_s}{(b_j;t)_s}
\cdot
\frac{(a_{N};t)_s}{(t;t)_s}.
\end{align}
Clearly, this is a symmetric function in the parameters $a_j$ and $b_j$ separately. Let $\n=(\n_1,\dots,\n_N)$ and $\nt=(\nt_1,\dots,\nt_N)$ be s.t. $\n_{i+1}\geq \n_i$ and $\nt_{i+1}\geq \nt_i$, in which $\nt_N=\n_N$. By agreement we also define $\nt_i=\n_i=0$ for $i\leq 0$. Then we set $a_i=t^{\nt_i-\n_{i-1}+1}$ and $b_i=t^{\nt_i-\n_i+1}$ and use the short notation
\begin{align}
\label{eq:hgfspecial}
\phi_{\n|\nt}(z;t)=\pFq[5]{N}{N-1}{t^{\nt_1+1}\text{,}\dots\text{,}t^{\nt_{N-1}-\n_{N-2}+1}\text{,}t^{\nt_{N}-\n_{N-1}+1}}{t^{\nt_{1}-\n_{1}+1}\text{,}\dots\text{,}t^{\nt_{N-1}-\n_{N-1}+1}}{z}.
\end{align}
In this definition $\n_N$ does not appear, so we are free to set it as $\n_N:=\nt_N$ which will be important below. 
Notice one feature of this choice of parametrization. If we denote $\m_i=\nt_i-\n_{i-1}$ and $\l_i=\nt_i-\n_i$ then we find the form of the hypergeometric function as stated in (\ref{eq:hgfintro}) and the positivity  conditions (\ref{eq:pos1}) and (\ref{eq:pos2}) reflect the fact that $\n$ and $\nt$ are chosen to be partitions. 
Let us also recall the infinite binomial series
\begin{align}
\label{eq:FF}
&\Phi_{\n|\nt}(z;t):=
(z;t)_{\nt_N+1}
\sum_{s=0}^{\infty}
z^s
\prod_{j=0}^{N-1}
\binom{\nt_{j+1}-\n_{j}+ s}{\nt_{j}-\n_{j}+ s}_t.
\end{align}
Note that the first binomial factor corresponding to $j=0$ contains $\n_0=0$ and $\nt_0=0$, hence it will need to be treated separately sometimes.
Let us establish an important property of $\Phi_{\n|\nt}(z;t)$. Define the rotation operation $r$ which acts on integer vectors $\l=(\l_1,\dots,\l_N)$ as follows:
\begin{align}
\label{eq:rot}
r(\l)=
(\l_2-\l_1,\l_3-\l_1,\dots,\l_{N}-\l_1,\l_N).
\end{align}
In terms of components we have: $r(\l)_j=\l_{j+1}-\l_1$, $j=1,\dots,N-1$ and $r(\l)_N=\l_{N}$.
Let us denote by $r^k$ a $k$-fold rotation which is defined recursively by $r^{k}(\l)=r(r^{k-1}(\l))$. It is clear that due to cancellations at each step of the action of $r$ the action of $r^k$ will be
\begin{align}
\label{eq:rotk}
r^k(\l)=
(\l_{k+1}-\l_k,\l_{k+2}-\l_k,\dots,\l_{N}-\l_k,\l_{N}+\l_{1}-\l_k,\dots,\l_{N}+\l_{k-1}-\l_{k},\l_N),
\end{align}
where $k$ must be considered modulo $N$ and $\l_0=0$.
It is important that the last component $\l_N$ remains untouched by this action. 

\begin{prop}\label{prop:rot}
$\Phi_{\n|\nt}(z;t)$ satisfies the following property:
\begin{align}
\label{eq:Phirot}
\Phi_{\n|\nt}(z;t)=
z^{\n_k-\nt_k}\Phi_{r^k(\n)|r^k(\nt)}(z;t), 
\qquad
\text{for all}~k\geq 1.
\end{align}
\end{prop}
\begin{proof}
It suffices to show (\ref{eq:Phirot}) for $k=1$, since the generic $k$ case follows by induction. By the definition (\ref{eq:FF}), we have
\begin{align*}
&\Phi_{r(\n)|r(\nt)}(z;t)=
(z;t)_{r(\nt)_N+1}
\sum_{s=0}^{\infty}
z^s
\binom{r(\nt)_{1}+ s}{ s}_t
\prod_{j=1}^{N-1}
\binom{r(\nt)_{j+1}-r(\n)_{j}+ s}{r(\nt)_{j}-r(\n)_{j}+ s}_t
=\\
&(z;t)_{\nt_N+1}
\sum_{s=0}^{\infty}
z^s
\binom{\nt_{2}-\nt_{1}+ s}{s}_t
\prod_{j=1}^{N-2}
\binom{\nt_{j+2}-\nt_{1}-\n_{j+1}+\n_{1}+ s}{\nt_{j+1}-\nt_{1}-\n_{j+1}+\n_{1}+ s}_t
\cdot
\binom{\nt_{N}-\n_{N}+\n_1+ s}{\nt_{N}-\nt_1-\n_{N}+\n_{1}+ s}_t.
\end{align*}
Now we use that $\nt_N=\n_N$ and shift the summation index $s=s'+\nt_1-\n_1$
\begin{align*}
&\Phi_{r(\n)|r(\nt)}(z;t)=
(z;t)_{\nt_N+1}
\sum_{s'=\n_1-\nt_1}^{\infty}
z^{s'+\nt_1-\n_1}
\binom{\nt_{2}-\n_1+ s'}{\nt_1-\n_1+s'}_t
\prod_{j=1}^{N-2}
\binom{\nt_{j+2}-\n_{j+1}+ s'}{\nt_{j+1}-\n_{j+1}+s'}_t
\cdot
\binom{\nt_1+ s'}{s'}_t.
\end{align*}
If $\n_1-\nt_1<0$ then we can replace the bottom summation terminal of $s'$ by $0$, due to the last binomial. If on the other hand $\nt_1-\n_1>0$, then we can still replace the bottom terminal of $s'$ by $0$, due to the first binomial factor. Rearranging the binomial factors we find that the expression is equal to $z^{\nt_1-\n_1}\Phi_{\n|\nt}(z;t)$. 
\end{proof}
\begin{cor}\label{cor:rot}
Let $\n$ and $\nt$ be two partitions of length $N$ with $\n_N=\nt_N$ and for some $k$ $\text{min}(\nt-\n)=\nt_k-\n_k<0$. Set $\t=r^k(\n)$ and $\tt=r^k(\nt)$ then $\text{min}_i(\tt_i-\t_i)=0$ and 
\begin{align}
\label{eq:PhiPhi}
\Phi_{\n|\nt}(z;t)=
z^{\n_k-\nt_k}\Phi_{\t|\tt}(z;t).
\end{align}
\end{cor}
\begin{proof}
The proof follows from (\ref{eq:rotk}) and (\ref{eq:Phirot}).
\end{proof}
By Corollary \ref{cor:rot}, we only need to consider $\Phi_{\n|\nt}(z;t)$ for which the partition $\n$ is contained in $\nt$; in other words, $\text{min}(\nt-\n)= 0$. 
The following theorem gives a finite expression for the function $\Phi_{\n|\nt}(z;t)$ (\ref{eq:FF}). This expression is given by an explicit polynomial in $t$ and $z$ with non-negative integer coefficients.  
\begin{thm}\label{th:Fpos}
The binomial series $\Phi_{\n|\nt}(z;t)$, given in (\ref{eq:FF}), can be expressed as
\begin{align}\label{eq:Fpos}
\Phi_{\n|\nt}(z;t)=
\sum_{p}
 t^{\eta(p,\nt)}
 \prod_{k=1}^{N-1}  
 z^{\s_{k}-p_{k}^{k}}  
 \binom{ \nt_{k+1}-p_{k+1}^{1,k}}
 {\nt_{k}-p_k^{1,k}}_t
\cdot
  \prod_{1\leq i\leq k\leq N}
\binom{p_{k+1}^{i}}{p_k^{i}}_t,
\end{align}
where $p_i^{j,k}:=\sum_{a=j}^k p_{i}^a$, the summation runs over the set $\{p_i^k\}_{1\leq k\leq i\leq N-1}$ satisfying 
\begin{align}
\label{eq:pset}
0\leq p_{k}^k\leq p_{k+1}^k\leq  \dots \leq  p_{N-1}^k\leq   p_{N}^k = \s_k,
\end{align} 
the exponent of $t$ is given by
\begin{align}
\label{eq:eta}
\eta(p,\nt)=\sum_{1\leq i\leq k \leq N-1}(p_{k+1}^i-p_k^i)(\nt_k-p_k^{i,k}),
\end{align} 
and the components of $\s$ are defined by $\s_i=\n_i-\n_{i-1}$, for $i=1,\dots,N-1$.
\end{thm}
The proof of this theorem is given in Subsection \ref{spec:proofFi}. 
\begin{cor}\label{cor:pos}
Let $\n$ and $\nt$ be two partitions with $\nt_i\geq \n_i$ for all $i$. The basic hypergeometric series $\phi_{\n|\nt}(z;t)$ can be expressed as
\begin{align}
\label{eq:hgfpos}
&\phi_{\n|\nt}(z;t)
=
\frac{1}{(z;t)_{\nt_N+1}}
\prod_{j=1}^{N-1}\binom{\nt_j-\n_{j-1}}{\nt_j-\n_j}_t^{-1}
\sum_{p}
 t^{\eta(p,\nt)}
 \prod_{k=1}^{N-1}  
 z^{\s_{k}-p_{k}^{k}}  
 \binom{ \nt_{k+1}-p_{k+1}^{1,k}}
 {\nt_{k}-p_k^{1,k}}_t
\cdot
  \prod_{1\leq i\leq k\leq N}
\binom{p_{k+1}^{i}}{p_k^{i}}_t,
\end{align}
where $\eta(p,\nt)$ is defined in (\ref{eq:eta}) and the summation runs over the same set as in Theorem \ref{th:Fpos}.
\end{cor}
\begin{proof}
This can be shown via the binomial series $\Phi_{\n|\nt}$ and Theorem \ref{th:Fpos}. Set $w=t$ in (\ref{eq:wt}) and use it to rewrite (\ref{eq:hgfspecial}) as 
\begin{align}
\label{eq:qHGF-F}
\phi_{\n|\nt}(z;t)
&=
\sum_{s=0}^{\infty}
z^s
\frac{(t^{\nt_1+1};t)_s}{(t;t)_s}
\prod_{j=1}^{N-1}\frac{(t^{\nt_{j+1}-\n_{j}+1};t)_s}{(t^{\nt_j-\n_{j}+1};t)_s}	 \nonumber \\
&=
\sum_{s=0}^{\infty}
z^s
\frac{(t;t)_{\nt_1+s}}{(t;t)_s (t;t)_{\nt_{1}}}	
\prod_{j=1}^{N-1}
\frac{(t;t)_{\nt_{j+1}-\n_j+s}}{(t;t)_{\nt_{j+1}-\n_j}}
\frac{(t;t)_{\nt_j-\n_j}}{(t;t)_{\nt_j-\n_j+s}}	 \nonumber \\
&=
\sum_{s=0}^{\infty}
z^s
\frac{(t;t)_{\nt_1+s}}{(t;t)_s (t;t)_{\nt_{1}}}	
\prod_{j=1}^{N-1}
\frac{(t;t)_{\nt_{j+1}-\n_j+s}}{(t;t)_{\nt_j-\n_j+s}(t;t)_{\nt_{j+1}-\nt_j}}
\frac{(t;t)_{\nt_{j+1}-\nt_j}(t;t)_{\nt_j-\n_j}}{(t;t)_{\nt_{j+1}-\n_j}}	 \nonumber \\
&=
\prod_{j=1}^{N-1}\binom{\nt_{j+1}-\n_j}{\nt_j-\n_j}_t^{-1}
\cdot
\sum_{s=0}^{\infty}
z^s
 \binom{\nt_1+s}{s}_t
\prod_{j=1}^{N-1}
\binom{\nt_{j+1}-\n_j+s}{\nt_j-\n_j+s}_t .
\end{align}
From (\ref{eq:qHGF-F}) the  basic hypergeometric series $\phi_{\n|\nt}(z;t)$ is equal, up to normalization, to $\Phi_{\n|\nt}(z;t)$:
\begin{align}
\label{eq:HFG_F}
\phi_{\n|\nt}(z;t)
=
\frac{1}{(z;t)_{\nt_N+1}}
\prod_{j=1}^{N-1}\binom{\nt_{j+1}-\n_j}{\nt_j-\n_j}_t^{-1} 
\Phi_{\n|\nt}(z;t).
\end{align}
Therefore the statement of the ?orollary relies on the fact that  $\Phi_{\n|\nt}(z;t)$ can be written explicitly as a polynomial in $z$ and $t$ with non-negative integer coefficients, as given in (\ref{eq:Fpos}).
\end{proof}


\subsection{The binomial series as a finite sum}\label{ssec:proofpos1}
Let  us define the unnormalized infinite binomial series
\begin{align}
\label{eq:FFp1}
\Phit_{\n|\nt}(z;t)=\frac{1}{(z;t)_{\nt_N+1}}\Phi_{\n|\nt}(z;t)=
\sum_{s=0}^{\infty}
z^s
\prod_{j=0}^{N-1}
\binom{\nt_{j+1}-\n_{j}+ s}{\nt_{j}-\n_{j}+ s}_t,
\end{align}
The infinite sum over $s$ can be computed resulting in a finite expression for $\Phit_{\n|\nt}(z;t)$. Let $\l=(\l_1,\dots,\l_{N-1})$ and $\s=(\s_1,\dots,\s_{N-1})$ be two compositions of integers. We write $\l\subseteq \s$ if $0\leq \l_i\leq \s_i$ for every $i$. Introduce also a partial sum notation $\l_{i,j}=\sum_{a=i}^j  \l_a$. 
\begin{thm}\label{th:Ft}
$\Phi_{\n|\nt}(z;t)$  can be written in a finite form as
\begin{align}
\label{eq:FFpf}
\Phi_{\n|\nt}(z;t)=
\sum_{\l\subseteq \s}
\prod_{j=1}^{N-1}
(z  t^{\n_{j-1}-\l_{1,j-1}})^{\l_{j}}
(zt^{\n_{j-1}-\l_{1,j-1}};t)_{\s_j-\l_{j}}
\binom{\s_j}{\l_j}_t
\binom{\nt_{j+1}-\n_j+\l_{1,j}}{\nt_{j}-\n_j+\l_{1,j}}_t.
\end{align}
\end{thm}
\begin{proof}
We use binomial identities in order to rewrite (\ref{eq:FFp1}) such that it has only one binomial depending on the summation index $s$. After this we use (\ref{eq:infbin}). Let us show this in detail. In (\ref{eq:FFp1}) we can remove the $s$-dependence from all binomials but one. In order to do this we use the following formula
\begin{align*}
\binom{a}{b}_t \binom{b+m}{c}_t = \sum_{l=0}^m t^{(m-l)(c-l)}\binom{m}{l}_t
\binom{a-c+l}{b-c+l}_t \binom{a}{c-l}_t.
\end{align*}
Consider the first two binomials in (\ref{eq:FFp1}), they can be written in the above form as follows:
\begin{align*}
&\binom{\nt_{2}-\n_{1}+ s}{\nt_{1}-\n_{1}+ s}_t
\binom{\nt_{1}-\n_1+ s+\n_1}{ s}_t
= \\
& \sum_{l=0}^{\n_1} t^{(\n_{1}-l)(s-l)}
\binom{\n_{1}}{l}_t
\binom{\nt_{2}-\n_{1}+l}{\nt_{2}-\nt_{1}}_t 
\binom{\nt_{2}-\n_{1}+ s}{s-l}_t.
\end{align*}
Using this in $\Phit_{\n|\nt}(z;t)$ we write 
\begin{align*}
\Phit_{\n|\nt}(z;t)
&=
\sum_{s=0}^{\infty}
z^s
\binom{\nt_{2}-\n_{1}+ s}{\nt_{1}-\n_{1}+ s}_t
\binom{\nt_{1}-\n_1+ s+\n_1}{ s}_t
\prod_{j=2}^{N-1}
\binom{\nt_{j+1}-\n_{j}+ s}{\nt_{j}-\n_{j}+ s}_t
\\
&=
 \sum_{l=0}^{\n_1} t^{-l(\n_{1}-l)}
\binom{\n_{1}}{l}_t
\binom{\nt_{2}-\n_{1}+l}{\nt_{1}-\n_{1}+l}_t 
\cdot
\sum_{s=0}^{\infty}
z^s
t^{s(\n_{1}-l)}
\binom{\nt_{2}-\n_{1}+ s}{s-l}_t
\prod_{j=2}^{N-1}
\binom{\nt_{j+1}-\n_{j}+ s}{\nt_{j}-\n_{j}+ s}_t.
\end{align*}
In the summation over $s$ above all terms with $s<l$ vanish. Therefore we can shift $s$ by the value of $l$, this will also conveniently cancel the power of $t$ next to the summation over $l$: 
\begin{align*}
\Phit_{\n|\nt}(z;t)
=
\sum_{l=0}^{\n_1} 
z^l
\binom{\n_{1}}{l}_t
\binom{\nt_{2}-\n_{1}+l}{\nt_{1}-\n_{1}+l}_t 
\cdot
\sum_{s=0}^{\infty}
z^s
t^{s(\n_{1}-l)}
\binom{\nt_{2}-\n_{1}+l+s}{s}_t
\prod_{j=2}^{N-1}
\binom{\nt_{j+1}-\n_{j}+l+ s}{\nt_{j}-\n_{j}+l+ s}_t.
\end{align*}
Let us define two new partitions $\n^{1}=(\n^{1}_1,\dots,\n^1_{N-1})$ and $\nt^{1}=(\nt^{1}_1,\dots,\nt^1_{N-1})$, by $\n^{1}_j=\n_{j+1}-\n_1$ and $\nt^{1}_j=\nt_{j+1}-\n_1+l$. After this we notice that the sum over $s$ is equal to $\Phit_{\n^1|\nt^1}(z t^{\n_1-l};t)$, and the above equation becomes a recurrence relation 
\begin{align}\label{eq:Frec}
\Phit_{\n|\nt}(z;t)
=
\sum_{l=0}^{\n_1} 
z^l
\binom{\n_{1}}{l}_t
\binom{\nt_{2}-\n_{1}+l}{\nt_{1}-\n_{1}+l}_t 
\cdot
\Phit_{\n^1|\nt^1}(z t^{\n_1-l};t).
\end{align}
This recurrence relation can be iterated. Using compositions $\s$ and $\l$ and the partial sum notation $\l_{i,j}$ we can write the result of this iteration as
\begin{align*}
\Phit_{\n|\nt}(z;t)
=&
\sum_{\l\subseteq \s}
\prod_{j=1}^{N-1}
(z  t^{\n_{j-1}-\l_{1,j-1}})^{\l_{j}}
\binom{\s_j}{\l_j}_t
\binom{\nt_{j+1}-\n_j+\l_{1,j}}{\nt_{j}-\n_j+\l_{1,j}}_t \\
&
\sum_{s=0}^{\infty}z^{s}
t^{s(\n_{N-1}-\l_{1,N-1})}
\binom{\nt_N-\n_{N-1}+\l_{1,N-1}+s}{s}_t.
\end{align*}
With the binomial formula (\ref{eq:infbin}) we compute the infinite summation over $s$ and then simplify it using (\ref{eq:wt}):
\begin{align*}
\sum_{s=0}^{\infty}z^{s}
t^{s(\n_{N-1}-\l_{1,N-1})}
\binom{\nt_N-\n_{N-1}+\l_{1,N-1}+s}{s}_t
&=
\frac{1}{(z t^{\n_{N-1}-\l_{1,N-1}};t)_{\nt_N-\n_{N-1}+\l_{1,N-1}+1}} \\
&= 
\frac{(z;t)_{\n_{N-1}-\l_{1,N-1}}}{(z;t)_{\nt_N+1}}.
\end{align*}
Now the infinite binomial series $\Phit_{\n|\nt}(z;t)$ takes the form 
\begin{align}
\label{eq:FFp2}
\Phit_{\n|\nt}(z;t)=
&
\frac{1}{(z;t)_{\nt_N+1}} 
\sum_{\l\subseteq \s}
(z;t)_{\n_{N-1}-\l_{1,N-1}}
\prod_{j=1}^{N-1}
(z  t^{\n_{j-1}-\l_{1,j-1}})^{\l_{j}}
\binom{\s_j}{\l_j}_t
\binom{\nt_{j+1}-\n_j+\l_{1,j}}{\nt_{j}-\n_j+\l_{1,j}}_t.
\end{align}
From this we get the expression for $\Phi_{\n|\nt}(z;t)$ using the relation (\ref{eq:FFp1})
\begin{align*}
\Phi_{\n|\nt}(z;t)=
\sum_{\l\subseteq \s}
(z;t)_{\n_{N-1}-\l_{1,N-1}}
\prod_{j=1}^{N-1}
(z  t^{\n_{j-1}-\l_{1,j-1}})^{\l_{j}}
\binom{\s_j}{\l_j}_t
\binom{\nt_{j+1}-\n_j+\l_{1,j}}{\nt_{j}-\n_j+\l_{1,j}}_t.
\end{align*}
The $t$-Pochhammer  symbol can be factorized by a repetitive use of (\ref{eq:wt}), thus we find
\begin{align}\label{eq:FFp3}
\Phi_{\n|\nt}(z;t)=
\sum_{\l\subseteq \s}
\prod_{j=1}^{N-1}
(z  t^{\n_{j-1}-\l_{1,j-1}})^{\l_{j}}
(zt^{\n_{j-1}-\l_{1,j-1}};t)_{\s_j-\l_{j}}
\binom{\s_j}{\l_j}_t
\binom{\nt_{j+1}-\n_j+\l_{1,j}}{\nt_{j}-\n_j+\l_{1,j}}_t,
\end{align}
which proves the theorem.
\end{proof}

\begin{rmk}\label{rmk:Phi1}
At $z=1$ the binomial series $\Phi$ simplifies:
\begin{align}\label{eq:Phiz1}
\Phi_{\n|\nt}(1;t)=
\prod_{j=1}^{N-1}
\binom{\nt_{j+1}}{\nt_{j}}_t.
\end{align}
\end{rmk}
\begin{proof}
This formula follows from (\ref{eq:FFp3}) where one needs to set $z=1$  and observe that the Pochhammer symbols are non-zero only for $\l_j=\s_j=\n_j-\n_{j-1}$. This collapses the summation. 
\end{proof}


\subsection{The $q$-binomial series as a positive polynomial}\label{spec:proofFi}
Our positive formula relies on the following lemma.
\begin{lem}\label{lem:g}
Let $m$ be a non-negative integer, $v$ a complex parameter and $a=(a_1,\dots,a_n)$ and $b=(b_1,\dots,b_n)$ be two sets of non-negative integers, and set $c_i:=a_i+m- b_i$, for $i=1,\dots,n$. The function $g_{m}(a,b;v)$ defined by
\begin{align}
\label{eq:defg}
&g_{m}(a,b;v)
:=\sum_{k=0}^m v^k (v;t)_{m-k}\binom{m}{k}_t \binom{k+a_1}{b_1}_t\dots\binom{k+a_n}{b_n}_t,
\end{align}
can be written in the following form
\begin{align}
\label{eq:posg}
&g_{m}(a,b;v)
=\sum_{p}  \prod_{i=1}^n v^{p_{i-1}-p_{i}}  t^{(p_{i-1}-p_{i}) (c_i-p_i)}  
\binom{p_{i-1}}{p_i}_t \binom{a_i+m-p_{i-1}}{c_i-p_{i}}_t,
\end{align}
where $p=(p_1,\dots,p_n)$,  such that $m=p_0\geq p_1\geq \dots \geq p_{n} $.
Therefore $g_{m}(a,b;v)$ is a polynomial in $v$ and $t$ with non-negative integer coefficients.
\end{lem}
Let us point out first that in (\ref{eq:posg}) the summation set over $p$ is more restricted. The condition on $p$, which means that it is a set of partitions inside the $n\times m$ box, is imposed by the binomials $\binom{p_{i-1}}{p_i}_t$ in (\ref{eq:posg}). The binomials $\binom{a_i+m-p_{i-1}}{c_i-p_i}_t$ put two more restrictions:
\begin{align*}
&p_i\leq c_i,\\
&p_{i-1}-p_i\leq b_i.
\end{align*}
The intersection of all these conditions gives the minimal summation set in  (\ref{eq:posg}). 
\begin{proof}
First we recall the following well known formula:
\begin{align*}
&
\sum_{k=0}^m v^k (v;t)_{m-k}\binom{m}{k}_t=1.
\end{align*}
Therefore when $n=0$
\begin{align}
\label{eq:gn0}
&g_{m}(\{\},\{\};v)
=\sum_{k=0}^m v^k (v;t)_{m-k}\binom{m}{k}_t=1.
\end{align}
Starting with (\ref{eq:defg})
\begin{align*}
\sum_{k=0}^m v^k (v;t)_{m-k}\binom{m}{k}_t \prod_{i=1}^n \binom{k+a_i}{b_i}_t=
\sum_{k=0}^m v^k (v;t)_{m-k}\binom{m}{k}_t \binom{k+a_1}{b_1}_t
\prod_{i=2}^n \binom{k+a_i}{b_i}_t,
\end{align*}
we perform the following steps. Use (\ref{eq:binom2}) to separate $a_1$ and $k$ in the corresponding binomial coefficient
\begin{align}\label{eq:lemmaproof}
\sum_{s_1=0}^{a_1} t^{(a_1-s_1)(b_1-s_1)} \binom{a_1}{s_1}_t
\sum_{k=0}^m v^k (v;t)_{m-k}\binom{m}{k}_t \binom{k}{b_1-s_1}_t 
\prod_{i=2}^n \binom{k+a_i}{b_i}_t,
\end{align}
Then use (\ref{eq:binom1}) and (\ref{eq:binom0}) in the first two binomials in the sum over $k$
\begin{align*}
\sum_{s_1=0}^{a_1} t^{(a_1-s_1)(b_1-s_1)} \binom{a_1}{s_1}_t
\sum_{k=0}^m v^k (v;t)_{m-k}
\binom{m}{b_1-s_1}_t 
\binom{m-b_1+s_1}{m-k}_t
\prod_{i=2}^n \binom{k+a_i}{b_i}_t.
\end{align*}
The value of $s_1$ cannot exceed $b_1$, therefore we can set the limit of the summation over $s_1$ to $b_1$ after which we reverse the order of this sum by switching $s_1\rightarrow b_1-s_1$
\begin{align*}
\sum_{s_1=0}^{b_1} t^{s_1(a_1-b_1+s_1)} 
\binom{a_1}{b_1-s_1}_t 
\binom{m}{s_1}_t 
\sum_{k=0}^m v^k (v;t)_{m-k}
\binom{m-s_1}{m-k}_t
\prod_{i=2}^n \binom{k+a_i}{b_i}_t.
\end{align*}
The non-zero terms of the sum over $k$ start from $k=s_1$ due to the first binomial factor in this sum. Use (\ref{eq:binom0}) in this binomial and replace $k\rightarrow k+s_1$
\begin{align*}
\sum_{s_1=0}^{b_1} t^{s_1(a_1-b_1+s_1)} 
\binom{a_1}{b_1-s_1}_t 
\binom{m}{s_1}_t 
v^{s_1}
\sum_{k=0}^{m-s_1} v^{k} (v;t)_{m-s_1-k}
\binom{m-s_1}{k}_t
\prod_{i=2}^n \binom{k+s_1+a_i}{b_i}_t.
\end{align*}
The summation over $k$ has the same form as the one we started with:
\begin{align*}
\sum_{k=0}^{m'} v^{k} (v;t)_{m'-k}
\binom{m'}{k}_t
\prod_{i=1}^{n'} \binom{k+a_i'}{b_i'}_t,
\end{align*}
where $m'=m-s_1$, $a_i'=s_1+a_{i+1}$, $b_i'=b_{i+1}$, $n'=n-1$ and therefore can be expressed in terms of $g_{m-s_1}$ with fewer arguments. We find that (\ref{eq:defg}) can be written in the following recursive form:
\begin{align*}
&g_{m}(\{a_1,\dots,a_n\},\{b_1,\dots,b_n\};v)=
\sum_{s_1=0}^{\text{min}(m,b_1)} t^{s_1(a_1-b_1+s_1)} 
\binom{a_1}{b_1-s_1}_t 
\binom{m}{s_1}_t 
v^{s_1} \\
&\times 
g_{m-s_1}(\{a_2+s_1,\dots,a_n+s_1\},\{b_2,\dots,b_n\};v).
\end{align*}
Notice that the summation bound of $\text{min}(m,b_1)$ can be replaced by either of the two numbers.  Iterating this recurrence relation until we exhaust all $a$ and $b$ arguments, and recalling (\ref{eq:gn0}) we find:
\begin{align}\label{eq:posgs}
&g_{m}(a,b;v)
=\left(\prod_{i=1}^n \sum_{s_i=0}^{b_i}\right) \prod_{i=1}^n v^{s_i}  t^{s_i (a_i-b_i+s_{1,i})}  
\binom{m-s_{1,i-1}}{s_i}_t \binom{a_i+s_{1,i-1}}{b_i-s_i}_t,
\end{align}
where we used the partial sum indices $s_{i,j}=\sum_{k=i}^js_k$.
Changing the summation variables to $p_{i}=s_{1,i}$ and using (\ref{eq:binom0}) we arrive at
\begin{align}
\label{eq:posgp}
&g_{m}(a,b;v)
=\sum_{p}  \prod_{i=1}^n v^{p_i-p_{i-1}}  t^{(p_i-p_{i-1}) (a_i-b_i+p_i)}  
\binom{m-p_{i-1}}{m-p_i}_t \binom{a_i+p_{i-1}}{b_i-p_i+p_{i-1}}_t,
\end{align}
where we let the sum run over $p_i=0,\dots,m$, $i=1,\dots,n$ and $p_0=0$. Reversing the summation by replacing $p_i$ with $m-p_i$ we find 
\begin{align}
\label{eq:posgp2}
&g_{m}(a,b;v)
=\sum_{p}  \prod_{i=1}^n v^{p_{i-1}-p_i}  t^{(p_{i-1}-p_i) (a_i-b_i+m-p_i)}  
\binom{p_{i-1}}{p_i}_t \binom{a_i+m-p_{i-1}}{b_i-p_{i-1}+p_i}_t.
\end{align}
After such replacement the  summation indices $p$ must satisfy $m=p_0\geq p_1\geq \dots \geq p_{n} $.
Using (\ref{eq:binom0}) in the last binomial and the definition of $c_i$ leads to (\ref{eq:posg}). 
\end{proof}


\subsection{Proof of Theorem \ref{th:Fpos}}
Recall that $\s$ is a composition of length $N-1$ with components $\s_i=\n_i-\n_{i-1}$, for $i=1,\dots,N-1$. Let $\s^{(N-1-j)}$ be the composition obtained from $\s$ with the last $j$ parts removed. Recall the formula (\ref{eq:FFp3}), changing the second binomial according to (\ref{eq:binom0}) and isolating the very last summation we get
\begin{align*}
&\Phi_{\n|\nt}(z;t)=
\sum_{\l\subseteq \s^{(N-2)}}
\prod_{j=1}^{N-2}
(z  t^{\n_{j-1}-\l_{1,j-1}})^{\l_{j}}
(zt^{\n_{j-1}-\l_{1,j-1}};t)_{\s_j-\l_{j}}
\binom{\s_j}{\l_j}_t
\binom{\nt_{j+1}-\n_j+\l_{1,j}}{\nt_{j+1}-\nt_j}_t  \\
&
\sum_{\l_{N-1}=0}^{\s_{N-1}}
(z  t^{\n_{N-2}-\l_{1,N-2}})^{\l_{N-1}}
(zt^{\n_{N-2}-\l_{1,N-2}};t)_{\s_{N-1}-\l_{N-1}}
\binom{\s_{N-1}}{\l_{N-1}}_t
\binom{\nt_{N}-\n_{N-1}+\l_{1,N-1}}{\nt_{N}-\nt_{N-1}}_t.
\end{align*}
For convenience we introduce 
\begin{align*}
\a_1^1=\nt_{N}-\n_{N-1}, \qquad 
\b_1^1= \nt_{N}-\nt_{N-1}.
\end{align*}
According to the definition of the function $g$ in (\ref{eq:defg}) the sum over $\l_{N-1}$ in the above expression is equal to
\begin{align*}
&g_{\s_{N-1}}(\{\a_1^1+\l_{1,N-2}\},\{\b_1^1\};z  t^{\n_{N-2}-\l_{1,N-2}}).
\end{align*}
This function can be rewritten in the positive form by Lemma \ref{lem:g}. For convenience we choose to rewrite $g$ in the form (\ref{eq:posgp2}). The result depends on $\l_{2,N}$ only through a binomial coefficient. Inserting the positive expression for $g$ into the above formula and performing a few simplifications we get:
\begin{align*}
\Phi_{\n|\nt}(z;t)=
&\sum_{\l\subseteq \s^{(N-2)}}
\prod_{j=1}^{N-2}
(z  t^{\n_{j-1}-\l_{1,j-1}})^{\l_{j}}
(zt^{\n_{j-1}-\l_{1,j-1}};t)_{\s_j-\l_{j}}
\binom{\s_j}{\l_j}_t
\binom{\nt_{j+1}-\n_j+\l_{1,j}}{\nt_{j+1}-\nt_j}_t
 \\
&
\sum_{p_1^1=0}^{\s_{N-1}}
(z t^{\n_{N-2}})^{p_1^1-p_0^1}
t^{(p_1^1-p_0^1)(\a_1^1-\b_1^1+\s_{N-1}-p_1^1)}
\binom{p_0^1}{p_1^1}_t
\binom{\a_1^1+\s_{N-1}-p_0^1+\l_{1,N-2}}{\b_1^1-p_0^1+p_1^1}_t.
\end{align*}
The choice of the indexing of the summation index $p$ as well as the use of $\a_1^1$ and $\b_1^1$ are convenient for further derivation. 
In the above summation over $\l\subseteq \s^{(N-2)}$ we can isolate the sum over $\l_{N-2}$ and show that it is of the form (\ref{eq:defg}), i.e. 
\begin{align*}
&g_{\s_{N-2}}(\{\a_1^2+\l_{1,N-3},\a_2^2+\l_{1,N-3}\},
\{\b_1^2,\b_2^2\};z t^{\n_{N-3}-\l_{1,N-3}}),  \\
& \a_1^2 =\a_1^1+\s_{N-1}-p_0^1, \qquad 
\a_2^2=\nt_{N-1}-\n_{N-2} ,\\
&
\b_1^2 =\b_1^1 -p_0^1+p_1^1,\qquad
\b_2^2=\nt_{N-1}-\nt_{N-2}.
\end{align*}
This allows us to rewrite the expression for $\Phi_{\n|\nt}$ using (\ref{eq:posgp2}) with the only dependence on $\l_{1,N-3}$ appearing through a binomial coefficient while the unwanted power of $t$ depending on $\l_{N-3}$ cancels. After this we find
\begin{align*}
\Phi_{\n|\nt}(z;t)=
&\sum_{\l\subseteq \s^{(N-3)}}
\prod_{j=1}^{N-3}
(z  t^{\n_{j-1}-\l_{1,j-1}})^{\l_{j}}
(zt^{\n_{j-1}-\l_{1,j-1}};t)_{\s_j-\l_{j}}
\binom{\s_j}{\l_j}_t
\binom{\nt_{j+1}-\n_j+\l_{1,j}}{\nt_{j+1}-\nt_j}_t
\\
&
\sum_{p_1^1=0}^{\s_{N-1}}
(z t^{\n_{N-2}})^{p_1^1-p_0^1}
t^{(p_1^1-p_0^1)(\a_1^1-\b_1^1+\s_{N-1}-p_1^1)}
\binom{p_0^1}{p_1^1}_t \\
&
\sum_{p_1^2=0}^{\s_{N-2}}
(z t^{\n_{N-3}})^{p_1^2-p_0^2}
t^{(p_1^2-p_0^2)(\a_1^2-\b_1^2+\s_{N-2}-p_1^2)}
\binom{p_0^2}{p_1^2}_t 
\binom{\a_1^2+\s_{N-2}-p_{0}^2+\l_{1,N-3}}{\b_{1}^2-p_0^2+p_1^2}_t\\
&
\sum_{p_2^2=0}^{p_1^2}
(z t^{\n_{N-3}})^{p_2^2-p_1^2}
t^{(p_2^2-p_1^2)(\a_2^2-\b_2^2+\s_{N-2}-p_2^2)}
\binom{p_1^2}{p_2^2}_t 
\binom{\a_2^2+\s_{N-2}-p_{1}^2+\l_{1,N-3}}{\b_{2}^2-p_1^2+p_2^2}_t.
\end{align*}
We can repeat this procedure of isolating the last summation in $\l$ and rewriting it using Lemma \ref{lem:g}. At each step the unwanted powers of $t$ cancel and the dependence on $\l$ with the highest index is precisely of the form of the polynomial $g$ and hence we can rewrite it in the positive form. In doing so the number of arguments $a_i$ and $b_i$ of the polynomial $g$ in (\ref{eq:defg}) at each step increases by one and these arguments can be defined recursively.  The polynomial $g_{\s_{N-k}}$ at the $k$-th step of this recurrence will be expressed using (\ref{eq:posgp2}) as a sum over $p^{k}=(p^{k}_1,\dots,p^{k}_k)$ and have the following arguments
\begin{align*}
&g_{\s_{N-k}}(\{\a_1^{k}+\l_{1,N-k-1},\dots,\a_k^{k}+\l_{1,N-k-1}\},\{\b_1^{k},\dots,\b_k^{k}\}; z t^{\n_{N-k-1}-\l_{1,N-k-1}}),
\end{align*}
where $\a_j^k$ and $\b_j^k$ are defined by
\begin{align}
\a_i^k=\a_{i}^{k-1}+\s_{N-k+1}-p_{i-1}^{k-1},\qquad 
\b_i^k=\b_{i}^{k-1}-p_{i-1}^{k-1}+p_i^{k-1}, 
\qquad \text{for}~i=1,\dots,k-1,		\label{eq:abrec}
\end{align}
and 
\begin{align}
\a_k^{k}=\nt_{N-k+1}-\n_{N-k},\qquad  
\b_k^{k}=\nt_{N-k+1}-\nt_{N-k} . 		\label{eq:abreck} 
\end{align}
These recurrence relations can be easily solved
\begin{align}
&\a_i^{k}=\nt_{N-i+1}-\n_{N-k} -p_{i-1}^{i,k-1}, 		\label{eq:arecsol} \\
&\b_i^k=\nt_{N-i+1}-\nt_{N-i}- p_{i-1}^{i,k-1}+p_{i}^{i,k-1},\label{eq:brecsol}
\end{align}
where we used another partial sum notation $p_{i}^{j,k}=\sum_{a=j}^k p_i^a$.
Each time we replace the polynomial $g_{\n_{N-k}}$ with its positive expansion (\ref{eq:posgp2}) we get
\begin{align}
&g_{\s_{N-k}}(\{\a_1^{k}+\l_{1,N-k-1},\dots,\a_k^{k}+\l_{1,N-k-1}\},\{\b_1^{k},\dots,\b_k^{k}\}; z t^{\n_{N-k-1}-\l_{1,N-k-1}})= \nonumber 
\\
&\sum_{p^{k}} 
 \prod_{i=1}^k
 (z t^{\n_{N-k-1}})^{p_{i-1}^{k}-p_{i}^{k}}  
 t^{(p_{i-1}^{k}-p_{i}^{k}) (\a_i^k-\b_i^k+\s_{N-k}-p_{i}^{k})}  
\binom{p_{i-1}^{k}}{p_i^{k}}_t 
\binom{\a_i^k+\s_{N-k}- p_{i-1}^{k}+\l_{1,N-k-1}}{\b_i^k-p_{i-1}^{k}+p_i^{k}}_t,
\end{align}
and the sum runs over $p^k=(p^k_1,\dots,p^k_k)$. The binomials depending on $\l_{1,N-k-1}$ contribute to the $g$-polynomial at the next level of the recursion while all the rest gets accumulated. After $N-1$ iterations we get
\begin{align}
\Phi_{\n|\nt}(z;t)=
\sum_{p^{N-1}} \dots \sum_{p^{1}} 
 \prod_{k=1}^{N-1}  
 \binom{\a_k^{N}}{\b_k^{N}}_t
  \prod_{i=1}^k
& (z t^{\n_{N-k-1}})^{p_{i-1}^{k}-p_{i}^{k}}  
 t^{(p_{i-1}^{k}-p_{i}^{k}) (\a_i^k-\b_i^k+\s_{N-k}-p_{i}^{k})}  
\binom{p_{i-1}^{k}}{p_i^{k}}_t,
\end{align} 
substituting $\a_i^k$ and $\b_i^k$ from (\ref{eq:arecsol}) and (\ref{eq:brecsol}) and using (\ref{eq:binom0}) we get
\begin{align}
\Phi_{\n|\nt}(z;t)=
\sum_{p}
 t^{\sum_{1\leq i\leq k\leq N-1} (p_{i-1}^{k}-p_{i}^{k}) (\nt_{N-i}-p_{i}^{i,k})}
 \prod_{k=1}^{N-1}  
 z^{p_{0}^{k}-p_{k}^{k}}  
 \binom{ \nt_{N-k+1}-p_{k-1}^{k,N-1}}
 {\nt_{N-k}-p_k^{k,N-1}}_t
\cdot  \prod_{1\leq i\leq k\leq N-1}
\binom{p_{i-1}^{k}}{p_i^{k}}_t.
\end{align} 
The summation runs over the set $\{p_i^k\}_{1\leq i\leq k\leq N-1}$ satisfying
\begin{align*}
0\leq p_{k}^k \leq p_{k-1}^k \leq \dots \leq p_{1}^k \leq p_{0}^k = \s_{N-k}.
\end{align*} 
In the last step we need to introduce a new index $s_{N-i}^{N-j}=p_i^j$. After rearranging the factors in the products we get
\begin{align}
\Phi_{\n|\nt}(z;t)=
\sum_{s}
 t^{\eta(s,\nt)}
 \prod_{k=1}^{N-1}  
 z^{\s_{k}-s_{k}^{k}}  
 \binom{ \nt_{k+1}-s_{k+1}^{1,k}}
 {\nt_{k}-s_k^{1,k}}_t
\cdot  \prod_{1\leq i\leq k\leq N-1}
\binom{s_{k+1}^{i}}{s_k^{i}}_t,
\end{align} 
where $\eta(s,\nt)=\sum_{1\leq i\leq k \leq N-1}(s_{k+1}^i-s_k^i)(\nt_k-s_k^{i,k})$. The summation set is now given by $\{s_k^i\}_{1\leq i\leq k\leq N-1}$ subject to
\begin{align*}
0\leq s_{k}^k\leq s_{k+1}^k\leq  \dots \leq  s_{N-1}^k\leq   s_{N}^k = \s_k.
\end{align*} 
This proves Theorem \ref{th:Fpos}.

\begin{ex} Fix $\n=(1,3,4,5)$ and $\nt=(2,3,5,5)$, the polynomial $\Phi _{\nu|\tilde{\nu }}(z;t)$ reads
\begin{align*}
\Phi _{\nu|\tilde{\nu }}(u,t)=t^8 z^3+2 (1+t) \left(1+t+t^2\right) t^4 z^2+\left(1+t+t^2\right) \left(1+3 t+t^2\right) t z+(1+t).
\end{align*} 
In order to see the contributions from different summands we represent the summation set over $s_k^i$, $1\leq i\leq k\leq 3$, by triangular integer arrays with rows labeled by $i$ and positions within rows by $k-i+1$. Then we have in total eight non-vanishing elements in the summation set $s$:  
\begin{align*}
\Phi _{\nu|\tilde{\nu }}(z;t)
&=
t^8 z^3 \left[
\begin{array}{c}
 (0,1,1) \\
 (0,2) \\
 (1) \\
\end{array}
\right]+(1+t) t^6 z^2 \left[
\begin{array}{c}
 (1,1,1) \\
 (0,2) \\
 (1) \\
\end{array}
\right]+(1+t)^2 t^4 z^2 \left[
\begin{array}{c}
 (0,1,1) \\
 (1,2) \\
 (1) \\
\end{array}
\right]\\
&+
t^2 z \left[
\begin{array}{c}
 (0,1,1) \\
 (2,2) \\
 (1) \\
\end{array}
\right]+
(1+t)^3 t^2 z \left[
\begin{array}{c}
 (1,1,1) \\
 (1,2) \\
 (1) \\
\end{array}
\right]+(1+t) \left(1+t+t^2\right) t z \left[
\begin{array}{c}
 (0,0,1) \\
 (2,2) \\
 (1) \\
\end{array}
\right]\\
&+(1+t) \left(1+t+t^2\right) t^4 z^2 \left[
\begin{array}{c}
 (0,0,1) \\
 (1,2) \\
 (1) \\
\end{array}
\right]+(1+t) \left[
\begin{array}{c}
 (1,1,1) \\
 (2,2) \\
 (1) \\
\end{array}
\right].
\end{align*} 
\end{ex}


\section*{Acknowledgments}

We gratefully acknowledge support from the Australian Research Council Centre of Excellence for Mathematical and Statistical Frontiers (ACEMS), and MW acknowledges support by an Australian Research Council DECRA. We thank the program {\it Non-equilibrium systems and special functions} held at the MATRIX institute in Creswick, where part of this work was completed. 
We would like to thank Alexei Borodin, Jan de Gier, Atsuo Kuniba, Vladimir Mangazeev, Ole Warnaar and Paul Zinn-Justin for useful discussions.


\end{document}